\documentclass{amsart}

\usepackage{graphicx}
\usepackage{hyperref}
\usepackage{amsmath}
\usepackage{amscd}
\usepackage{amssymb}
\usepackage{enumitem}

\usepackage{pgf,tikz}
\usetikzlibrary{arrows}

\newtheorem{theorem}{Theorem}[section]

\newtheorem{corollary}[theorem]{Corollary}
\newtheorem{lemma}[theorem]{Lemma}

\newtheorem{proposition}[theorem]{Proposition}
\theoremstyle{definition}
\newtheorem{definition}[theorem]{Definition}
\theoremstyle{remark}
\newtheorem{remark}[theorem]{Remark}

\numberwithin{equation}{section}

\begin{document}

\title[Lefschetz fibrations via Stallings twist]{Lefschetz fibrations on knot surgery  $4$-manifolds via Stallings twist}
\author{Jongil Park}
\address{Department of Mathematical Sciences, Seoul National University, Seoul 151-747, \& Korea Institute for Advanced Study, Seoul 130-722, Republic of Korea}
\email{jipark@snu.ac.kr}
\author{Ki-Heon Yun}
\address{Department of Mathematics, Sungshin Women's University, 
  Seoul 136-742, Republic of Korea}
\email{kyun@sungshin.ac.kr}%


\date{\today}

\begin{abstract}
 In this article we construct a family of knot surgery $4$-manifolds admitting arbitrarily many nonisomorphic Lefschetz fibration structures with the same genus fiber. We obtain such families by performing knot surgery on an elliptic surface $E(2)$ using connected sums of fibered knots obtained by Stallings twist from a slice knot $3_1 \sharp 3^*_1$. 
 By comparing their monodromy groups induced from the corresponding monodromy factorizations, we show that they admit mutually nonisomorphic Lefschetz fibration structures. 
\end{abstract}

\maketitle


\section{Introduction}\label{section:intro}

 Since it was known that any closed symplectic $4$-manifold admits a     Lefschetz pencil~\cite{Donaldson:99} and a Lefschetz fibration structure 
can be obtained from a Lefschetz pencil by blowing-up the base loci, 
a study on Lefschetz fibrations has become an important research theme to understand  symplectic $4$-manifolds topologically. 
In fact, Lefschetz pencils and Lefschetz fibrations have long been studied
extensively by algebraic geometers and topologists in complex category
and these notions are extended in symplectic category.
 It is also well known that an isomorphism class of Lefschetz fibrations over $S^2$ is completely characterized by {\em monodromy factorization}, an ordered sequence of right handed Dehn twists whose product becomes identity in the surface mapping class group corresponding to the generic fiber, up to Hurwitz equivalence and global conjugation equivalence.  
 Note that the Hurwitz equivalence problem of monodromy factorizations is one of very interesting but hard questions in topology. For example, people working on this area want to get an answer for the following question:
 
\begin{quote}
 Is the Hurwitz problem for mapping class group factorizations decidable? 
Are there interestting criteria which can be used to conclude that two given 
factorizations are equivalent, or inequivalent, 
up to Hurwitz moves and global conjugation? ~\cite{Auroux:2006}
\end{quote}

 On the other hand, since the gauge theory, in particular Seiberg-Witten theory, was introduced, topologists and geometers working on $4$-manifolds have developed various techniques and they have obtained many fruitful and remarkable results on $4$-manifolds topology in last 30 years. Among them, a knot surgery technique introduced by R.~Fintushel and R.~Stern turned out to be one of most effective techniques to modify smooth structures without changing the topological type of a given $4$-manifold.
  Note that Fintushel-Stern's knot surgery $4$-manifold $X_K$ is the following~\cite{FS:98}:
 Suppose that $X$ is a simply connected smooth $4$-manifold containing
an embedded torus $T$ of square $0$ and $\pi_1(X\setminus T) = 1$.
Then, for any knot $K \subset S^3$,
one can construct a new $4$-manifold, called {\em a knot surgery $4$-manifold},
 \begin{equation*}
  X_K = X\sharp_{T=T_m} (S^1\times M_K)
 \end{equation*}
by taking a fiber sum along a torus $T$ in $X$ and $T_m = S^1\times m$ in
\mbox{$S^1 \times M_K$}, where $M_K$ is the $3$-manifold obtained by doing
$0$-framed surgery along $K$ and $m$ is the meridian of $K$.
Then Fintushel and Stern proved that, under a mild condition on $X$
and $T$, the knot surgery $4$-manifold $X_K$ is homeomorphic,
but not diffeomorphic, to a given $X$.
 Furthermore, if $X$ is a simply connected elliptic surface $E(n)$,
$T$ is a generic elliptic fiber, and $K$ is a fibered knot in $S^3$,
then it is also known that the knot surgery $4$-manifold $E(n)_K$ admits
not only a symplectic structure but also a genus $2g(K)+n-1$ Lefschetz
fibration structure~\cite{FS:2004}.

 In this article we continue to investigate inequivalent Lefschetz fibration structures on the knot surgery $4$-manifold $E(2)_K$ and we answer the following question proposed by I. Smith~\cite{Smith:98}:

\begin{quotation}
 Does the diffeomorphism type of a smooth $4$-manifold determine the
 equivalence class of a Lefschetz fibration by curves of some given genus?
 \end{quotation}

\noindent
 Regarding this question, I. Smith first showed that $(T^2 \times \Sigma_2)\sharp 9\overline{\mathbb{CP}}^2$ admits two nonisomorphic genus $9$ Lefschetz fibrations by using the divisibility of the second integral cohomology group~\cite{Smith:98}.
We also studied Lefschetz fibration structures on $E(n)_K$ using several families of fibered knots $K$ such as $2$-bridge knots and Kanenobu knots, and we obtained some fruitful results on them such as a family of some knot surgery $4$-manifolds admit at least two nonisomorphic Lefschetz fibration structures~\cite{Park-Yun:2009, Park-Yun:2011}.

 In this article we show that, for each integer $n>0$, some $E(2)_K$ admit at least  $2^n$ nonisomorphic Lefschetz fibration structures, which is a significant extension of our previous result mentioned above. In order to find such examples, we first perform a knot surgery on $E(2)$ using a family of connected sums of fibered knots obtained by Stallings twist from a slice knot $3_1 \sharp 3^*_1$ and we consider the corresponding monodromy factorizations and monodromy groups. 
Then, using a graph method developed by S. Humphries~\cite{Humphries:79}, 
we prove that these monodromy groups (so the corresponding monodromy factorizations) are mutually distinct, which is a key step in the proof. By further investigation of these monodromy factorizations, we finally conclude that the corresponding Lefschetz fibration structures are mutually nonisomorphic to each other. The main result of this article is following:  
 
\begin{theorem}
For each integer $n > 0$ and $(m_1, m_2, \cdots, m_n) \in \mathbb{Z}^n$, 
a knot surgery $4$-manifold
\[
E(2)_{K_{m_1} \sharp K_{m_2} \sharp \cdots \sharp K_{m_n}}
\]
admits at least $2^n$ nonisomorphic genus $(4n+1)$ Lefschetz fibrations 
over $S^2$. Here $K_{m_i} (1 \leq i \leq n)$ denotes a fibered knot obtained by performing $|m_i|$ left/right handed full twist on a slice knot $K_0=3_1 \sharp 3^{*}_1$ as in Figure~\ref{fig:Kn}.
\end{theorem}

\begin{remark}
 Recently I. Baykur obtained a similar result for non-minimal cases. That is, he proved that, for any closed symplectic $4$-manifold $X$ which is not a rational or ruled surface and any integer $n>0$, there are $n$ nonisomorphic Lefschetz pencils of the same genus on a blowing-up of $X$ ~\cite{Baykur:2014}. 
\end{remark}

\subsection*{Acknowledgment}
 Part of this work was done during the authors' visit to the Max Planck Institute for Mathematics in Bonn. The authors wish to thank MPIM for hospitality and financial support.
 Jongil Park was supported by Leaders Research Grant funded by Seoul National University and by the National Research Foundation of Korea Grant (2010-0019516). He also holds a joint appointment at KIAS and in the Research Institute of Mathematics, SNU.
 Ki-Heon Yun was supported by the National Research Foundation of Korea Grant  (2012R1A1B4003427 and 2014K2A7A1043948).


\section{Preliminaries}\label{section:prelim}

In this section we briefly review some basic facts about knot surgery $4$-manifolds, Lefschetz fibrations and monodromy factorizations, and Humphries' graph method on the mapping class group of surfaces for completeness. 

\subsection{Knot surgery $4$-manifold} 
\begin{definition}
 Let $X$ be a closed simply connected smooth $4$-manifold which contains
 an embedded torus $T$ of square $0$. Then, for any knot $K \subset S^3$,
 one can construct a new $4$-manifold $X_K$,
 called {\em a knot surgery $4$-manifold},
 \[ X_K= X\sharp_{T = T_m}  S^1 \times M_K  = [X\setminus (T\times D^2)]
    \cup [S^1 \times (S^3 \setminus N(K))] \]
 by taking a fiber sum along a torus $T$ in $X$ and
 $T_m = S^1\times m$ in \mbox{$S^1 \times M_K$}
 requiring that in the second expression the two pieces are glued together
 in such a way that the homology class $[pt \times \partial D^2]$
 is identified with $[pt \times l ]$,
 where $M_K$ is the $3$-manifold obtained by
 doing $0$-framed surgery along $K$, and $m$ and $l$ are the meridian and
 the longitude of $K$ respectively.
\end{definition}

\begin{theorem}[\cite{FS:98}]
 \label{thm:FS-knot}
 Suppose that $X$ is a smooth 4-manifold
 containing a $c$-embedded torus $T$ and
 $\pi_1(X) = 1 = \pi_1(X\setminus T)$.
 Then a knot surgery $4$-manifold $X_K$ is homeomorphic to $X$
 and the Seiberg-Witten invariant is given by
 \[\mathcal{SW}_{X_K} = \mathcal{SW}_X\cdot \Delta_K(t) ,\]
 where $t=exp(2[T])$ and $\Delta_K$ is the symmetrized Alexander
 polynomial of $K$.
\end{theorem}

\begin{remark}
Let $E(2)$ be a simply connected elliptic surface with holomorphic Euler characteristic $2$ and no multiple fibers. Suppose that $T$ is a generic elliptic fiber of $E(2)$ and $K$ is any fibered knot in $S^3$. Then $E(2)_K$ admits not only a symplectic structure but also a Lefscetz fibration structure ~\cite{FS:2004}.
\end{remark}

\subsection{Lefschetz fibration}

\begin{definition}
 \label{defn:lefschetz}
 Let $X$ be a compact, oriented smooth 4-manifold. A
 Lefschetz fibration is a proper smooth map $\pi : X \to B$, where $B$
 is a compact connected oriented surface and $\pi^{-1}(\partial B) =
 \partial X$ such that
\begin{itemize}
 \item[(1)]  the set of critical points $C= \{p_1, p_2, \cdots, p_n\}$ of $\pi$
        is non-empty and lies in $int(X)$ and $\pi$ is injective on $C$
 \item[(2)]  for each $p_i$ and $b_i:=\pi(p_i)$, there are local complex
       coordinate charts agreeing with the orientations of $X$ and $B$ such that
        $\pi$ can be expressed as $\pi(z_1, z_2) = z_1^2 + z_2^2$.
\end{itemize}
\end{definition}

 Since each singular point in a Lefschetz fibration is related to a right-handed Dehn twist, if $X$ is a Lefschetz fibration over $S^2$ with generic fiber $F$, then it gives a sequence of right-handed Dehn twists 
whose product becomes the identity element in the mapping class group of $F$. 
This ordered sequence of right-handed Dehn twists is called 
\emph{monodromy factorization} of the Lefschetz fibration.
 Note that monodromy factorization is well defined up to Hurwitz equivalence and global conjugation equivalence~\cite{Kas:80,Matsumoto:96,GS:99}.

\begin{definition}
Two monodromy factorizations $W_1$ and $W_2$ are called \emph{Hurwitz equivalence} 
if $W_1$ can be changed to $W_2$ in finitely many steps of the following two operations:
\begin{itemize}
 \item[(1)] \emph{Hurwitz move:}
       $t_{c_n}\cdot ... \cdot t_{c_{i+1}} \cdot t_{c_i}
       \cdot ... \cdot t_{c_1} \sim  t_{c_n}\cdot ... \cdot t_{c_{i+1}}(t_{c_i})
       \cdot t_{ c_{i+1}} \cdot ... \cdot t_{c_1}$
 \item[(2)] \emph{inverse Hurwitz move:}
       $t_{c_n}\cdot ... \cdot t_{c_{i+1}} \cdot
       t_{ c_i} \cdot ... \cdot t_{c_1} \sim  t_{c_n}\cdot ...  \cdot t_{ c_{i}}
       \cdot t_{c_{i}}^{-1}(t_{c_{i+1}}) \cdot ... \cdot t_{c_1}$
\end{itemize}
and this relation comes from the choice of Hurwitz system, a set of arcs which connecting the base point $b_0$ to $b_i$.

 Furthermore, a choice of generic fiber also gives another equivalence relation. Two monodromy factorizations $W_1$ and $W_2$ are called 
\emph{global conjugation equivalence} 
if $W_2 = f(W_1)$ for some $f\in \mathcal{M}_g$, where $\Sigma_g$ is a generic fiber of the Lefschetz fibration $W_1$.
\end{definition}

\begin{definition}
 Two Lefschetz fibrations $f_1: X_1\to B_1$, $f_2:  X_2\to B_2$ are
 called \emph{isomorphic} if
 there are orientation preserving diffeomorphisms $H: X_1\to X_2$ and
 $h:B_1\to B_2$ such that the following diagram commutes:
 \begin{equation}
 \begin{CD}
    X_1  @>H>> X_2 \\
    @V{f_1}VV    @VV{f_2}V \\
    B_1 @>h>> B_2
 \end{CD}
 \end{equation}
\end{definition}

\begin{definition}\label{defn:monodromy_group}
Let $\pi : X \to S^2$ be a Lefschetz fibration and let $F$ be a fixed generic
fiber of the Lefschetz fibration. Let $W = w_n\cdot ... \cdot w_2 \cdot w_1$ be a monodromy factorization of the Lefschetz fibration corresponding to $F$.
Then \emph{monodromy group} $G_F(W) \subset \mathcal{M}_F$ is defined to be a subgroup of $\mathcal{M}_F$ generated by $\{w_1, w_2, \cdots, w_n\}$. 
\end{definition}

\begin{lemma}\label{lemma:mon-Group}
 If two monodromy factorizations $W_1$ and $W_2$ give isomorphic Lefschetz fibrations over $S^2$ with respect to the chosen generic fibers $F_1$ and $F_2$ respectively which are homeomorphic to $F$, then monodromy groups $G_{F_1}(W_1)$ and $G_{F_2}(W_2)$ are isomorphic as a subgroup of the mapping class group $\mathcal{M}_F$. Moreover if  a generic fiber $F = F_1 =F_2$ is fixed, then $G_F(W_1) = G_F(W_2)$.
\end{lemma}

\begin{remark}
As we mentioned above, a role of global conjugation equivalence is the choice of a generic fiber.
If we use the same fixed generic fiber for $W_1$ and $W_2$, that is, $F_1 = F = F_2$, then the global conjugation cannot happen. Therefore we get $G_F(W_1) = G_F(W_2)$.
\end{remark}

 Note that monodromy factorizations of knot surgery $4$-manifolds $E(n)_K$ were originally studied by R.~Fintushel and R.~Stern~\cite{FS:2004} and the second author could also find an explicit monodromy factorization of $E(n)_K$ ~\cite{Yun:2008} with the help of a factorization of the identity element in the mapping class group which were discovered by Y. Matsumoto~\cite{Matsumoto:96}, M. Korkmaz~\cite{Korkmaz:2001} and Y. Gurtas~\cite{Gurtas:2004}.

\smallskip

 Let $M(n,g)$ be the desingularization of a double cover of $\Sigma_g \times S^2$ branched over $2n(\{pt.\}\times S^2 )\cup 2(\Sigma_g \times \{ pt.\})$.

\medskip

\begin{lemma}[\cite{Korkmaz:2001}] \label{lemma:Yun}
$M(2,g)$ has monodromy factorization $\eta_{1,g}^2$ with
\[
\eta_{1,g} = t_{B_0} \cdot t_{B_1} \cdot t_{B_2}\cdot \cdots \cdot t_{B_{2g}} 
\cdot t_{B_{2g+1}} \cdot t_{b_{g+1}}^2 \cdot t_{b_{g+1}'}^2 ,
\]
where  $B_j$, $b_{g+1}$, $b_{g+1}'$ are simple closed curves on $\Sigma_{2g+1}$ as in Figure~\ref{fig:generator}.
\end{lemma}


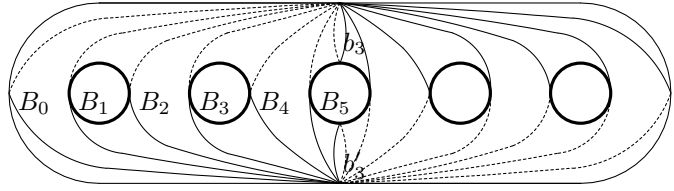
\begin{figure}[htbp]

\begin{center}

\begin{tikzpicture}[line cap=round,line join=round,>=triangle 45,x=2.0cm,y=2.0cm]
\clip(-0.093863288622,0.288605871722) rectangle (4.53308001838,1.70961467762);
\draw [line width=1.2pt] (0.6,1.) circle (0.4cm);
\draw [line width=1.2pt] (1.4,1.) circle (0.4cm);
\draw [line width=1.2pt] (2.2,1.) circle (0.4cm);
\draw (2.2,1.6)-- (0.6,1.6);
\draw [shift={(0.6,1.)}] plot[domain=1.57079632679:4.71238898038,variable=\t]({1.*0.6*cos(\t r)+0.*0.6*sin(\t r)},{0.*0.6*cos(\t r)+1.*0.6*sin(\t r)});
\draw (0.6,0.4)-- (2.2,0.4);
\draw [shift={(0.719736842105,1.26315789474)}] plot[domain=3.49212396048:4.62126274781,variable=\t]({1.*0.766337523188*cos(\t r)+0.*0.766337523188*sin(\t r)},{0.*0.766337523188*cos(\t r)+1.*0.766337523188*sin(\t r)});
\draw [shift={(3.775,36.875)}] plot[domain=4.62668875494:4.66923552698,variable=\t]({1.*36.5089886192*cos(\t r)+0.*36.5089886192*sin(\t r)},{0.*36.5089886192*cos(\t r)+1.*36.5089886192*sin(\t r)});
\draw [shift={(0.795,0.9925)}] plot[domain=3.12260759309:4.59823771333,variable=\t]({1.*0.395071196115*cos(\t r)+0.*0.395071196115*sin(\t r)},{0.*0.395071196115*cos(\t r)+1.*0.395071196115*sin(\t r)});
\draw [shift={(3.555,15.58)}] plot[domain=4.52728287323:4.62336273892,variable=\t]({1.*15.2403551468*cos(\t r)+0.*15.2403551468*sin(\t r)},{0.*15.2403551468*cos(\t r)+1.*15.2403551468*sin(\t r)});
\draw [shift={(1.21138385354,1.02955989538)}] plot[domain=3.21332414839:4.31036013076,variable=\t]({1.*0.412444496103*cos(\t r)+0.*0.412444496103*sin(\t r)},{0.*0.412444496103*cos(\t r)+1.*0.412444496103*sin(\t r)});
\draw [shift={(2.96590909091,6.69318181818)}] plot[domain=4.40537658781:4.59128000583,variable=\t]({1.*6.33961782225*cos(\t r)+0.*6.33961782225*sin(\t r)},{0.*6.33961782225*cos(\t r)+1.*6.33961782225*sin(\t r)});
\draw [shift={(1.58214285714,0.960714285714)}] plot[domain=3.03914880008:4.36123513778,variable=\t]({1.*0.384156908844*cos(\t r)+0.*0.384156908844*sin(\t r)},{0.*0.384156908844*cos(\t r)+1.*0.384156908844*sin(\t r)});
\draw [shift={(-2.875,-17.125)}] plot[domain=1.28892054902:1.33146732107,variable=\t]({1.*18.2450335708*cos(\t r)+0.*18.2450335708*sin(\t r)},{0.*18.2450335708*cos(\t r)+1.*18.2450335708*sin(\t r)});
\draw [shift={(2.39,1.27)}] plot[domain=3.47091882064:3.89310548708,variable=\t]({1.*0.83486525859*cos(\t r)+0.*0.83486525859*sin(\t r)},{0.*0.83486525859*cos(\t r)+1.*0.83486525859*sin(\t r)});
\draw [shift={(3.39,2.51)}] plot[domain=3.98540382175:4.19887516705,variable=\t]({1.*2.42243678968*cos(\t r)+0.*2.42243678968*sin(\t r)},{0.*2.42243678968*cos(\t r)+1.*2.42243678968*sin(\t r)});
\draw [shift={(2.90333333333,0.967777777778)}] plot[domain=3.10593741403:3.82074900194,variable=\t]({1.*0.903907839725*cos(\t r)+0.*0.903907839725*sin(\t r)},{0.*0.903907839725*cos(\t r)+1.*0.903907839725*sin(\t r)});
\draw [shift={(2.67502710352,0.6)}] plot[domain=2.74309061429:3.54009469289,variable=\t]({1.*0.515413182868*cos(\t r)+0.*0.515413182868*sin(\t r)},{0.*0.515413182868*cos(\t r)+1.*0.515413182868*sin(\t r)});
\draw [shift={(3.68026315789,0.736842105263)}] plot[domain=0.350531306885:1.47967009422,variable=\t]({1.*0.766337523188*cos(\t r)+0.*0.766337523188*sin(\t r)},{0.*0.766337523188*cos(\t r)+1.*0.766337523188*sin(\t r)});
\draw [shift={(0.625,-34.875)}] plot[domain=1.48509610135:1.52764287339,variable=\t]({1.*36.5089886192*cos(\t r)+0.*36.5089886192*sin(\t r)},{0.*36.5089886192*cos(\t r)+1.*36.5089886192*sin(\t r)});
\draw [shift={(3.605,1.0075)}] plot[domain=-0.0189850604988:1.45664505974,variable=\t]({1.*0.395071196115*cos(\t r)+0.*0.395071196115*sin(\t r)},{0.*0.395071196115*cos(\t r)+1.*0.395071196115*sin(\t r)});
\draw [shift={(0.845,-13.58)}] plot[domain=1.38569021964:1.48177008533,variable=\t]({1.*15.2403551468*cos(\t r)+0.*15.2403551468*sin(\t r)},{0.*15.2403551468*cos(\t r)+1.*15.2403551468*sin(\t r)});
\draw [shift={(3.18861614646,0.970440104616)}] plot[domain=0.0717314947969:1.16876747717,variable=\t]({1.*0.412444496103*cos(\t r)+0.*0.412444496103*sin(\t r)},{0.*0.412444496103*cos(\t r)+1.*0.412444496103*sin(\t r)});
\draw [shift={(1.43409090909,-4.69318181818)}] plot[domain=1.26378393422:1.44968735224,variable=\t]({1.*6.33961782225*cos(\t r)+0.*6.33961782225*sin(\t r)},{0.*6.33961782225*cos(\t r)+1.*6.33961782225*sin(\t r)});
\draw [shift={(2.81785714286,1.03928571429)}] plot[domain=-0.102443853507:1.21964248419,variable=\t]({1.*0.384156908844*cos(\t r)+0.*0.384156908844*sin(\t r)},{0.*0.384156908844*cos(\t r)+1.*0.384156908844*sin(\t r)});
\draw [shift={(7.275,19.125)}] plot[domain=4.43051320261:4.47305997466,variable=\t]({1.*18.2450335708*cos(\t r)+0.*18.2450335708*sin(\t r)},{0.*18.2450335708*cos(\t r)+1.*18.2450335708*sin(\t r)});
\draw [shift={(2.01,0.73)}] plot[domain=0.329326167048:0.751512833493,variable=\t]({1.*0.83486525859*cos(\t r)+0.*0.83486525859*sin(\t r)},{0.*0.83486525859*cos(\t r)+1.*0.83486525859*sin(\t r)});
\draw [shift={(1.01,-0.51)}] plot[domain=0.843811168163:1.05728251346,variable=\t]({1.*2.42243678968*cos(\t r)+0.*2.42243678968*sin(\t r)},{0.*2.42243678968*cos(\t r)+1.*2.42243678968*sin(\t r)});
\draw [shift={(1.49666666667,1.03222222222)}] plot[domain=-0.0356552395604:0.679156348354,variable=\t]({1.*0.903907839725*cos(\t r)+0.*0.903907839725*sin(\t r)},{0.*0.903907839725*cos(\t r)+1.*0.903907839725*sin(\t r)});
\draw [shift={(1.72497289648,1.4)}] plot[domain=-0.398502039303:0.398502039303,variable=\t]({1.*0.515413182868*cos(\t r)+0.*0.515413182868*sin(\t r)},{0.*0.515413182868*cos(\t r)+1.*0.515413182868*sin(\t r)});
\draw (3.8,1.6)-- (2.2,1.6);
\draw [shift={(0.719736842105,0.736842105263)},dash pattern=on 1pt off 1pt]  plot[domain=1.66192255937:2.7910613467,variable=\t]({1.*0.766337523188*cos(\t r)+0.*0.766337523188*sin(\t r)},{0.*0.766337523188*cos(\t r)+1.*0.766337523188*sin(\t r)});
\draw [shift={(3.775,-34.875)},dash pattern=on 1pt off 1pt]  plot[domain=1.6139497802:1.65649655224,variable=\t]({1.*36.5089886192*cos(\t r)+0.*36.5089886192*sin(\t r)},{0.*36.5089886192*cos(\t r)+1.*36.5089886192*sin(\t r)});
\draw [shift={(0.795,1.0075)},dash pattern=on 1pt off 1pt]  plot[domain=1.68494759385:3.16057771409,variable=\t]({1.*0.395071196115*cos(\t r)+0.*0.395071196115*sin(\t r)},{0.*0.395071196115*cos(\t r)+1.*0.395071196115*sin(\t r)});
\draw [shift={(3.555,-13.58)},dash pattern=on 1pt off 1pt]  plot[domain=1.65982256826:1.75590243395,variable=\t]({1.*15.2403551468*cos(\t r)+0.*15.2403551468*sin(\t r)},{0.*15.2403551468*cos(\t r)+1.*15.2403551468*sin(\t r)});
\draw [shift={(1.21138385354,0.970440104616)},dash pattern=on 1pt off 1pt]  plot[domain=1.97282517642:3.06986115879,variable=\t]({1.*0.412444496103*cos(\t r)+0.*0.412444496103*sin(\t r)},{0.*0.412444496103*cos(\t r)+1.*0.412444496103*sin(\t r)});
\draw [shift={(2.96590909091,-4.69318181818)},dash pattern=on 1pt off 1pt]  plot[domain=1.69190530135:1.87780871937,variable=\t]({1.*6.33961782225*cos(\t r)+0.*6.33961782225*sin(\t r)},{0.*6.33961782225*cos(\t r)+1.*6.33961782225*sin(\t r)});
\draw [shift={(1.58214285714,1.03928571429)},dash pattern=on 1pt off 1pt]  plot[domain=1.9219501694:3.2440365071,variable=\t]({1.*0.384156908844*cos(\t r)+0.*0.384156908844*sin(\t r)},{0.*0.384156908844*cos(\t r)+1.*0.384156908844*sin(\t r)});
\draw [shift={(-2.875,19.125)},dash pattern=on 1pt off 1pt]  plot[domain=4.95171798611:4.99426475816,variable=\t]({1.*18.2450335708*cos(\t r)+0.*18.2450335708*sin(\t r)},{0.*18.2450335708*cos(\t r)+1.*18.2450335708*sin(\t r)});
\draw [shift={(2.39,0.73)},dash pattern=on 1pt off 1pt]  plot[domain=2.3900798201:2.81226648654,variable=\t]({1.*0.83486525859*cos(\t r)+0.*0.83486525859*sin(\t r)},{0.*0.83486525859*cos(\t r)+1.*0.83486525859*sin(\t r)});
\draw [shift={(3.39,-0.51)},dash pattern=on 1pt off 1pt]  plot[domain=2.08431014013:2.29778148543,variable=\t]({1.*2.42243678968*cos(\t r)+0.*2.42243678968*sin(\t r)},{0.*2.42243678968*cos(\t r)+1.*2.42243678968*sin(\t r)});
\draw [shift={(2.90333333333,1.03222222222)},dash pattern=on 1pt off 1pt]  plot[domain=2.46243630524:3.17724789315,variable=\t]({1.*0.903907839725*cos(\t r)+0.*0.903907839725*sin(\t r)},{0.*0.903907839725*cos(\t r)+1.*0.903907839725*sin(\t r)});
\draw [shift={(2.67502710352,1.4)},dash pattern=on 1pt off 1pt]  plot[domain=2.74309061429:3.54009469289,variable=\t]({1.*0.515413182868*cos(\t r)+0.*0.515413182868*sin(\t r)},{0.*0.515413182868*cos(\t r)+1.*0.515413182868*sin(\t r)});
\draw (0.6,1.6)-- (2.2,1.6);
\draw [shift={(3.68026315789,1.26315789474)},dash pattern=on 1pt off 1pt]  plot[domain=4.80351521296:5.93265400029,variable=\t]({1.*0.766337523188*cos(\t r)+0.*0.766337523188*sin(\t r)},{0.*0.766337523188*cos(\t r)+1.*0.766337523188*sin(\t r)});
\draw [shift={(0.625,36.875)},dash pattern=on 1pt off 1pt]  plot[domain=4.75554243379:4.79808920583,variable=\t]({1.*36.5089886192*cos(\t r)+0.*36.5089886192*sin(\t r)},{0.*36.5089886192*cos(\t r)+1.*36.5089886192*sin(\t r)});
\draw [shift={(3.605,0.9925)},dash pattern=on 1pt off 1pt]  plot[domain=-1.45664505974:0.0189850604988,variable=\t]({1.*0.395071196115*cos(\t r)+0.*0.395071196115*sin(\t r)},{0.*0.395071196115*cos(\t r)+1.*0.395071196115*sin(\t r)});
\draw [shift={(0.845,15.58)},dash pattern=on 1pt off 1pt]  plot[domain=4.80141522185:4.89749508754,variable=\t]({1.*15.2403551468*cos(\t r)+0.*15.2403551468*sin(\t r)},{0.*15.2403551468*cos(\t r)+1.*15.2403551468*sin(\t r)});
\draw [shift={(3.18861614646,1.02955989538)},dash pattern=on 1pt off 1pt]  plot[domain=5.11441783001:6.21145381238,variable=\t]({1.*0.412444496103*cos(\t r)+0.*0.412444496103*sin(\t r)},{0.*0.412444496103*cos(\t r)+1.*0.412444496103*sin(\t r)});
\draw [shift={(1.43409090909,6.69318181818)},dash pattern=on 1pt off 1pt]  plot[domain=4.83349795494:5.01940137296,variable=\t]({1.*6.33961782225*cos(\t r)+0.*6.33961782225*sin(\t r)},{0.*6.33961782225*cos(\t r)+1.*6.33961782225*sin(\t r)});
\draw [shift={(2.81785714286,0.960714285714)},dash pattern=on 1pt off 1pt]  plot[domain=-1.21964248419:0.102443853507,variable=\t]({1.*0.384156908844*cos(\t r)+0.*0.384156908844*sin(\t r)},{0.*0.384156908844*cos(\t r)+1.*0.384156908844*sin(\t r)});
\draw [shift={(7.275,-17.125)},dash pattern=on 1pt off 1pt]  plot[domain=1.81012533252:1.85267210457,variable=\t]({1.*18.2450335708*cos(\t r)+0.*18.2450335708*sin(\t r)},{0.*18.2450335708*cos(\t r)+1.*18.2450335708*sin(\t r)});
\draw [shift={(2.01,1.27)},dash pattern=on 1pt off 1pt]  plot[domain=5.53167247369:5.95385914013,variable=\t]({1.*0.83486525859*cos(\t r)+0.*0.83486525859*sin(\t r)},{0.*0.83486525859*cos(\t r)+1.*0.83486525859*sin(\t r)});
\draw [shift={(1.01,2.51)},dash pattern=on 1pt off 1pt]  plot[domain=5.22590279372:5.43937413902,variable=\t]({1.*2.42243678968*cos(\t r)+0.*2.42243678968*sin(\t r)},{0.*2.42243678968*cos(\t r)+1.*2.42243678968*sin(\t r)});
\draw [shift={(1.49666666667,0.967777777778)},dash pattern=on 1pt off 1pt]  plot[domain=-0.679156348354:0.0356552395604,variable=\t]({1.*0.903907839725*cos(\t r)+0.*0.903907839725*sin(\t r)},{0.*0.903907839725*cos(\t r)+1.*0.903907839725*sin(\t r)});
\draw [shift={(1.72497289648,0.6)},dash pattern=on 1pt off 1pt]  plot[domain=-0.398502039303:0.398502039303,variable=\t]({1.*0.515413182868*cos(\t r)+0.*0.515413182868*sin(\t r)},{0.*0.515413182868*cos(\t r)+1.*0.515413182868*sin(\t r)});
\draw (2.2,0.4)-- (3.8,0.4);
\draw (3.8,0.4)-- (2.2,0.4);
\draw [line width=1.2pt] (3.,1.) circle (0.4cm);
\draw [line width=1.2pt] (3.8,1.) circle (0.4cm);
\draw [shift={(3.8,1.)}] plot[domain=-1.57079632679:1.57079632679,variable=\t]({1.*0.6*cos(\t r)+0.*0.6*sin(\t r)},{0.*0.6*cos(\t r)+1.*0.6*sin(\t r)});
\draw (-0.00143995165326,1.0684277774) node[anchor=north west] {$B_0$};
\draw (0.397135689025,1.0684277774) node[anchor=north west] {$B_1$};
\draw (0.801487788263,1.0684277774) node[anchor=north west] {$B_2$};
\draw (1.20006342894,1.0684277774) node[anchor=north west] {$B_3$};
\draw (1.59863906962,1.0684277774) node[anchor=north west] {$B_4$};
\draw (1.9972147103,1.0684277774) node[anchor=north west] {$B_5$};
\draw (2.15895554999,0.669852136719) node[anchor=north west] {$b_3'$};
\draw (2.15895554999,1.46700341807) node[anchor=north west] {$b_3$};
\end{tikzpicture}
\end{center}
\caption{Vanishing cycles on $\eta_{1,g}$ with $g=2$}\label{fig:generator}
\end{figure}

 Next, it is also well-known fact that a fibered knot or link can be obtained from an unknot by a sequence of Hopf plumbings, deplumbings and Stallings twists. Let $L\subset S^3$ be a fibered link, then it has a fibration structure over $S^1$ 
\[
S^3 \setminus N(L) \approx  (\Sigma_L \times [0,1])/_{(x, 1) \sim (\phi_L(x), 0)}.
\]
We denote this fibration structure by $(\Sigma_L, \phi_L)$.

\begin{lemma}[\cite{Harer:82}]
\label{lemma:Harer}
 Every fibered link in $S^3$ is related to the unknot
 by a sequence of Hopf plumbings, deplumbings and twistings.
 \begin{itemize}
 \item[(1)] If $(\Sigma_L, \phi_L)$ is a fibration structure of a fibered link $L$ and 
 $\Sigma_{L'}$ is a plumbing of positive(left handed) Hopf band $H^+$ and $\Sigma_L$, then $L'$ 
 has a fibration structure $(\Sigma_{L'}, t_c \circ \phi_L)$,
 where $c$ is the core circle of the Hopf band. 
 If we perform negative(right handed) Hopf band plumbing, 
 then the monodromy map becomes $t_c^{-1}\circ \phi_L$.
 \item[(2)] Let $(\Sigma_L, \phi_L)$ be a fibration structure of a fibered link $L$ 
 and suppose that there is an embedded circle $c$ in $\Sigma_L$ 
 which is unknotted in $S^3$ and $\mathrm{lk}(c, c^+) = 0$ where $c^+$ is a push off 
 of $c$ in the chosen positive direction of $\Sigma_L$. 
 Then a $(\pm 1)$-Dehn surgery along $c^+$ yield a new fibration structure
 $(\Sigma_{L'}, t_c^{\pm 1} \circ \phi_L)$ - This operation is called Stallings twist.
 \end{itemize}
\end{lemma}

\begin{theorem}[\cite{FS:2004}]\label{theorem:FS2004}
Let $K \subset S^3$ be a fibered knot of genus $g$. Then $E(2)_K$ has 
a monodromy factorization of the form
\[
\Phi_K (\eta_{1, g}^2) \cdot  \eta_{1, g}^2 ,
\]
where $\eta_{1, g}$ as in Lemma~\ref{lemma:Yun} and
\[
\Phi_K = \phi_K\oplus id \oplus id:
\Sigma_g \sharp \Sigma_{1} \sharp \Sigma_g \to \Sigma_g \sharp \Sigma_{1} \sharp \Sigma_g
\]
is an extension of a monodromy $\phi_K$ of the fibered knot $K$ such that
\[
S^3 \setminus N(K) = (\Sigma_g^1 \times [0,1])/ ((x, 1) \sim (\phi_K(x), 0))
\]
where $\Sigma_g^1$ is an oriented surface of genus $g\!=\! g(K)$
with one boundary component.
\end{theorem}

\begin{remark}
Note that $M(2,g)$ has a genus $(2g+1)$ Lefschetz fibration structure over $S^2$ with monodromy factorization $\eta_{1,g}^2$ corresponding to a fixed generic fiber 
\[
F = \Sigma_{2g+1} =\Sigma_g \sharp \Sigma_{1} \sharp \Sigma_g= \Sigma_{g}^1 \cup \Sigma_{1}^2 \cup \Sigma_{g}^1 
\]
and $E(2)_K$ can also be considered as a twisted fiber sum of two $M(2,g)$'s. 
Here a fiber surface of the fibered knot $K$ is identified 
with the first $\Sigma^1_g$.

Let us investigate a role of the choice of generic fiber in Lefschetz fibrations of $E(2)_K$.
Suppose $\tilde{F}$ is another choice of generic fiber of genus $(2g+1)$ and 
$f: F \to \tilde{F}$ be an orientation preserving diffeomorphism as 
in the following diagram:
\begin{equation}
 \begin{CD}
    \Sigma_{2g+1}  @= F @>f>> \tilde{F} \\
    @V{\Phi}VV    @V{\Phi}VV  @VV{\tilde{\Phi}}V \\
    \Sigma_{2g+1}  @= F @>f>> \tilde{F}
 \end{CD}
 \end{equation}
Then a monodromy factorization of $M(2,g)$ corresponding to the generic fiber $\tilde{F}$ becomes $f(\eta_{1,2}^2)$ and a monodromy map of fibered knot $K$ satisfies
\begin{align*}
S^3 \setminus N(K) &= (\Sigma_g^1 \times [0,1])/ ((x, 1) \sim (\phi_K(x), 0)) \\
&= (f(\Sigma_g^1) \times [0,1])/ ((x, 1) \sim (f\circ\phi_K\circ f^{-1}(x), 0)).
\end{align*}

Therefore a monodromy factorization of $E(2)_K$ corresponding to the generic fiber $\tilde{F}$ can be read as
\[
(f\circ \Phi_K \circ f^{-1})(f(\eta_{1,g}^2)) \cdot f(\eta_{1,g}^2) = f(\Phi_K(\eta_{1,g}^2) \cdot \eta_{1,g}^2).
\] 
Hence, in each equivalent class of monodromy factorizations of $E(2)_K$, we can always choose a monodromy factorization with respect to the chosen generic fiber $F$. This monodromy factorization depends on the choice of monodromy map $\phi_K$ corresponding to the fiber surface $\Sigma_g^1$ and is well defined up to Hurwitz equivalence.
This means that if two monodromy factorizations of $E(2)_K$ with respect to the chosen same generic fiber $F$ are not related by a sequence of Hurwitz moves and inverse Hurwitz moves, then they are located in different equivalence class of monodromy factorizations and therefore they are not isomorphic to each other.
\end{remark}

\subsection{Humphries' graph method}

\begin{definition} \label{definition:graph}
 Suppose $\{\gamma_1, \gamma_2, \cdots, \gamma_{2g}\}$ is a set of simple closed curves on $\Sigma_{g}$ which generates $H_1(\Sigma_g;\mathbb{Z}_2)$. 
Let us define a graph $\Gamma(\gamma_1, \gamma_2, \cdots, \gamma_{2g})$ as follows:
\begin{itemize}
 \item a vertex for each simple closed curve $\gamma_i (1 \leq i \leq 2g)$
 \item an edge between $\gamma_i$ and $\gamma_j$ 
       if $i_2(\gamma_i, \gamma_j) = 1$, where $i_2(\gamma_i, \gamma_j)$ 
       is the modulo $2$ intersection number between two curves $\gamma_i$
       and $\gamma_j$
 \item we assume there is no intersection between any two edges. 
\end{itemize}
Then, for any simple closed curve $\gamma$ on $\Sigma_g$, we can express
$\gamma = \sum_{i=1}^{2g} \varepsilon_i \gamma_i, (\varepsilon_i = 0, 1)$, 
as an element of $H_1(\Sigma_g; \mathbb{Z}_2)$. 
By denoting $\overline{\gamma} := \sum_{\varepsilon_i = 1} \overline{\gamma_i}$, where $\overline{\gamma_i}$ is the union of all closure of half edges with one end vertex $\gamma_i$,
we define $\chi_\Gamma(\gamma):= \chi_\Gamma(\overline{\gamma})$ as 
modulo $2$ Euler number.
\end{definition}

 In~\cite{Humphries:79}, Humphries showed that the minimal number of Dehn twist generators of the mapping class group $\mathcal{M}_g$ or $\mathcal{M}_g^1$ is $2g+1$ by using symplectic transvection and modulo $2$ Euler number of a graph. Furthermore, he also proved the following:

\begin{lemma}[\cite{Humphries:79}, \cite{Park-Yun:2011}]\label{lemma:chi=1}
Let $\Gamma(\gamma_1, \gamma_2, \cdots, \gamma_{2g})$ be a graph corresponding to a set of simple closed curves $\{ \gamma_1, \gamma_2, \cdots, \gamma_{2g} \}$ which generates $H_1(\Sigma_g; \mathbb{Z}_2)$.  
Suppose $G_{\Gamma,g}$ is a subgroup of $\mathcal{M}_g$ generated by
\[
 \{ t_\alpha \ | \   \alpha \textrm{ is a nonseparating simple closed curve on } 
 \Sigma_g \textrm{ such that } \chi_\Gamma(\alpha)=1 \}.
\]
Then $G_{\Gamma, g}$ is a nontrivial proper subgroup of $\mathcal{M}_g$. 
Moreover, if $\beta$ is a nonseparating simple closed curve on $\Sigma_g$ 
with $\chi_\Gamma (\beta) = 0$, then $t_\beta \not\in G_{\Gamma, g}$.
\end{lemma}

\begin{corollary}[\cite{Humphries:79}, \cite{Park-Yun:2011}]\label{corollary:Humphries}
 Under the same notation as in Lemma~\ref{lemma:chi=1} above, 
 for any nonseparating simple closed curves $c$ and $\gamma$ on $\Sigma_g$,
 we have   
\begin{enumerate}
\item if $\chi_\Gamma (c) = 1$, then $\chi_\Gamma(t_c(\gamma)) = \chi_\Gamma(\gamma)$
\item if $\chi_\Gamma (c) =0$, then 
$\chi_\Gamma(t_c(\gamma)) = \chi_\Gamma(\gamma) + i_2(c, \gamma)$.
\end{enumerate}
\end{corollary}

\medskip


\section{Nonisomorphic Lefschetz fibrations}\label{section:nonisom}

As mentioned in Introduction above, we have studied nonisomorphic Lefschetz fibration structures on knot surgery $4$-manifolds $E(n)_K$
for $2$-bridge knot case~\cite{Park-Yun:2009} and for Kanenobu knot case~\cite{Park-Yun:2011}, and we could show that they have at least two nonisomorphic Lefschetz fibration structures in each case. 
Recently I. Baykur obtained a similar result about Lefschetz fibration structures on non-minmal symplectic $4$-manifolds, 
which is following 

\vspace{.5 em}

\begin{theorem}[\cite{Baykur:2014}] 
Given any closed symplectic $4$-manifold $X$ which is not a rational or ruled surface and any integer $n>0$, there are $n$ nonisomorphic Lefschetz pencils of the same genus on a blow-up of $X$, which are not equivalent via fibered Luttinger surgeries. These pencils can be chosen so that they only have nonseparating vanishing cycles. 
\end{theorem}

In this section we construct a family of simply connected minimal symplectic  $4$-manifolds $E(2)_K$, each of which admits arbitrarily many nonismorphic Lefschetz fibration structures with the same genus fiber. 
 In order to obtain such families, we first construct a family of connected sums of fibered knots obtained by Stallings twist from a slice knot $3_1 \sharp 3^*_1$ as follows:

\subsection{Square knot as a building block}
Let $K_0$ be a slice knot $3_1 \sharp 3_1^*$ and $K_n$ be a knot obtained by Stallings twist from the slice knot $3_1 \sharp 3^*_1$ as in Figure~\ref{fig:Kn}, where $n$ in the box means $n$ left handed full twists when $n$ is a positive integer and $|n|$ right handed full twist when $n$ is a negative integer. Then it is well known that
 
\begin{lemma}[\cite{Stallings:78},\cite{Quach_Weber:79}]\label{Lem:Stallings}
 Let $K_n$ be a knot as in Figure~\ref{fig:Kn}, then
\begin{itemize}
\item[(1)] $\Delta_{K_n} = (t^2 - t + 1)^2$ 
\item[(2)] $K_n$ has an Alexander matrix 
\[\begin{pmatrix} 
1 & -1 & 0 & 0 \\ 
0 & n+1 & -n & 0 \\ 
0 & -n & n-1 & 0 \\ 
0 & 0 & 1 & -1 
\end{pmatrix} \]
\item[(3)] $K_n$ has the second Alexander module 
\[
\begin{pmatrix} t^2 - t + 1 & n t \\ 0 & t^2 - t + 1 \end{pmatrix}
\]
\end{itemize}
\end{lemma}

\vspace{-.5 em}

\begin{figure}[htbp]
\begin{center}
\definecolor{qqqqff}{rgb}{0.,0.,1.}
\definecolor{ffqqqq}{rgb}{1.,0.,0.}
\begin{tikzpicture}[line cap=round,line join=round,>=triangle 45,x=1.0cm,y=1.0cm]
\clip(-0.408763256412,1.05) rectangle (8.65365351446,3.85);
\draw [shift={(1.4,2.)}] plot[domain=2.35619449019:3.92699081699,variable=\t]({1.*0.565685424949*cos(\t r)+0.*0.565685424949*sin(\t r)},{0.*0.565685424949*cos(\t r)+1.*0.565685424949*sin(\t r)});
\draw [shift={(0.8,2.)}] plot[domain=-0.785398163397:0.785398163397,variable=\t]({1.*0.565685424949*cos(\t r)+0.*0.565685424949*sin(\t r)},{0.*0.565685424949*cos(\t r)+1.*0.565685424949*sin(\t r)});
\draw [shift={(1.4,3.)}] plot[domain=2.35619449019:3.92699081699,variable=\t]({1.*0.565685424949*cos(\t r)+0.*0.565685424949*sin(\t r)},{0.*0.565685424949*cos(\t r)+1.*0.565685424949*sin(\t r)});
\draw (1.,2.4)-- (1.2,2.6);
\draw (1.,1.4)-- (1.2,1.6);
\draw [shift={(0.6,3.)}] plot[domain=0.982793723247:2.55359005004,variable=\t]({1.*0.721110255093*cos(\t r)+0.*0.721110255093*sin(\t r)},{0.*0.721110255093*cos(\t r)+1.*0.721110255093*sin(\t r)});
\draw [shift={(1.6,2.4)}] plot[domain=2.58299333825:3.70019196893,variable=\t]({1.*1.88679622641*cos(\t r)+0.*1.88679622641*sin(\t r)},{0.*1.88679622641*cos(\t r)+1.*1.88679622641*sin(\t r)});
\draw (1.,3.4)-- (1.2,3.6);
\draw [shift={(1.6,3.2)}] plot[domain=2.35619449019:3.13352128024,variable=\t]({0.707106781187*0.565685424949*cos(\t r)+0.707106781187*0.565685424949*sin(\t r)},{-0.707106781187*0.565685424949*cos(\t r)+0.707106781187*0.565685424949*sin(\t r)});
\draw (1.2,3.2)-- (1.2,2.8);
\draw (1.2,2.8)-- (2.,2.8);
\draw (2.,2.8)-- (2.,3.2);
\draw (2.,3.2)-- (1.2,3.2);
\draw [shift={(0.8,3.)}] plot[domain=0.361367123907:0.785398163397,variable=\t]({1.*0.565685424949*cos(\t r)+0.*0.565685424949*sin(\t r)},{0.*0.565685424949*cos(\t r)+1.*0.565685424949*sin(\t r)});
\draw [shift={(0.8,3.)}] plot[domain=5.49778714378:5.92181818327,variable=\t]({1.*0.565685424949*cos(\t r)+0.*0.565685424949*sin(\t r)},{0.*0.565685424949*cos(\t r)+1.*0.565685424949*sin(\t r)});
\draw [shift={(1.6,1.6)}] plot[domain=3.60524026259:4.71238898038,variable=\t]({1.*0.4472135955*cos(\t r)+0.*0.4472135955*sin(\t r)},{0.*0.4472135955*cos(\t r)+1.*0.4472135955*sin(\t r)});
\draw [shift={(0.5,1.9)}] plot[domain=3.92699081699:5.49778714378,variable=\t]({1.*0.707106781187*cos(\t r)+0.*0.707106781187*sin(\t r)},{0.*0.707106781187*cos(\t r)+1.*0.707106781187*sin(\t r)});
\draw [shift={(6.4,2.)}] plot[domain=2.35619449019:3.92699081699,variable=\t]({1.*0.565685424949*cos(\t r)+0.*0.565685424949*sin(\t r)},{0.*0.565685424949*cos(\t r)+1.*0.565685424949*sin(\t r)});
\draw [shift={(5.8,2.)}] plot[domain=-0.785398163397:0.785398163397,variable=\t]({1.*0.565685424949*cos(\t r)+0.*0.565685424949*sin(\t r)},{0.*0.565685424949*cos(\t r)+1.*0.565685424949*sin(\t r)});
\draw [shift={(6.4,3.)}] plot[domain=2.35619449019:3.92699081699,variable=\t]({1.*0.565685424949*cos(\t r)+0.*0.565685424949*sin(\t r)},{0.*0.565685424949*cos(\t r)+1.*0.565685424949*sin(\t r)});
\draw [shift={(5.6,3.)}] plot[domain=0.982793723247:2.55359005004,variable=\t]({1.*0.721110255093*cos(\t r)+0.*0.721110255093*sin(\t r)},{0.*0.721110255093*cos(\t r)+1.*0.721110255093*sin(\t r)});
\draw [shift={(6.6,2.4)}] plot[domain=2.58299333825:3.70019196893,variable=\t]({1.*1.88679622641*cos(\t r)+0.*1.88679622641*sin(\t r)},{0.*1.88679622641*cos(\t r)+1.*1.88679622641*sin(\t r)});
\draw [shift={(6.6,3.2)}] plot[domain=2.35619449019:3.13352128024,variable=\t]({0.707106781187*0.565685424949*cos(\t r)+0.707106781187*0.565685424949*sin(\t r)},{-0.707106781187*0.565685424949*cos(\t r)+0.707106781187*0.565685424949*sin(\t r)});
\draw [shift={(5.8,3.)}] plot[domain=0.361367123907:0.785398163397,variable=\t]({1.*0.565685424949*cos(\t r)+0.*0.565685424949*sin(\t r)},{0.*0.565685424949*cos(\t r)+1.*0.565685424949*sin(\t r)});
\draw [shift={(5.8,3.)}] plot[domain=5.49778714378:5.92181818327,variable=\t]({1.*0.565685424949*cos(\t r)+0.*0.565685424949*sin(\t r)},{0.*0.565685424949*cos(\t r)+1.*0.565685424949*sin(\t r)});
\draw [shift={(6.6,1.6)}] plot[domain=3.60524026259:4.71238898038,variable=\t]({1.*0.4472135955*cos(\t r)+0.*0.4472135955*sin(\t r)},{0.*0.4472135955*cos(\t r)+1.*0.4472135955*sin(\t r)});
\draw [shift={(5.5,1.9)}] plot[domain=3.92699081699:5.49778714378,variable=\t]({1.*0.707106781187*cos(\t r)+0.*0.707106781187*sin(\t r)},{0.*0.707106781187*cos(\t r)+1.*0.707106781187*sin(\t r)});
\draw (6.,2.4)-- (6.2,2.6);
\draw (6.,1.4)-- (6.2,1.6);
\draw (6.,3.4)-- (6.2,3.6);
\draw [shift={(1.8,2.)}] plot[domain=2.35619449019:3.92699081699,variable=\t]({-1.*0.565685424949*cos(\t r)+0.*0.565685424949*sin(\t r)},{0.*0.565685424949*cos(\t r)+1.*0.565685424949*sin(\t r)});
\draw [shift={(2.4,2.)}] plot[domain=-0.785398163397:0.785398163397,variable=\t]({-1.*0.565685424949*cos(\t r)+0.*0.565685424949*sin(\t r)},{0.*0.565685424949*cos(\t r)+1.*0.565685424949*sin(\t r)});
\draw [shift={(1.8,3.)}] plot[domain=2.35619449019:3.92699081699,variable=\t]({-1.*0.565685424949*cos(\t r)+0.*0.565685424949*sin(\t r)},{0.*0.565685424949*cos(\t r)+1.*0.565685424949*sin(\t r)});
\draw [shift={(2.6,3.)}] plot[domain=0.982793723247:2.55359005004,variable=\t]({-1.*0.721110255093*cos(\t r)+0.*0.721110255093*sin(\t r)},{0.*0.721110255093*cos(\t r)+1.*0.721110255093*sin(\t r)});
\draw [shift={(1.6,2.4)}] plot[domain=2.58299333825:3.70019196893,variable=\t]({-1.*1.88679622641*cos(\t r)+0.*1.88679622641*sin(\t r)},{0.*1.88679622641*cos(\t r)+1.*1.88679622641*sin(\t r)});
\draw [shift={(1.6,3.2)}] plot[domain=2.35619449019:3.13352128024,variable=\t]({-0.707106781187*0.565685424949*cos(\t r)+-0.707106781187*0.565685424949*sin(\t r)},{-0.707106781187*0.565685424949*cos(\t r)+0.707106781187*0.565685424949*sin(\t r)});
\draw [shift={(2.4,3.)}] plot[domain=0.361367123907:0.785398163397,variable=\t]({-1.*0.565685424949*cos(\t r)+0.*0.565685424949*sin(\t r)},{0.*0.565685424949*cos(\t r)+1.*0.565685424949*sin(\t r)});
\draw [shift={(2.4,3.)}] plot[domain=5.49778714378:5.92181818327,variable=\t]({-1.*0.565685424949*cos(\t r)+0.*0.565685424949*sin(\t r)},{0.*0.565685424949*cos(\t r)+1.*0.565685424949*sin(\t r)});
\draw [shift={(1.6,1.6)}] plot[domain=3.60524026259:4.71238898038,variable=\t]({-1.*0.4472135955*cos(\t r)+0.*0.4472135955*sin(\t r)},{0.*0.4472135955*cos(\t r)+1.*0.4472135955*sin(\t r)});
\draw [shift={(2.7,1.9)}] plot[domain=3.92699081699:5.49778714378,variable=\t]({-1.*0.707106781187*cos(\t r)+0.*0.707106781187*sin(\t r)},{0.*0.707106781187*cos(\t r)+1.*0.707106781187*sin(\t r)});
\draw (2.2,2.4)-- (2.,2.6);
\draw (2.2,1.4)-- (2.,1.6);
\draw (2.2,3.4)-- (2.,3.6);
\draw (2.,3.2)-- (2.,2.8);
\draw [shift={(5.8,3.)}] plot[domain=-0.785398163397:0.785398163397,variable=\t]({1.*0.565685424949*cos(\t r)+0.*0.565685424949*sin(\t r)},{0.*0.565685424949*cos(\t r)+1.*0.565685424949*sin(\t r)});
\draw [color=ffqqqq] (6.6,3.5)-- (6.2,3.5);
\draw [shift={(6.16,3.02)},color=ffqqqq]  plot[domain=1.79607310601:4.38426908537,variable=\t]({1.*0.49244289009*cos(\t r)+0.*0.49244289009*sin(\t r)},{0.*0.49244289009*cos(\t r)+1.*0.49244289009*sin(\t r)});
\draw [color=ffqqqq] (6.6,2.45)-- (6.3,2.45);
\draw [shift={(6.3,2.85)},color=ffqqqq]  plot[domain=4.46117647478:4.71238898038,variable=\t]({1.*0.400344898933*cos(\t r)+0.*0.400344898933*sin(\t r)},{0.*0.400344898933*cos(\t r)+1.*0.400344898933*sin(\t r)});
\draw [color=qqqqff] (6.6,2.55)-- (6.25,2.55);
\draw [shift={(6.25,2.25)},color=qqqqff]  plot[domain=1.57079632679:1.95130270391,variable=\t]({1.*0.3*cos(\t r)+0.*0.3*sin(\t r)},{0.*0.3*cos(\t r)+1.*0.3*sin(\t r)});
\draw [shift={(6.15,2.)},color=qqqqff]  plot[domain=1.76819188664:4.51499342053,variable=\t]({1.*0.509901951359*cos(\t r)+0.*0.509901951359*sin(\t r)},{0.*0.509901951359*cos(\t r)+1.*0.509901951359*sin(\t r)});
\draw [color=qqqqff] (6.6,1.5)-- (6.2,1.5);
\draw [shift={(6.8,2.)}] plot[domain=2.35619449019:3.92699081699,variable=\t]({-1.*0.565685424949*cos(\t r)+0.*0.565685424949*sin(\t r)},{0.*0.565685424949*cos(\t r)+1.*0.565685424949*sin(\t r)});
\draw [shift={(7.4,2.)}] plot[domain=-0.785398163397:0.785398163397,variable=\t]({-1.*0.565685424949*cos(\t r)+0.*0.565685424949*sin(\t r)},{0.*0.565685424949*cos(\t r)+1.*0.565685424949*sin(\t r)});
\draw [shift={(6.8,3.)}] plot[domain=2.35619449019:3.92699081699,variable=\t]({-1.*0.565685424949*cos(\t r)+0.*0.565685424949*sin(\t r)},{0.*0.565685424949*cos(\t r)+1.*0.565685424949*sin(\t r)});
\draw [shift={(7.6,3.)}] plot[domain=0.982793723247:2.55359005004,variable=\t]({-1.*0.721110255093*cos(\t r)+0.*0.721110255093*sin(\t r)},{0.*0.721110255093*cos(\t r)+1.*0.721110255093*sin(\t r)});
\draw [shift={(6.6,2.4)}] plot[domain=2.58299333825:3.70019196893,variable=\t]({-1.*1.88679622641*cos(\t r)+0.*1.88679622641*sin(\t r)},{0.*1.88679622641*cos(\t r)+1.*1.88679622641*sin(\t r)});
\draw [shift={(6.6,3.2)}] plot[domain=2.35619449019:3.13352128024,variable=\t]({-0.707106781187*0.565685424949*cos(\t r)+-0.707106781187*0.565685424949*sin(\t r)},{-0.707106781187*0.565685424949*cos(\t r)+0.707106781187*0.565685424949*sin(\t r)});
\draw [shift={(7.4,3.)}] plot[domain=0.361367123907:0.785398163397,variable=\t]({-1.*0.565685424949*cos(\t r)+0.*0.565685424949*sin(\t r)},{0.*0.565685424949*cos(\t r)+1.*0.565685424949*sin(\t r)});
\draw [shift={(7.4,3.)}] plot[domain=5.49778714378:5.92181818327,variable=\t]({-1.*0.565685424949*cos(\t r)+0.*0.565685424949*sin(\t r)},{0.*0.565685424949*cos(\t r)+1.*0.565685424949*sin(\t r)});
\draw [shift={(6.6,1.6)}] plot[domain=3.60524026259:4.71238898038,variable=\t]({-1.*0.4472135955*cos(\t r)+0.*0.4472135955*sin(\t r)},{0.*0.4472135955*cos(\t r)+1.*0.4472135955*sin(\t r)});
\draw [shift={(7.7,1.9)}] plot[domain=3.92699081699:5.49778714378,variable=\t]({-1.*0.707106781187*cos(\t r)+0.*0.707106781187*sin(\t r)},{0.*0.707106781187*cos(\t r)+1.*0.707106781187*sin(\t r)});
\draw [shift={(7.4,3.)}] plot[domain=-0.785398163397:0.785398163397,variable=\t]({-1.*0.565685424949*cos(\t r)+0.*0.565685424949*sin(\t r)},{0.*0.565685424949*cos(\t r)+1.*0.565685424949*sin(\t r)});
\draw [shift={(7.04,3.02)},color=ffqqqq]  plot[domain=1.79607310601:4.38426908537,variable=\t]({-1.*0.49244289009*cos(\t r)+0.*0.49244289009*sin(\t r)},{0.*0.49244289009*cos(\t r)+1.*0.49244289009*sin(\t r)});
\draw [shift={(6.9,2.85)},color=ffqqqq]  plot[domain=4.46117647478:4.71238898038,variable=\t]({-1.*0.400344898933*cos(\t r)+0.*0.400344898933*sin(\t r)},{0.*0.400344898933*cos(\t r)+1.*0.400344898933*sin(\t r)});
\draw [shift={(6.95,2.25)},color=qqqqff]  plot[domain=1.57079632679:1.95130270391,variable=\t]({-1.*0.3*cos(\t r)+0.*0.3*sin(\t r)},{0.*0.3*cos(\t r)+1.*0.3*sin(\t r)});
\draw [shift={(7.05,2.)},color=qqqqff]  plot[domain=1.76819188664:4.51499342053,variable=\t]({-1.*0.509901951359*cos(\t r)+0.*0.509901951359*sin(\t r)},{0.*0.509901951359*cos(\t r)+1.*0.509901951359*sin(\t r)});
\draw (7.2,2.4)-- (7.,2.6);
\draw (7.2,1.4)-- (7.,1.6);
\draw (7.2,3.4)-- (7.,3.6);
\draw [color=ffqqqq] (6.6,3.5)-- (7.,3.5);
\draw [color=ffqqqq] (6.6,2.45)-- (6.9,2.45);
\draw [color=qqqqff] (6.6,2.55)-- (6.95,2.55);
\draw [color=qqqqff] (6.6,1.5)-- (7.,1.5);
\draw (1.35,3.2) node[anchor=north west] {$n$};
\draw [color=ffqqqq](7.6,3.28537649486) node[anchor=north west] {$c$};
\draw [color=qqqqff](7.6,2.1307897428) node[anchor=north west] {$d$};
\end{tikzpicture}

\caption{Stallings knot $K_n$}
\label{fig:Kn}
\end{center}
\end{figure}
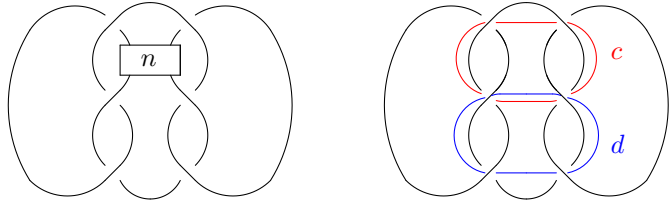

\subsection{Main Construction}
 Even though monodromy group is a very rough invariant for Lefschetz fibrations, it sometimes gives enough information to distinguish some pairs of Lefschetz fibration structures. For example, we were successful to distinguish some Lefschetz fibration structures on $E(2)_{K}$ for a family of  Kanenobu knots up to parity~\cite{Park-Yun:2011}. The main tool used in the proof was Humphries' graph method. In this paper we use Humphries' graph method again to show that the corresponding monodromy groups appeared in the main theorem are mutually distinct. For this, we start with the following lemma.

\vspace{-.5 em}

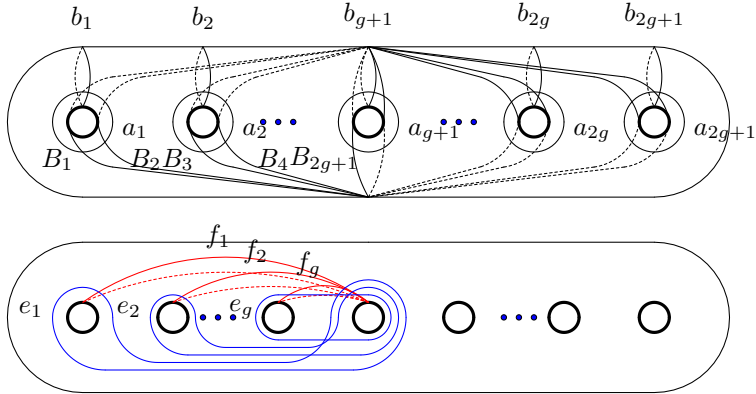
\begin{figure}[htbp]
\begin{center}
\definecolor{qqqqff}{rgb}{0.,0.,1.}
\definecolor{ffqqqq}{rgb}{1.,0.,0.}
\begin{tikzpicture}[line cap=round,line join=round,>=triangle 45,x=1.0cm,y=1.0cm]
\clip(0.0,-2.05) rectangle (11.0,3.2);
\draw [shift={(1.2,1.6)}] plot[domain=0.:3.14159265359,variable=\t]({1.*0.4*cos(\t r)+0.*0.4*sin(\t r)},{0.*0.4*cos(\t r)+1.*0.4*sin(\t r)});
\draw [shift={(1.2,1.6)}] plot[domain=-3.14159265359:0.,variable=\t]({1.*0.4*cos(\t r)+0.*0.4*sin(\t r)},{0.*0.4*cos(\t r)+1.*0.4*sin(\t r)});
\draw [shift={(2.8,1.6)}] plot[domain=0.:3.14159265359,variable=\t]({1.*0.4*cos(\t r)+0.*0.4*sin(\t r)},{0.*0.4*cos(\t r)+1.*0.4*sin(\t r)});
\draw [shift={(2.8,1.6)}] plot[domain=-3.14159265359:0.,variable=\t]({1.*0.4*cos(\t r)+0.*0.4*sin(\t r)},{0.*0.4*cos(\t r)+1.*0.4*sin(\t r)});
\draw [shift={(5.,1.6)}] plot[domain=0.:3.14159265359,variable=\t]({1.*0.4*cos(\t r)+0.*0.4*sin(\t r)},{0.*0.4*cos(\t r)+1.*0.4*sin(\t r)});
\draw [shift={(5.,1.6)}] plot[domain=-3.14159265359:0.,variable=\t]({1.*0.4*cos(\t r)+0.*0.4*sin(\t r)},{0.*0.4*cos(\t r)+1.*0.4*sin(\t r)});
\draw [shift={(1.2,1.6)},line width=1.2pt]  plot[domain=0.:3.14159265359,variable=\t]({1.*0.2*cos(\t r)+0.*0.2*sin(\t r)},{0.*0.2*cos(\t r)+1.*0.2*sin(\t r)});
\draw [shift={(1.2,1.6)},line width=1.2pt]  plot[domain=-3.14159265359:0.,variable=\t]({1.*0.2*cos(\t r)+0.*0.2*sin(\t r)},{0.*0.2*cos(\t r)+1.*0.2*sin(\t r)});
\draw [shift={(2.8,1.6)},line width=1.2pt]  plot[domain=0.:3.14159265359,variable=\t]({1.*0.2*cos(\t r)+0.*0.2*sin(\t r)},{0.*0.2*cos(\t r)+1.*0.2*sin(\t r)});
\draw [shift={(2.8,1.6)},line width=1.2pt]  plot[domain=-3.14159265359:0.,variable=\t]({1.*0.2*cos(\t r)+0.*0.2*sin(\t r)},{0.*0.2*cos(\t r)+1.*0.2*sin(\t r)});
\draw [shift={(5.,1.6)},line width=1.2pt]  plot[domain=0.:3.14159265359,variable=\t]({1.*0.2*cos(\t r)+0.*0.2*sin(\t r)},{0.*0.2*cos(\t r)+1.*0.2*sin(\t r)});
\draw [shift={(5.,1.6)},line width=1.2pt]  plot[domain=-3.14159265359:0.,variable=\t]({1.*0.2*cos(\t r)+0.*0.2*sin(\t r)},{0.*0.2*cos(\t r)+1.*0.2*sin(\t r)});
\draw [shift={(1.2,1.6)}] plot[domain=1.57079632679:4.71238898038,variable=\t]({1.*1.*cos(\t r)+0.*1.*sin(\t r)},{0.*1.*cos(\t r)+1.*1.*sin(\t r)});
\draw (1.2,2.6)-- (5.,2.6);
\draw (1.2,0.6)-- (5.,0.6);
\draw [shift={(0.45,2.2)}] plot[domain=-0.489957326254:0.489957326254,variable=\t]({1.*0.85*cos(\t r)+0.*0.85*sin(\t r)},{0.*0.85*cos(\t r)+1.*0.85*sin(\t r)});
\draw [shift={(2.05,2.2)}] plot[domain=-0.489957326254:0.489957326254,variable=\t]({1.*0.85*cos(\t r)+0.*0.85*sin(\t r)},{0.*0.85*cos(\t r)+1.*0.85*sin(\t r)});
\draw [shift={(1.95,2.2)},dash pattern=on 1pt off 1pt]  plot[domain=2.65163532734:3.63154997984,variable=\t]({1.*0.85*cos(\t r)+0.*0.85*sin(\t r)},{0.*0.85*cos(\t r)+1.*0.85*sin(\t r)});
\draw [shift={(3.55,2.2)},dash pattern=on 1pt off 1pt]  plot[domain=2.65163532734:3.63154997984,variable=\t]({1.*0.85*cos(\t r)+0.*0.85*sin(\t r)},{0.*0.85*cos(\t r)+1.*0.85*sin(\t r)});
\draw [shift={(1.7,1.7)}] plot[domain=3.28348970819:4.57049192578,variable=\t]({1.*0.707106781187*cos(\t r)+0.*0.707106781187*sin(\t r)},{0.*0.707106781187*cos(\t r)+1.*0.707106781187*sin(\t r)});
\draw [shift={(2.1,1.7)}] plot[domain=3.28348970819:4.57049192578,variable=\t]({1.*0.707106781187*cos(\t r)+0.*0.707106781187*sin(\t r)},{0.*0.707106781187*cos(\t r)+1.*0.707106781187*sin(\t r)});
\draw [shift={(3.3,1.7)}] plot[domain=3.28348970819:4.57049192578,variable=\t]({1.*0.707106781187*cos(\t r)+0.*0.707106781187*sin(\t r)},{0.*0.707106781187*cos(\t r)+1.*0.707106781187*sin(\t r)});
\draw [shift={(3.75487086278,1.72743543139)}] plot[domain=3.30883332614:4.37667530848,variable=\t]({1.*0.765551963384*cos(\t r)+0.*0.765551963384*sin(\t r)},{0.*0.765551963384*cos(\t r)+1.*0.765551963384*sin(\t r)});
\draw [shift={(6.4,1.4)}] plot[domain=3.01723765904:3.66073876784,variable=\t]({1.*1.61245154966*cos(\t r)+0.*1.61245154966*sin(\t r)},{0.*1.61245154966*cos(\t r)+1.*1.61245154966*sin(\t r)});
\draw (3.50266506208,1.00462007271)-- (5.,0.6);
\draw (3.2,1.)-- (5.,0.6);
\draw (2.,1.)-- (5.,0.6);
\draw (1.6,1.)-- (5.,0.6);
\draw [shift={(1.7,1.5)},dash pattern=on 1pt off 1pt]  plot[domain=1.7126933814:2.99969559899,variable=\t]({1.*0.707106781187*cos(\t r)+0.*0.707106781187*sin(\t r)},{0.*0.707106781187*cos(\t r)+1.*0.707106781187*sin(\t r)});
\draw [dash pattern=on 1pt off 1pt] (1.6,2.2)-- (5.,2.6);
\draw [shift={(2.1,1.5)},dash pattern=on 1pt off 1pt]  plot[domain=1.7126933814:2.99969559899,variable=\t]({1.*0.707106781187*cos(\t r)+0.*0.707106781187*sin(\t r)},{0.*0.707106781187*cos(\t r)+1.*0.707106781187*sin(\t r)});
\draw [dash pattern=on 1pt off 1pt] (2.,2.2)-- (5.,2.6);
\draw [shift={(3.3,1.5)},dash pattern=on 1pt off 1pt]  plot[domain=1.7126933814:2.99969559899,variable=\t]({1.*0.707106781187*cos(\t r)+0.*0.707106781187*sin(\t r)},{0.*0.707106781187*cos(\t r)+1.*0.707106781187*sin(\t r)});
\draw [dash pattern=on 1pt off 1pt] (3.2,2.2)-- (5.,2.6);
\draw [shift={(3.75487086278,1.47256456861)},dash pattern=on 1pt off 1pt]  plot[domain=1.9065099987:2.97435198104,variable=\t]({1.*0.765551963384*cos(\t r)+0.*0.765551963384*sin(\t r)},{0.*0.765551963384*cos(\t r)+1.*0.765551963384*sin(\t r)});
\draw [dash pattern=on 1pt off 1pt] (3.50266506208,2.19537992729)-- (5.,2.6);
\draw [shift={(6.4,1.8)},dash pattern=on 1pt off 1pt]  plot[domain=2.62244653934:3.26594764814,variable=\t]({1.*1.61245154966*cos(\t r)+0.*1.61245154966*sin(\t r)},{0.*1.61245154966*cos(\t r)+1.*1.61245154966*sin(\t r)});
\draw [shift={(8.8,1.6)}] plot[domain=-3.14159265359:0.,variable=\t]({1.*0.4*cos(\t r)+0.*0.4*sin(\t r)},{0.*0.4*cos(\t r)+1.*0.4*sin(\t r)});
\draw [shift={(8.8,1.6)}] plot[domain=0.:3.14159265359,variable=\t]({1.*0.4*cos(\t r)+0.*0.4*sin(\t r)},{0.*0.4*cos(\t r)+1.*0.4*sin(\t r)});
\draw [shift={(7.2,1.6)}] plot[domain=-3.14159265359:0.,variable=\t]({1.*0.4*cos(\t r)+0.*0.4*sin(\t r)},{0.*0.4*cos(\t r)+1.*0.4*sin(\t r)});
\draw [shift={(7.2,1.6)}] plot[domain=0.:3.14159265359,variable=\t]({1.*0.4*cos(\t r)+0.*0.4*sin(\t r)},{0.*0.4*cos(\t r)+1.*0.4*sin(\t r)});
\draw [shift={(8.8,1.6)},line width=1.2pt]  plot[domain=-3.14159265359:0.,variable=\t]({1.*0.2*cos(\t r)+0.*0.2*sin(\t r)},{0.*0.2*cos(\t r)+1.*0.2*sin(\t r)});
\draw [shift={(8.8,1.6)},line width=1.2pt]  plot[domain=0.:3.14159265359,variable=\t]({1.*0.2*cos(\t r)+0.*0.2*sin(\t r)},{0.*0.2*cos(\t r)+1.*0.2*sin(\t r)});
\draw [shift={(7.2,1.6)},line width=1.2pt]  plot[domain=-3.14159265359:0.,variable=\t]({1.*0.2*cos(\t r)+0.*0.2*sin(\t r)},{0.*0.2*cos(\t r)+1.*0.2*sin(\t r)});
\draw [shift={(7.2,1.6)},line width=1.2pt]  plot[domain=0.:3.14159265359,variable=\t]({1.*0.2*cos(\t r)+0.*0.2*sin(\t r)},{0.*0.2*cos(\t r)+1.*0.2*sin(\t r)});
\draw [shift={(8.8,1.6)}] plot[domain=-1.57079632679:1.57079632679,variable=\t]({1.*1.*cos(\t r)+0.*1.*sin(\t r)},{0.*1.*cos(\t r)+1.*1.*sin(\t r)});
\draw [shift={(8.3,1.5)}] plot[domain=0.141897054604:1.42889927219,variable=\t]({1.*0.707106781187*cos(\t r)+0.*0.707106781187*sin(\t r)},{0.*0.707106781187*cos(\t r)+1.*0.707106781187*sin(\t r)});
\draw [shift={(7.9,1.5)}] plot[domain=0.141897054604:1.42889927219,variable=\t]({1.*0.707106781187*cos(\t r)+0.*0.707106781187*sin(\t r)},{0.*0.707106781187*cos(\t r)+1.*0.707106781187*sin(\t r)});
\draw [shift={(6.7,1.5)}] plot[domain=0.141897054604:1.42889927219,variable=\t]({1.*0.707106781187*cos(\t r)+0.*0.707106781187*sin(\t r)},{0.*0.707106781187*cos(\t r)+1.*0.707106781187*sin(\t r)});
\draw [shift={(6.24512913722,1.47256456861)}] plot[domain=0.16724067255:1.23508265489,variable=\t]({1.*0.765551963384*cos(\t r)+0.*0.765551963384*sin(\t r)},{0.*0.765551963384*cos(\t r)+1.*0.765551963384*sin(\t r)});
\draw [shift={(3.6,1.8)}] plot[domain=-0.124354994547:0.519146114247,variable=\t]({1.*1.61245154966*cos(\t r)+0.*1.61245154966*sin(\t r)},{0.*1.61245154966*cos(\t r)+1.*1.61245154966*sin(\t r)});
\draw [shift={(8.3,1.7)},dash pattern=on 1pt off 1pt]  plot[domain=4.85428603499:6.14128825258,variable=\t]({1.*0.707106781187*cos(\t r)+0.*0.707106781187*sin(\t r)},{0.*0.707106781187*cos(\t r)+1.*0.707106781187*sin(\t r)});
\draw [shift={(7.9,1.7)},dash pattern=on 1pt off 1pt]  plot[domain=4.85428603499:6.14128825258,variable=\t]({1.*0.707106781187*cos(\t r)+0.*0.707106781187*sin(\t r)},{0.*0.707106781187*cos(\t r)+1.*0.707106781187*sin(\t r)});
\draw [shift={(6.7,1.7)},dash pattern=on 1pt off 1pt]  plot[domain=4.85428603499:6.14128825258,variable=\t]({1.*0.707106781187*cos(\t r)+0.*0.707106781187*sin(\t r)},{0.*0.707106781187*cos(\t r)+1.*0.707106781187*sin(\t r)});
\draw [shift={(6.24512913722,1.72743543139)},dash pattern=on 1pt off 1pt]  plot[domain=5.04810265229:6.11594463463,variable=\t]({1.*0.765551963384*cos(\t r)+0.*0.765551963384*sin(\t r)},{0.*0.765551963384*cos(\t r)+1.*0.765551963384*sin(\t r)});
\draw [shift={(3.6,1.4)},dash pattern=on 1pt off 1pt]  plot[domain=-0.519146114247:0.124354994547,variable=\t]({1.*1.61245154966*cos(\t r)+0.*1.61245154966*sin(\t r)},{0.*1.61245154966*cos(\t r)+1.*1.61245154966*sin(\t r)});
\draw (8.8,0.6)-- (5.,0.6);
\draw (8.8,2.6)-- (5.,2.6);
\draw (6.49733493792,2.19537992729)-- (5.,2.6);
\draw (6.8,2.2)-- (5.,2.6);
\draw (8.,2.2)-- (5.,2.6);
\draw (8.4,2.2)-- (5.,2.6);
\draw [dash pattern=on 1pt off 1pt] (8.4,1.)-- (5.,0.6);
\draw [dash pattern=on 1pt off 1pt] (8.,1.)-- (5.,0.6);
\draw [dash pattern=on 1pt off 1pt] (6.8,1.)-- (5.,0.6);
\draw [dash pattern=on 1pt off 1pt] (6.49733493792,1.00462007271)-- (5.,0.6);
\draw [shift={(4.25,2.2)}] plot[domain=-0.489957326254:0.489957326254,variable=\t]({1.*0.85*cos(\t r)+0.*0.85*sin(\t r)},{0.*0.85*cos(\t r)+1.*0.85*sin(\t r)});
\draw [shift={(5.75,2.2)},dash pattern=on 1pt off 1pt]  plot[domain=2.65163532734:3.63154997984,variable=\t]({1.*0.85*cos(\t r)+0.*0.85*sin(\t r)},{0.*0.85*cos(\t r)+1.*0.85*sin(\t r)});
\draw [shift={(7.95,2.2)},dash pattern=on 1pt off 1pt]  plot[domain=2.65163532734:3.63154997984,variable=\t]({1.*0.85*cos(\t r)+0.*0.85*sin(\t r)},{0.*0.85*cos(\t r)+1.*0.85*sin(\t r)});
\draw [shift={(6.45,2.2)}] plot[domain=-0.489957326254:0.489957326254,variable=\t]({1.*0.85*cos(\t r)+0.*0.85*sin(\t r)},{0.*0.85*cos(\t r)+1.*0.85*sin(\t r)});
\draw [shift={(8.05,2.2)}] plot[domain=-0.489957326254:0.489957326254,variable=\t]({1.*0.85*cos(\t r)+0.*0.85*sin(\t r)},{0.*0.85*cos(\t r)+1.*0.85*sin(\t r)});
\draw [shift={(9.55,2.2)},dash pattern=on 1pt off 1pt]  plot[domain=2.65163532734:3.63154997984,variable=\t]({1.*0.85*cos(\t r)+0.*0.85*sin(\t r)},{0.*0.85*cos(\t r)+1.*0.85*sin(\t r)});
\draw (1.2,1.1) node[anchor=east] {$B_1$};
\draw (1.7,1.1) node[anchor= west] {$B_2$};
\draw (2.8,1.1) node[anchor=east] {$B_3$};
\draw (3.4,1.1) node[anchor=west] {$B_4$};
\draw (1.19539397879,2.71427007172) node[anchor=south] {$b_1$};
\draw (2.80390866895,2.71427007172) node[anchor=south] {$b_2$};
\draw (5.0040609463,2.71427007172) node[anchor=south] {$b_{g+1}$};
\draw (1.60214481998,1.70663730604) node[anchor=north west] {$a_1$};
\draw (3.20141517284,1.70663730604) node[anchor=north west] {$a_2$};
\draw (5.40156745019,1.70663730604) node[anchor=north west] {$a_{g+1}$};
\draw (7.20421322365,2.71427007172) node[anchor=south] {$b_{2g}$};
\draw (7.60171972754,1.70663730604) node[anchor=north west] {$a_{2g}$};
\draw (8.80348357651,2.71427007172) node[anchor=south] {$b_{2g+1}$};
\draw (9.2009900804,1.70663730604) node[anchor=north west] {$a_{2g+1}$};
\draw (5.0,1.1) node[anchor=east] {$B_{2g+1}$};
\draw [shift={(1.2,-1.)}] plot[domain=1.57079632679:4.71238898038,variable=\t]({1.*1.*cos(\t r)+0.*1.*sin(\t r)},{0.*1.*cos(\t r)+1.*1.*sin(\t r)});
\draw [shift={(1.2,-1.)},line width=1.2pt]  plot[domain=0.:pi,variable=\t]({1.*0.2*cos(\t r)+0.*0.2*sin(\t r)},{0.*0.2*cos(\t r)+1.*0.2*sin(\t r)});
\draw [shift={(1.2,-1.)},line width=1.2pt]  plot[domain=-3.14159265359:0.,variable=\t]({1.*0.2*cos(\t r)+0.*0.2*sin(\t r)},{0.*0.2*cos(\t r)+1.*0.2*sin(\t r)});
\draw [shift={(2.4,-1.)},line width=1.2pt]  plot[domain=0.:3.14159265359,variable=\t]({1.*0.2*cos(\t r)+0.*0.2*sin(\t r)},{0.*0.2*cos(\t r)+1.*0.2*sin(\t r)});
\draw [shift={(2.4,-1.)},line width=1.2pt]  plot[domain=-3.14159265359:0.,variable=\t]({1.*0.2*cos(\t r)+0.*0.2*sin(\t r)},{0.*0.2*cos(\t r)+1.*0.2*sin(\t r)});
\draw [shift={(5.,-1.)},line width=1.2pt]  plot[domain=0.:3.14159265359,variable=\t]({1.*0.2*cos(\t r)+0.*0.2*sin(\t r)},{0.*0.2*cos(\t r)+1.*0.2*sin(\t r)});
\draw [shift={(5.,-1.)},line width=1.2pt]  plot[domain=-3.14159265359:0.,variable=\t]({1.*0.2*cos(\t r)+0.*0.2*sin(\t r)},{0.*0.2*cos(\t r)+1.*0.2*sin(\t r)});
\draw [shift={(3.8,-1.)},line width=1.2pt]  plot[domain=0.:3.14159265359,variable=\t]({1.*0.2*cos(\t r)+0.*0.2*sin(\t r)},{0.*0.2*cos(\t r)+1.*0.2*sin(\t r)});
\draw [shift={(3.8,-1.)},line width=1.2pt]  plot[domain=-3.14159265359:0.,variable=\t]({1.*0.2*cos(\t r)+0.*0.2*sin(\t r)},{0.*0.2*cos(\t r)+1.*0.2*sin(\t r)});
\draw (1.2,0.)-- (5.,0.);
\draw (1.2,-2.)-- (5.,-2.);
\draw [shift={(3.1,-3.5)},color=ffqqqq]  plot[domain=0.95758897792:2.18400367567,variable=\t]({1.*3.30151480384*cos(\t r)+0.*3.30151480384*sin(\t r)},{0.*3.30151480384*cos(\t r)+1.*3.30151480384*sin(\t r)});
\draw [shift={(3.7,-2.7)},color=ffqqqq]  plot[domain=0.970746113393:2.1708465402,variable=\t]({1.*2.30217288664*cos(\t r)+0.*2.30217288664*sin(\t r)},{0.*2.30217288664*cos(\t r)+1.*2.30217288664*sin(\t r)});
\draw [shift={(4.4,-1.5)},color=ffqqqq]  plot[domain=0.862170054667:2.27942259892,variable=\t]({1.*0.921954445729*cos(\t r)+0.*0.921954445729*sin(\t r)},{0.*0.921954445729*cos(\t r)+1.*0.921954445729*sin(\t r)});
\draw [shift={(3.1,-5.1)},dash pattern=on 1pt off 1pt,color=ffqqqq]  plot[domain=1.15473182107:1.98686083252,variable=\t]({1.*4.70106370942*cos(\t r)+0.*4.70106370942*sin(\t r)},{0.*4.70106370942*cos(\t r)+1.*4.70106370942*sin(\t r)});
\draw [shift={(3.7,-4.7)},dash pattern=on 1pt off 1pt,color=ffqqqq]  plot[domain=1.2490457724:1.89254688119,variable=\t]({1.*4.11096095822*cos(\t r)+0.*4.11096095822*sin(\t r)},{0.*4.11096095822*cos(\t r)+1.*4.11096095822*sin(\t r)});
\draw [shift={(4.4,-2.69552707791)},dash pattern=on 1pt off 1pt,color=ffqqqq]  plot[domain=1.26423999689:1.8773526567,variable=\t]({1.*1.98822103979*cos(\t r)+0.*1.98822103979*sin(\t r)},{0.*1.98822103979*cos(\t r)+1.*1.98822103979*sin(\t r)});
\draw [shift={(8.8,-1.)},line width=1.2pt]  plot[domain=0.:3.14159265359,variable=\t]({1.*0.2*cos(\t r)+0.*0.2*sin(\t r)},{0.*0.2*cos(\t r)+1.*0.2*sin(\t r)});
\draw [shift={(8.8,-1.)},line width=1.2pt]  plot[domain=-3.14159265359:0.,variable=\t]({1.*0.2*cos(\t r)+0.*0.2*sin(\t r)},{0.*0.2*cos(\t r)+1.*0.2*sin(\t r)});
\draw [shift={(7.6,-1.)},line width=1.2pt]  plot[domain=0.:3.14159265359,variable=\t]({1.*0.2*cos(\t r)+0.*0.2*sin(\t r)},{0.*0.2*cos(\t r)+1.*0.2*sin(\t r)});
\draw [shift={(7.6,-1.)},line width=1.2pt]  plot[domain=-3.14159265359:0.,variable=\t]({1.*0.2*cos(\t r)+0.*0.2*sin(\t r)},{0.*0.2*cos(\t r)+1.*0.2*sin(\t r)});
\draw [shift={(6.2,-1.)},line width=1.2pt]  plot[domain=0.:3.14159265359,variable=\t]({1.*0.2*cos(\t r)+0.*0.2*sin(\t r)},{0.*0.2*cos(\t r)+1.*0.2*sin(\t r)});
\draw [shift={(6.2,-1.)},line width=1.2pt]  plot[domain=-3.14159265359:0.,variable=\t]({1.*0.2*cos(\t r)+0.*0.2*sin(\t r)},{0.*0.2*cos(\t r)+1.*0.2*sin(\t r)});
\draw (8.8,0.)-- (5.,0.);
\draw (8.8,-2.)-- (5.,-2.);
\draw [shift={(8.8,-1.)}] plot[domain=-1.57079632679:1.57079632679,variable=\t]({1.*1.*cos(\t r)+0.*1.*sin(\t r)},{0.*1.*cos(\t r)+1.*1.*sin(\t r)});
\draw [shift={(3.8,-1.)},color=qqqqff]  plot[domain=1.57079632679:4.71238898038,variable=\t]({1.*0.3*cos(\t r)+0.*0.3*sin(\t r)},{0.*0.3*cos(\t r)+1.*0.3*sin(\t r)});
\draw [shift={(5.,-1.)},color=qqqqff]  plot[domain=-1.57079632679:1.57079632679,variable=\t]({1.*0.3*cos(\t r)+0.*0.3*sin(\t r)},{0.*0.3*cos(\t r)+1.*0.3*sin(\t r)});
\draw [color=qqqqff] (3.8,-0.7)-- (5.,-0.7);
\draw [color=qqqqff] (3.8,-1.3)-- (5.,-1.3);
\draw [shift={(2.4,-1.)},color=qqqqff]  plot[domain=0.:pi,variable=\t]({1.*0.3*cos(\t r)+0.*0.3*sin(\t r)},{0.*0.3*cos(\t r)+1.*0.3*sin(\t r)});
\draw [shift={(3.1,-1.)},color=qqqqff]  plot[domain=pi:4.71238898038,variable=\t]({1.*0.4*cos(\t r)+0.*0.4*sin(\t r)},{0.*0.4*cos(\t r)+1.*0.4*sin(\t r)});
\draw [shift={(2.6,-1.)},color=qqqqff]  plot[domain=pi:4.71238898038,variable=\t]({1.*0.5*cos(\t r)+0.*0.5*sin(\t r)},{0.*0.5*cos(\t r)+1.*0.5*sin(\t r)});
\draw [shift={(5.,-1.)},color=qqqqff]  plot[domain=0.:3.14159265359,variable=\t]({1.*0.4*cos(\t r)+0.*0.4*sin(\t r)},{0.*0.4*cos(\t r)+1.*0.4*sin(\t r)});
\draw [shift={(5.,-1.)},color=qqqqff]  plot[domain=0.:pi,variable=\t]({1.*0.5*cos(\t r)+0.*0.5*sin(\t r)},{0.*0.5*cos(\t r)+1.*0.5*sin(\t r)});
\draw [shift={(4.2,-1.)},color=qqqqff]  plot[domain=-1.57079632679:0.,variable=\t]({1.*0.4*cos(\t r)+0.*0.4*sin(\t r)},{0.*0.4*cos(\t r)+1.*0.4*sin(\t r)});
\draw [shift={(4.9,-1.)},color=qqqqff]  plot[domain=-1.57079632679:0.,variable=\t]({1.*0.5*cos(\t r)+0.*0.5*sin(\t r)},{0.*0.5*cos(\t r)+1.*0.5*sin(\t r)});
\draw [shift={(3.9,-1.)},color=qqqqff]  plot[domain=-1.57079632679:0.,variable=\t]({1.*0.6*cos(\t r)+0.*0.6*sin(\t r)},{0.*0.6*cos(\t r)+1.*0.6*sin(\t r)});
\draw [shift={(4.8,-1.)},color=qqqqff]  plot[domain=-1.57079632679:0.,variable=\t]({1.*0.7*cos(\t r)+0.*0.7*sin(\t r)},{0.*0.7*cos(\t r)+1.*0.7*sin(\t r)});
\draw [color=qqqqff] (3.1,-1.4)-- (4.2,-1.4);
\draw [color=qqqqff] (2.6,-1.5)-- (4.9,-1.5);
\draw [shift={(1.2,-1.)},color=qqqqff]  plot[domain=0.:pi,variable=\t]({1.*0.4*cos(\t r)+0.*0.4*sin(\t r)},{0.*0.4*cos(\t r)+1.*0.4*sin(\t r)});
\draw [shift={(2.2,-1.)},color=qqqqff]  plot[domain=pi:4.71238898038,variable=\t]({1.*0.6*cos(\t r)+0.*0.6*sin(\t r)},{0.*0.6*cos(\t r)+1.*0.6*sin(\t r)});
\draw [color=qqqqff] (2.2,-1.6)-- (3.9,-1.6);
\draw [shift={(1.5,-1.)},color=qqqqff]  plot[domain=pi:4.71238898038,variable=\t]({1.*0.7*cos(\t r)+0.*0.7*sin(\t r)},{0.*0.7*cos(\t r)+1.*0.7*sin(\t r)});
\draw [color=qqqqff] (1.5,-1.7)-- (4.8,-1.7);
\draw (0.797887474898,-0.89) node[anchor=east] {$e_1$};
\draw (2.10133903417,-0.891021475194) node[anchor=east] {$e_2$};
\draw (3.6,-0.891021475194) node[anchor=east] {$e_g$};
\draw (3.0,-0.2) node[anchor=south] {$f_1$};
\draw (3.5,-0.4) node[anchor=south] {$f_2$};
\draw (4.2,-0.6) node[anchor=south] {$f_g$};
\begin{scriptsize}
\draw [fill=qqqqff] (3.6,1.6) circle (1pt);
\draw [fill=qqqqff] (3.8,1.6) circle (1pt);
\draw [fill=qqqqff] (4.,1.6) circle (1pt);
\draw [fill=qqqqff] (6.,1.6) circle (1pt);
\draw [fill=qqqqff] (6.2,1.6) circle (1pt);
\draw [fill=qqqqff] (6.4,1.6) circle (1pt);
\draw [fill=qqqqff] (2.8,-1.) circle (1pt);
\draw [fill=qqqqff] (3.,-1.) circle (1pt);
\draw [fill=qqqqff] (3.2,-1.) circle (1pt);
\draw [fill=qqqqff] (6.8,-1.) circle (1pt);
\draw [fill=qqqqff] (7.,-1.) circle (1pt);
\draw [fill=qqqqff] (7.2,-1.) circle (1pt);
\end{scriptsize}
\end{tikzpicture}
\caption{Homology basis}
\label{fig:basis}
\end{center}
\end{figure}

\begin{lemma}\label{lemma:basis}
Let $\mathcal{B}$ be a subset of $H_1(\Sigma_{2g+1};\mathbb{Z}_2)$ such that 
\begin{itemize}
\item $B_1, B_2, \cdots, B_{2g}, a_{g+1}, b_{g+1} \in \mathcal{B}$, 
      and
\item for $1 \leq i \leq g$, one of $\{ a_i, e_i\}$ and one of $\{b_i, f_i\}$ 
are in $\mathcal{B}$, where $e_i$ is a simple closed curve representing 
$a_i + a_{g+1}$ and $f_i$ is a simple closed curve representing $b_i + b_{g+1}$ as in Figure~\ref{fig:basis}.
\end{itemize}
Then $\mathcal{B}$ is a basis of $H_1(\Sigma_{2g+1};\mathbb{Z}_2)$.
\end{lemma}

\begin{proof}
For $1 \leq i \leq g$, we have $a_i \in \mathcal{B}$ or 
$a_i = e_i - a_{g+1} \in \text{Span}(\mathcal{B})$ and 
$b_i \in \mathcal{B}$ or $b_i = f_i - b_{g+1} \in \text{Span}(\mathcal{B})$.
In $H_1(\Sigma_{2g+1};\mathbb{Z}_2)$, we also get
\begin{align*}
B_{2i-1} &= B_{2i} + a_i + a_{2g+2-i} \\
B_{2i-1} &= a_{i} + a_{i+1} + \cdots + a_{2g+2-i} + b_i + b_{2g+2-i} 
\end{align*}
for $1 \leq i \leq g$.
Therefore 
\[
a_1, b_1, a_2, b_2, \cdots, a_{2g+1}, b_{2g+1} \in \text{Span}(\mathcal{B}).
\]
Since $|\mathcal{B}| = 4g+2 = \mathrm{dim} H_1(\Sigma_{2g+1};\mathbb{Z}_2)$, 
$\mathcal{B}$ is a basis of $H_1(\Sigma_{2g+1};\mathbb{Z}_2)$.
\end{proof}

\begin{figure}[htbp]
\begin{center}
\definecolor{qqqqff}{rgb}{0.,0.,1.}
\definecolor{ffqqqq}{rgb}{1.,0.,0.}
\begin{tikzpicture}[line cap=round,line join=round,>=triangle 45,x=0.25cm,y=0.25cm]
\clip(-3.21305268906,-16.5) rectangle (26.944448688,3.3);
\draw [shift={(7.,-2.)},line width=1.2pt]  plot[domain=0.:pi,variable=\t]({1.*1.*cos(\t r)+0.*1.*sin(\t r)},{0.*1.*cos(\t r)+1.*1.*sin(\t r)});
\draw [shift={(5.,0.5)}] plot[domain=-0.643501108793:0.643501108793,variable=\t]({1.*2.5*cos(\t r)+0.*2.5*sin(\t r)},{0.*2.5*cos(\t r)+1.*2.5*sin(\t r)});
\draw [shift={(9.,0.5)},dash pattern=on 2pt off 2pt]  plot[domain=2.4980915448:3.78509376238,variable=\t]({1.*2.5*cos(\t r)+0.*2.5*sin(\t r)},{0.*2.5*cos(\t r)+1.*2.5*sin(\t r)});
\draw [shift={(7.,-2.)},line width=1.2pt]  plot[domain=-3.14159265359:0.,variable=\t]({1.*1.*cos(\t r)+0.*1.*sin(\t r)},{0.*1.*cos(\t r)+1.*1.*sin(\t r)});
\draw [shift={(7.,-2.)}] plot[domain=0.:pi,variable=\t]({1.*2.*cos(\t r)+0.*2.*sin(\t r)},{0.*2.*cos(\t r)+1.*2.*sin(\t r)});
\draw [shift={(7.,-2.)}] plot[domain=-3.14159265359:0.,variable=\t]({1.*2.*cos(\t r)+0.*2.*sin(\t r)},{0.*2.*cos(\t r)+1.*2.*sin(\t r)});
\draw [shift={(2.,-2.)},line width=1.2pt]  plot[domain=0.:pi,variable=\t]({1.*1.*cos(\t r)+0.*1.*sin(\t r)},{0.*1.*cos(\t r)+1.*1.*sin(\t r)});
\draw [shift={(0.,0.5)}] plot[domain=-0.643501108793:0.643501108793,variable=\t]({1.*2.5*cos(\t r)+0.*2.5*sin(\t r)},{0.*2.5*cos(\t r)+1.*2.5*sin(\t r)});
\draw [shift={(4.,0.5)},dash pattern=on 2pt off 2pt]  plot[domain=2.4980915448:3.78509376238,variable=\t]({1.*2.5*cos(\t r)+0.*2.5*sin(\t r)},{0.*2.5*cos(\t r)+1.*2.5*sin(\t r)});
\draw [shift={(2.,-2.)},line width=1.2pt]  plot[domain=-3.14159265359:0.,variable=\t]({1.*1.*cos(\t r)+0.*1.*sin(\t r)},{0.*1.*cos(\t r)+1.*1.*sin(\t r)});
\draw [shift={(2.,-2.)}] plot[domain=0.:pi,variable=\t]({1.*2.*cos(\t r)+0.*2.*sin(\t r)},{0.*2.*cos(\t r)+1.*2.*sin(\t r)});
\draw [shift={(2.,-2.)}] plot[domain=-3.14159265359:0.,variable=\t]({1.*2.*cos(\t r)+0.*2.*sin(\t r)},{0.*2.*cos(\t r)+1.*2.*sin(\t r)});
\draw [shift={(12.,-2.)},line width=1.2pt]  plot[domain=0.:pi,variable=\t]({1.*1.*cos(\t r)+0.*1.*sin(\t r)},{0.*1.*cos(\t r)+1.*1.*sin(\t r)});
\draw [shift={(10.,0.5)}] plot[domain=-0.643501108793:0.643501108793,variable=\t]({1.*2.5*cos(\t r)+0.*2.5*sin(\t r)},{0.*2.5*cos(\t r)+1.*2.5*sin(\t r)});
\draw [shift={(14.,0.5)},dash pattern=on 2pt off 2pt]  plot[domain=2.4980915448:3.78509376238,variable=\t]({1.*2.5*cos(\t r)+0.*2.5*sin(\t r)},{0.*2.5*cos(\t r)+1.*2.5*sin(\t r)});
\draw [shift={(12.,-2.)},line width=1.2pt]  plot[domain=-3.14159265359:0.,variable=\t]({1.*1.*cos(\t r)+0.*1.*sin(\t r)},{0.*1.*cos(\t r)+1.*1.*sin(\t r)});
\draw [shift={(12.,-2.)}] plot[domain=0.:pi,variable=\t]({1.*2.*cos(\t r)+0.*2.*sin(\t r)},{0.*2.*cos(\t r)+1.*2.*sin(\t r)});
\draw [shift={(12.,-2.)}] plot[domain=-3.14159265359:0.,variable=\t]({1.*2.*cos(\t r)+0.*2.*sin(\t r)},{0.*2.*cos(\t r)+1.*2.*sin(\t r)});
\draw [shift={(17.,-2.)},line width=1.2pt]  plot[domain=0.:pi,variable=\t]({1.*1.*cos(\t r)+0.*1.*sin(\t r)},{0.*1.*cos(\t r)+1.*1.*sin(\t r)});
\draw [shift={(15.,0.5)}] plot[domain=-0.643501108793:0.643501108793,variable=\t]({1.*2.5*cos(\t r)+0.*2.5*sin(\t r)},{0.*2.5*cos(\t r)+1.*2.5*sin(\t r)});
\draw [shift={(19.,0.5)},dash pattern=on 2pt off 2pt]  plot[domain=2.4980915448:3.78509376238,variable=\t]({1.*2.5*cos(\t r)+0.*2.5*sin(\t r)},{0.*2.5*cos(\t r)+1.*2.5*sin(\t r)});
\draw [shift={(17.,-2.)},line width=1.2pt]  plot[domain=-3.14159265359:0.,variable=\t]({1.*1.*cos(\t r)+0.*1.*sin(\t r)},{0.*1.*cos(\t r)+1.*1.*sin(\t r)});
\draw [shift={(17.,-2.)}] plot[domain=0.:pi,variable=\t]({1.*2.*cos(\t r)+0.*2.*sin(\t r)},{0.*2.*cos(\t r)+1.*2.*sin(\t r)});
\draw [shift={(17.,-2.)}] plot[domain=-3.14159265359:0.,variable=\t]({1.*2.*cos(\t r)+0.*2.*sin(\t r)},{0.*2.*cos(\t r)+1.*2.*sin(\t r)});
\draw [shift={(22.,-2.)},line width=1.2pt]  plot[domain=0.:pi,variable=\t]({1.*1.*cos(\t r)+0.*1.*sin(\t r)},{0.*1.*cos(\t r)+1.*1.*sin(\t r)});
\draw [shift={(20.,0.5)}] plot[domain=-0.643501108793:0.643501108793,variable=\t]({1.*2.5*cos(\t r)+0.*2.5*sin(\t r)},{0.*2.5*cos(\t r)+1.*2.5*sin(\t r)});
\draw [shift={(24.,0.5)},dash pattern=on 2pt off 2pt]  plot[domain=2.4980915448:3.78509376238,variable=\t]({1.*2.5*cos(\t r)+0.*2.5*sin(\t r)},{0.*2.5*cos(\t r)+1.*2.5*sin(\t r)});
\draw [shift={(22.,-2.)},line width=1.2pt]  plot[domain=-3.14159265359:0.,variable=\t]({1.*1.*cos(\t r)+0.*1.*sin(\t r)},{0.*1.*cos(\t r)+1.*1.*sin(\t r)});
\draw [shift={(22.,-2.)}] plot[domain=0.:pi,variable=\t]({1.*2.*cos(\t r)+0.*2.*sin(\t r)},{0.*2.*cos(\t r)+1.*2.*sin(\t r)});
\draw [shift={(22.,-2.)}] plot[domain=-3.14159265359:0.,variable=\t]({1.*2.*cos(\t r)+0.*2.*sin(\t r)},{0.*2.*cos(\t r)+1.*2.*sin(\t r)});
\draw [shift={(4.5,-2.5989359497)}] plot[domain=0.379894704408:2.76169794918,variable=\t]({1.*1.61515456593*cos(\t r)+0.*1.61515456593*sin(\t r)},{0.*1.61515456593*cos(\t r)+1.*1.61515456593*sin(\t r)});
\draw [shift={(4.5,-1.37347612273)},dash pattern=on 2pt off 2pt]  plot[domain=3.53724909539:5.88752886538,variable=\t]({1.*1.62558671525*cos(\t r)+0.*1.62558671525*sin(\t r)},{0.*1.62558671525*cos(\t r)+1.*1.62558671525*sin(\t r)});
\draw [shift={(7.,-2.)}] plot[domain=-1.57079632679:1.57079632679,variable=\t]({1.*2.5*cos(\t r)+0.*2.5*sin(\t r)},{0.*2.5*cos(\t r)+1.*2.5*sin(\t r)});
\draw [shift={(2.,-2.)}] plot[domain=1.57079632679:4.71238898038,variable=\t]({1.*2.5*cos(\t r)+0.*2.5*sin(\t r)},{0.*2.5*cos(\t r)+1.*2.5*sin(\t r)});
\draw (2.,0.5)-- (7.,0.5);
\draw (2.,-4.5)-- (7.,-4.5);
\draw [shift={(2.,-2.)}] plot[domain=1.57079632679:4.71238898038,variable=\t]({1.*4.*cos(\t r)+0.*4.*sin(\t r)},{0.*4.*cos(\t r)+1.*4.*sin(\t r)});
\draw [shift={(22.,-2.)}] plot[domain=-1.57079632679:1.57079632679,variable=\t]({1.*4.*cos(\t r)+0.*4.*sin(\t r)},{0.*4.*cos(\t r)+1.*4.*sin(\t r)});
\draw (2.,2.)-- (22.,2.);
\draw (2.,-6.)-- (22.,-6.);
\draw (2.0200021116,3.8) node[anchor=north ] {$b_1$};
\draw (6.98331181945,3.8) node[anchor=north ] {$b_2$};
\draw (12.0005705459,3.8) node[anchor=north ] {$b_3$};
\draw (17.0178292723,3.8) node[anchor=north ] {$b_4$};
\draw (21.9811389801,3.8) node[anchor=north ] {$b_5$};
\draw (1.8,-2.7) node[anchor=north west] {$a_1$};
\draw (6.7,-2.7) node[anchor=north west] {$a_2$};
\draw (11.7,-2.7) node[anchor=north west] {$a_3$};
\draw (16.7,-2.7) node[anchor=north west] {$a_4$};
\draw (21.7,-2.7) node[anchor=north west] {$a_5$};
\draw (4.0,-2.7) node[anchor=north west] {$c_2$};
\draw (4.01611579845,2.20617102412) node[anchor=north west] {$d_1$};
\draw [shift={(2.,-12.)}] plot[domain=1.57079632679:4.71238898038,variable=\t]({1.*4.*cos(\t r)+0.*4.*sin(\t r)},{0.*4.*cos(\t r)+1.*4.*sin(\t r)});
\draw [shift={(2.,-12.)},line width=1.2pt]  plot[domain=0.:pi,variable=\t]({1.*1.*cos(\t r)+0.*1.*sin(\t r)},{0.*1.*cos(\t r)+1.*1.*sin(\t r)});
\draw [shift={(2.,-12.)},line width=1.2pt]  plot[domain=-3.14159265359:0.,variable=\t]({1.*1.*cos(\t r)+0.*1.*sin(\t r)},{0.*1.*cos(\t r)+1.*1.*sin(\t r)});
\draw [shift={(7.,-12.)},line width=1.2pt]  plot[domain=0.:pi,variable=\t]({1.*1.*cos(\t r)+0.*1.*sin(\t r)},{0.*1.*cos(\t r)+1.*1.*sin(\t r)});
\draw [shift={(7.,-12.)},line width=1.2pt]  plot[domain=-3.14159265359:0.,variable=\t]({1.*1.*cos(\t r)+0.*1.*sin(\t r)},{0.*1.*cos(\t r)+1.*1.*sin(\t r)});
\draw [shift={(12.,-12.)},line width=1.2pt]  plot[domain=0.:pi,variable=\t]({1.*1.*cos(\t r)+0.*1.*sin(\t r)},{0.*1.*cos(\t r)+1.*1.*sin(\t r)});
\draw [shift={(12.,-12.)},line width=1.2pt]  plot[domain=-3.14159265359:0.,variable=\t]({1.*1.*cos(\t r)+0.*1.*sin(\t r)},{0.*1.*cos(\t r)+1.*1.*sin(\t r)});
\draw [shift={(17.,-12.)},line width=1.2pt]  plot[domain=0.:pi,variable=\t]({1.*1.*cos(\t r)+0.*1.*sin(\t r)},{0.*1.*cos(\t r)+1.*1.*sin(\t r)});
\draw [shift={(17.,-12.)},line width=1.2pt]  plot[domain=-3.14159265359:0.,variable=\t]({1.*1.*cos(\t r)+0.*1.*sin(\t r)},{0.*1.*cos(\t r)+1.*1.*sin(\t r)});
\draw [shift={(22.,-12.)},line width=1.2pt]  plot[domain=0.:pi,variable=\t]({1.*1.*cos(\t r)+0.*1.*sin(\t r)},{0.*1.*cos(\t r)+1.*1.*sin(\t r)});
\draw [shift={(22.,-12.)},line width=1.2pt]  plot[domain=-3.14159265359:0.,variable=\t]({1.*1.*cos(\t r)+0.*1.*sin(\t r)},{0.*1.*cos(\t r)+1.*1.*sin(\t r)});
\draw [shift={(22.,-12.)}] plot[domain=-1.57079632679:1.57079632679,variable=\t]({1.*4.*cos(\t r)+0.*4.*sin(\t r)},{0.*4.*cos(\t r)+1.*4.*sin(\t r)});
\draw (2.,-8.)-- (22.,-8.);
\draw (2.,-16.)-- (22.,-16.);
\draw [shift={(7.,-16.25)}] plot[domain=0.80978357257:2.33180908102,variable=\t]({1.*7.25*cos(\t r)+0.*7.25*sin(\t r)},{0.*7.25*cos(\t r)+1.*7.25*sin(\t r)});
\draw [shift={(7.,-17.47547805)},dash pattern=on 2pt off 2pt]  plot[domain=0.91327317777:2.22831947582,variable=\t]({1.*8.18118670953*cos(\t r)+0.*8.18118670953*sin(\t r)},{0.*8.18118670953*cos(\t r)+1.*8.18118670953*sin(\t r)});
\draw [shift={(8.64412906171,-12.9234837532)},dash pattern=on 2pt off 2pt]  plot[domain=0.373584362287:2.27805096505,variable=\t]({1.*2.53040512967*cos(\t r)+0.*2.53040512967*sin(\t r)},{0.*2.53040512967*cos(\t r)+1.*2.53040512967*sin(\t r)});
\draw [shift={(8.83333333333,-12.1666666667)}] plot[domain=0.0767718912698:2.57486343607,variable=\t]({1.*2.1730674684*cos(\t r)+0.*2.1730674684*sin(\t r)},{0.*2.1730674684*cos(\t r)+1.*2.1730674684*sin(\t r)});
\draw [shift={(2.,-12.)},color=ffqqqq]  plot[domain=0.:pi,variable=\t]({1.*2.*cos(\t r)+0.*2.*sin(\t r)},{0.*2.*cos(\t r)+1.*2.*sin(\t r)});
\draw [shift={(12.,-12.)},color=qqqqff]  plot[domain=0.:pi,variable=\t]({1.*1.5*cos(\t r)+0.*1.5*sin(\t r)},{0.*1.5*cos(\t r)+1.*1.5*sin(\t r)});
\draw [shift={(7.,-12.)},color=qqqqff]  plot[domain=0.:pi,variable=\t]({1.*1.5*cos(\t r)+0.*1.5*sin(\t r)},{0.*1.5*cos(\t r)+1.*1.5*sin(\t r)});
\draw [shift={(12.,-12.)},color=ffqqqq]  plot[domain=0.:pi,variable=\t]({1.*2.*cos(\t r)+0.*2.*sin(\t r)},{0.*2.*cos(\t r)+1.*2.*sin(\t r)});
\draw [shift={(9.5,-12.)},color=qqqqff]  plot[domain=-3.14159265359:0.,variable=\t]({1.*1.*cos(\t r)+0.*1.*sin(\t r)},{0.*1.*cos(\t r)+1.*1.*sin(\t r)});
\draw [shift={(7.,-12.)},color=qqqqff]  plot[domain=pi:4.71238898038,variable=\t]({1.*1.5*cos(\t r)+0.*1.5*sin(\t r)},{0.*1.5*cos(\t r)+1.*1.5*sin(\t r)});
\draw [shift={(12.,-12.)},color=qqqqff]  plot[domain=-1.57079632679:0.,variable=\t]({1.*1.5*cos(\t r)+0.*1.5*sin(\t r)},{0.*1.5*cos(\t r)+1.*1.5*sin(\t r)});
\draw [color=qqqqff] (7.,-13.5)-- (12.,-13.5);
\draw [shift={(7.,-10.75)},color=ffqqqq]  plot[domain=3.53638377329:5.88839418748,variable=\t]({1.*3.25*cos(\t r)+0.*3.25*sin(\t r)},{0.*3.25*cos(\t r)+1.*3.25*sin(\t r)});
\draw [shift={(3.33336991808,-11.9444200546)},color=ffqqqq]  plot[domain=3.15826490927:4.300925605,variable=\t]({1.*3.33383325035*cos(\t r)+0.*3.33383325035*sin(\t r)},{0.*3.33383325035*cos(\t r)+1.*3.33383325035*sin(\t r)});
\draw [shift={(10.5067060694,-11.8378040463)},color=ffqqqq]  plot[domain=5.15357744114:6.23678796613,variable=\t]({1.*3.49705733624*cos(\t r)+0.*3.49705733624*sin(\t r)},{0.*3.49705733624*cos(\t r)+1.*3.49705733624*sin(\t r)});
\draw [shift={(7.,16.1787769456)},color=ffqqqq]  plot[domain=4.5533773666:4.87140059417,variable=\t]({1.*31.57714572*cos(\t r)+0.*31.57714572*sin(\t r)},{0.*31.57714572*cos(\t r)+1.*31.57714572*sin(\t r)});
\draw (8.00834317215,-7.5) node[anchor=north west] {$f_1$};
\draw (8.8,-9.9) node[anchor=north west] {$f_2$};
\draw (2.0200021116,-13.3311463222) node[anchor=north west] {$e_1$};
\draw (10.004456859,-13.3311463222) node[anchor=north west] {$e_2$};
\end{tikzpicture}

\caption{Simple closed curves on $\Sigma_{2g+1}$ with $g=2$}
\label{fig:SCC}
\end{center}
\end{figure}
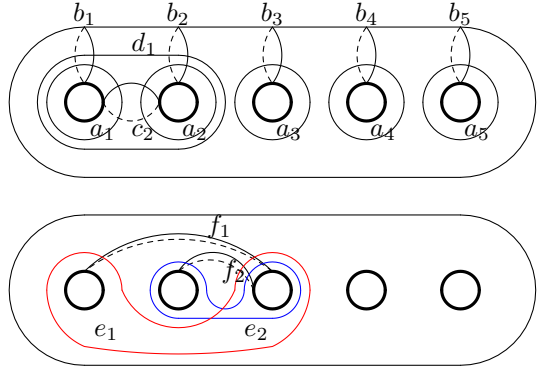

\begin{proposition}\label{proposition:step1}
For each integer $n \in \mathbb{Z}$,
a knot surgery $4$-manifold  $E(2)_{K_n}$ admits at least 
two nonisomorphic genus $5$ Lefschetz fibrations over $S^2$.
\end{proposition}

\begin{proof}
Let $(p, q) \in \mathbb{Z}^2$ and let $K_{p,q}$ be a knot obtained from $K_0$ by performing $\textrm{sign}(p)$ Stallings twists $|p|$ times along $c_2$  
and $\textrm{sign}(q)$ Stallings twists $|q|$ times along $d_1$ in Figure~\ref{fig:SCC} (corresponding to $c$ and $d$ in Figure~\ref{fig:Kn}), then we have
\[
K_{p,q} \sim K_{p+q}
\]
and 
\[
\phi_{K_{p,q}} = t_{d_1}^q \circ t_{c_2}^p \circ t_{a_2}^{-1}\circ t_{b_2}^{-1} 
\circ t_{a_1} \circ t_{b_1}.
\]
Since $K_n \sim K_{p,q}$ whenever $n=p+q$,
\[
\phi_{K_{p,q}} =  t_{d_1}^q \circ t_{c_2}^p \circ t_{a_2}^{-1}\circ t_{b_2}^{-1} 
\circ t_{a_1} \circ t_{b_1}
\]
is a monodromy of $K_n$ and 
\[
\Phi_{K_{p,q}} (\eta_{1, 2}^2) \cdot \eta_{1,2}^2
\]
is a monodromy factorization of $E(2)_{K_n}$.  Actually 
\begin{align*}
\phi_{K_{p,q}} &= t_{d_1}^q \circ t_{c_2}^{p+q} \circ t_{c_2}^{-q} \circ t_{a_2}^{-1}\circ t_{b_2}^{-1} 
\circ t_{a_1} \circ t_{b_1} \\
&= t_{d_1}^q  \circ (t_{c_2}^{p+q} \circ t_{a_2}^{-1}\circ t_{b_2}^{-1} 
\circ t_{a_1} \circ t_{b_1}) \circ t_{(t_{a_2}^{-1}\circ t_{b_2}^{-1} 
\circ t_{a_1} \circ t_{b_1})^{-1}(c_2)}^{-q} \\
&=  t_{d_1}^q \circ (t_{c_2}^n  \circ t_{a_2}^{-1}\circ t_{b_2}^{-1} 
\circ t_{a_1} \circ t_{b_1}) \circ t_{d_1}^{-q}
\end{align*}
because $t_{b_1}^{-1}\circ t_{b_2} \circ t_{a_1}^{-1} \circ t_{a_2} (c_2) = d_1$. 

Let $\varepsilon_p \equiv p \pmod{2}$ and $\varepsilon_q \equiv q \pmod{2}$  
with $\varepsilon_p, \varepsilon_q \in \{ 0, 1 \}$.
In order to use Humphries' graph method, we want to find a basis 
$\mathcal{B}_{\varepsilon_p + 2 \varepsilon_q}$ of $H_1(\Sigma_5;\mathbb{Z}_2)$ such that
\begin{itemize}
\item  $\{B_1, B_2, B_3, B_4, b_3, a_3\} \subset 
       \mathcal{B}_{\varepsilon_p + 2 \varepsilon_q}$ and
\item $G_F(\Phi_{K_{p,q}} (\eta_{1, 2}^2) \cdot \eta_{1,2}^2) 
      \leq G_{\Gamma(\mathcal{B}_{\varepsilon_p + 2 \varepsilon_q})}$.
\end{itemize}

In $H_1(\Sigma_5;\mathbb{Z}_2)$, we get 
\begin{align*}
\Phi_{K_{p,q}}(B_0) &= B_0 + b_1 + b_2 + a_1 + a_2, \\
\Phi_{K_{p,q}}(B_1) &= B_1 + b_1 + b_2 + a_2 + \varepsilon_p c_2 + \varepsilon_q d_1, \\
\Phi_{K_{p,q}}(B_2) &= B_2 + b_2 + a_1 + a_2 + \varepsilon_p c_2,  \\
\Phi_{K_{p,q}}(B_3) &= B_3 + b_2 + \varepsilon_p c_2, \\
\Phi_{K_{p,q}}(B_4) &= B_4 + a_2  + \varepsilon_p c_2 + \varepsilon_q d_1
\end{align*}
because we have, for $0 \leq i \leq 4$,
\begin{align*}
\Phi_{K_{p,q}}(B_i) &= B_i + i_2(B_i, b_1) b_1 + i_2(t_{b_1}(B_i), a_1) a_1 + i_2(t_{a_1}\circ t_{b_1}(B_i), b_2)b_2 \\
& + i_2(t_{b_2}^{-1} \circ t_{a_1} \circ t_{b_1}(B_i), a_2) a_2\\
& + \varepsilon_p i_2(t_{a_2}^{-1}\circ t_{b_2}^{-1}\circ t_{a_1} \circ t_{b_1}(B_i), c_2) c_2  \\
&+\varepsilon_q i_2(t_{c_2}^p \circ t_{a_2}^{-1}\circ t_{b_2}^{-1}\circ t_{a_1} \circ t_{b_1}(B_i), d_1) d_1.
\end{align*}

Since $\Phi_{K_{p,q}}(B_i) \in G_F(\Phi_{K_{p,q}} (\eta_{1, 2}^2) \cdot \eta_{1,2}^2)$ 
for $0 \leq i \leq 4$,  Corollary~\ref{corollary:Humphries} implies that
\begin{enumerate}[label=(\alph*)]
\item even number of $\{ b_1, b_2, a_1, a_2\}$ has $\chi_\Gamma = 0$,
\item even number of $\{b_1, b_2,  a_2,  \varepsilon_p c_2, \varepsilon_q d_1\}$ 
has $\chi_\Gamma = 0$,
\item even number of $\{b_2, a_1, a_2,  \varepsilon_p c_2\}$ has $\chi_\Gamma = 0$,
\item even number of $\{b_2,  \varepsilon_p c_2\}$ has $\chi_\Gamma = 0$,
\item even number of $\{a_2,  \varepsilon_p c_2, \varepsilon_q d_1\}$ has 
$\chi_\Gamma = 0$.
\end{enumerate}
As a convention, if $\varepsilon_p = 0$ or $\varepsilon_q = 0$, then the corresponding 
$c_2$ or $d_1$ is omitted from the set. 
We may assume that $\chi_\Gamma(c_2) = \chi_\Gamma(d_1) = 0$ for each graph 
$\Gamma(\mathcal{B}_{\varepsilon_p + 2 \varepsilon_q})$ ~\cite{Park-Yun:2011}.
Then (a)--(e) has a unique solution for each 
$(\varepsilon_p, \varepsilon_q)\in \{0, 1\} \times \{0, 1\}$ as follows:
\begin{itemize}
\item $a_1, a_2, b_1, b_2 \in G_{\Gamma(\mathcal{B}_0)}$,
\item $a_1, a_2, b_1, b_2 \not\in G_{\Gamma(\mathcal{B}_1)}$,
\item $b_1, b_2 \in G_{\Gamma(\mathcal{B}_2)}$ and 
$a_1, a_2 \not\in G_{\Gamma(\mathcal{B}_2)}$,
\item $a_1, a_2 \in G_{\Gamma(\mathcal{B}_3)}$ and 
$b_1, b_2 \not\in G_{\Gamma(\mathcal{B}_3)}$.
\end{itemize}

Let $e_i$ be a simple closed curve on $\Sigma_5$ representing 
$a_i + a_3 \in H_1(\Sigma_5; \mathbb{Z}_2)$ 
and let $f_i$ be a simple closed curve on $\Sigma_5$ representing 
$b_i + b_3 \in H_1(\Sigma_5; \mathbb{Z}_2)$ for $1 \leq i \leq 2$ as in Figure~\ref{fig:SCC}. 
Then, by Lemma~\ref{lemma:basis}, we can select a basis $\mathcal{B}_j$ as follows:
\begin{itemize}
\item $\mathcal{B}_0 = \{ B_1, B_2, B_3, B_4, b_3, a_3, a_1, a_2, b_1, b_2\}$,
\item $\mathcal{B}_1 = \{ B_1, B_2, B_3, B_4, b_3, a_3, e_1, e_2, f_1, f_2\}$,
\item $\mathcal{B}_2 = \{ B_1, B_2, B_3, B_4, b_3, a_3, e_1, e_2, b_1, b_2\}$,
\item $\mathcal{B}_3 = \{ B_1, B_2, B_3, B_4, b_3, a_3, a_1, a_2, f_1, f_2\}$.
\end{itemize}

Now we will show that 
\[
G_F(\Phi_{K_{p,q}} (\eta_{1, 2}^2) \cdot \eta_{1,2}^2) 
\leq G_{\Gamma(\mathcal{B}_{\varepsilon_p + 2 \varepsilon_q})}
\]
for each $(p,q) \in \mathbb{Z}^2$. To do this we only need to show that 
$B_0, B_5 \in G_{\Gamma(\mathcal{B}_{\varepsilon_p + 2 \varepsilon_q})}$ 
for each $(\varepsilon_p, \varepsilon_q)\in \{0, 1\} \times \{0, 1\}$. 
Since $B_5 = a_3 \in G_{\Gamma(\mathcal{B}_{\varepsilon_p + 2 \varepsilon_q})}$, 
it remains to show that 
\[
B_0 = B_1 + B_2 + B_3 + B_4 + a_3 
\in G_{\Gamma(\mathcal{B}_{\varepsilon_p + 2 \varepsilon_q})}.
\]
But it is clear because $\overline{B_0}$ is isotopic to disjoint union of five vertices
\[
\overline{B_0} \simeq B_1\cup B_2 \cup B_3 \cup B_4 \cup a_3
\]
in each graph $\Gamma(\mathcal{B}_j)$ ($0 \leq j \leq 3$) and therefore 
$\chi_{\Gamma(\mathcal{B}_j)}(B_0) = 1$.

Let us observe that
\begin{itemize}
\item $\chi_{\Gamma(\mathcal{B}_j)}(c_2) = 0 = \chi_{\Gamma(\mathcal{B}_j)}(d_1)$ 
for $0 \leq j \leq 3$
\item $i_2(\Phi_{K_{p,q}}(B_1), c_2) = 1$  and $i_2(\Phi_{K_{p,q}}(B_1), d_1) = 1$ 
for each $(p, q) \in \mathbb{Z}^2$
\item $i_2(\Phi_{K_{p,q}}(B_2), c_2) = 1$  and $i_2(\Phi_{K_{p,q}}(B_2), d_1) = 0$ 
for each $(p, q) \in \mathbb{Z}^2$.
\end{itemize}
It implies that
\begin{itemize}
\item $\Phi_{K_{p,q}}(B_1), \Phi_{K_{p,q}}(B_2) 
\in G_F(\Phi_{K_{p,q}} (\eta_{1, 2}^2) \cdot \eta_{1,2}^2) 
\leq G_{\Gamma(\mathcal{B}_{\varepsilon_p + 2 \varepsilon_q})}$
\item if $(p-r, q-s) \equiv (1,0) \text{ or } (0,1) \pmod{2}$, then
\(
\Phi_{K_{r,s}}(B_1) \not\in G_{\Gamma(\mathcal{B}_{\varepsilon_p + 2 \varepsilon_q})}
\)
\item  if $(p-r, q-s) \equiv (1,1) \pmod{2}$, then
\(
\Phi_{K_{r,s}}(B_2) \not\in G_{\Gamma(\mathcal{B}_{\varepsilon_p + 2 \varepsilon_q})}
\).
\end{itemize}
Therefore, whenever $(r, s) \not\equiv (p,q) \pmod{2}$, we have
\[
G_F(\Phi_{K_{p,q}} (\eta_{1, 2}^2) \cdot \eta_{1,2}^2) 
\ne G_F(\Phi_{K_{r,s}} (\eta_{1, 2}^2) \cdot \eta_{1,2}^2).
\]

Finally, since $(n, 0) \not\equiv (n-1, 1) \pmod{2}$ for each integer $n$, it is clear that $E(2)_{K_n}$ has at least two nonisomorphic genus $5$ 
Lefschetz fibration structures.
\end{proof}

Now we extend this result on $E(2)_K$ to the case of a family of connected sums of Stallings knots, which is following:

\begin{theorem}\label{theorem:main}
For each integer $n > 0$ and $(m_1, m_2, \cdots, m_n) \in \mathbb{Z}^n$,  a knot surgery $4$-manifold
\[
E(2)_{K_{m_1} \sharp K_{m_2} \sharp \cdots \sharp K_{m_n}}
\]
admits at least $2^n$ nonisomorphic genus $(4n+1)$  Lefschetz fibrations over $S^2$.
Here $K_{m_i} (1 \leq i \leq n)$ denotes a knot obtained by performing $|m_i|$ left/right handed full twist on a slice knot $K_0=3_1 \sharp 3^{*}_1$ as in Figure~\ref{fig:Kn}.
\end{theorem}

\begin{proof}
Let us decompose each $m_i$ as a sum of two integers $(p_i, q_i) \in\mathbb{Z}^2$ 
such that $m_i = p_i + q_i$. 
Let $\varepsilon_{p_i}=\varepsilon_{2i -1} \equiv p_i \pmod{2}$ 
and $\varepsilon_{q_i}=\varepsilon_{2i} \equiv q_i \pmod{2}$,
where $\varepsilon_{p_i}, \varepsilon_{q_i} \in \{ 0, 1\}$ for $1 \leq i \leq n$ and $\varepsilon_{i} \in \{ 0, 1\}$ for $1 \leq i \leq 2n$.
Since $K_{m_1} \sharp K_{m_2} \sharp \cdots \sharp K_{m_n}$ 
is a fibered knot of genus $2n$ and 
\[
K_{m_1} \sharp K_{m_2} \sharp \cdots \sharp K_{m_n} 
\sim K_{p_1, q_1} \sharp K_{p_2, q_2} \sharp \cdots \sharp K_{p_n,q_n}
\] 
whenever $m_i = p_i + q_i$ for $1 \leq i \leq n$, we get a monodromy map
\[
\phi_{p_1, q_1, \cdots, p_n, q_n} := 
\prod_{i=1}^n (t_{d_{2i-1}}^{q_i} \circ t_{c_{2i}}^{p_i} \circ t_{a_{2i}}^{-1} 
\circ t_{b_{2i}}^{-1} \circ t_{a_{2i-1}} \circ t_{b_{2i-1}}).
\]

\begin{figure}[htbp]
\begin{center}

\begin{tikzpicture}[line cap=round,line join=round,>=triangle 45,x=0.8cm,y=0.8cm]
\clip(0.8,-8.2) rectangle (16.6,1.27437289327);
\draw [shift={(1.4,0.2)}] plot[domain=0.:3.14159265359,variable=\t]({1.*0.2*cos(\t r)+0.*0.2*sin(\t r)},{0.*0.2*cos(\t r)+1.*0.2*sin(\t r)});
\draw [shift={(1.4,0.2)}] plot[domain=-3.14159265359:0.,variable=\t]({1.*0.2*cos(\t r)+0.*0.2*sin(\t r)},{0.*0.2*cos(\t r)+1.*0.2*sin(\t r)});
\draw [shift={(2.2,0.2)}] plot[domain=0.:pi,variable=\t]({1.*0.2*cos(\t r)+0.*0.2*sin(\t r)},{0.*0.2*cos(\t r)+1.*0.2*sin(\t r)});
\draw [shift={(2.2,0.2)}] plot[domain=-3.14159265359:0.,variable=\t]({1.*0.2*cos(\t r)+0.*0.2*sin(\t r)},{0.*0.2*cos(\t r)+1.*0.2*sin(\t r)});
\draw [shift={(3.,0.2)}] plot[domain=0.:3.14159265359,variable=\t]({1.*0.2*cos(\t r)+0.*0.2*sin(\t r)},{0.*0.2*cos(\t r)+1.*0.2*sin(\t r)});
\draw [shift={(3.,0.2)}] plot[domain=-3.14159265359:0.,variable=\t]({1.*0.2*cos(\t r)+0.*0.2*sin(\t r)},{0.*0.2*cos(\t r)+1.*0.2*sin(\t r)});
\draw [shift={(3.8,0.2)}] plot[domain=0.:3.14159265359,variable=\t]({1.*0.2*cos(\t r)+0.*0.2*sin(\t r)},{0.*0.2*cos(\t r)+1.*0.2*sin(\t r)});
\draw [shift={(3.8,0.2)}] plot[domain=-3.14159265359:0.,variable=\t]({1.*0.2*cos(\t r)+0.*0.2*sin(\t r)},{0.*0.2*cos(\t r)+1.*0.2*sin(\t r)});
\draw [shift={(4.6,0.2)}] plot[domain=0.:3.14159265359,variable=\t]({1.*0.2*cos(\t r)+0.*0.2*sin(\t r)},{0.*0.2*cos(\t r)+1.*0.2*sin(\t r)});
\draw [shift={(4.6,0.2)}] plot[domain=-3.14159265359:0.,variable=\t]({1.*0.2*cos(\t r)+0.*0.2*sin(\t r)},{0.*0.2*cos(\t r)+1.*0.2*sin(\t r)});
\draw [shift={(5.4,0.2)}] plot[domain=0.:3.14159265359,variable=\t]({1.*0.2*cos(\t r)+0.*0.2*sin(\t r)},{0.*0.2*cos(\t r)+1.*0.2*sin(\t r)});
\draw [shift={(5.4,0.2)}] plot[domain=-3.14159265359:0.,variable=\t]({1.*0.2*cos(\t r)+0.*0.2*sin(\t r)},{0.*0.2*cos(\t r)+1.*0.2*sin(\t r)});
\draw [shift={(6.2,0.2)}] plot[domain=0.:3.14159265359,variable=\t]({1.*0.2*cos(\t r)+0.*0.2*sin(\t r)},{0.*0.2*cos(\t r)+1.*0.2*sin(\t r)});
\draw [shift={(6.2,0.2)}] plot[domain=-3.14159265359:0.,variable=\t]({1.*0.2*cos(\t r)+0.*0.2*sin(\t r)},{0.*0.2*cos(\t r)+1.*0.2*sin(\t r)});
\draw [shift={(7.,0.2)}] plot[domain=0.:3.14159265359,variable=\t]({1.*0.2*cos(\t r)+0.*0.2*sin(\t r)},{0.*0.2*cos(\t r)+1.*0.2*sin(\t r)});
\draw [shift={(7.,0.2)}] plot[domain=-3.14159265359:0.,variable=\t]({1.*0.2*cos(\t r)+0.*0.2*sin(\t r)},{0.*0.2*cos(\t r)+1.*0.2*sin(\t r)});
\draw [shift={(7.8,0.2)}] plot[domain=0.:3.14159265359,variable=\t]({1.*0.2*cos(\t r)+0.*0.2*sin(\t r)},{0.*0.2*cos(\t r)+1.*0.2*sin(\t r)});
\draw [shift={(7.8,0.2)}] plot[domain=-3.14159265359:0.,variable=\t]({1.*0.2*cos(\t r)+0.*0.2*sin(\t r)},{0.*0.2*cos(\t r)+1.*0.2*sin(\t r)});
\draw [shift={(7.8,0.2)}] plot[domain=-1.57079632679:1.57079632679,variable=\t]({1.*0.6*cos(\t r)+0.*0.6*sin(\t r)},{0.*0.6*cos(\t r)+1.*0.6*sin(\t r)});
\draw [shift={(1.4,0.2)}] plot[domain=1.57079632679:4.71238898038,variable=\t]({1.*0.6*cos(\t r)+0.*0.6*sin(\t r)},{0.*0.6*cos(\t r)+1.*0.6*sin(\t r)});
\draw (1.4,0.8)-- (7.8,0.8);
\draw (1.4,-0.4)-- (7.8,-0.4);
\draw [shift={(1.4,-1.6)}] plot[domain=0.:pi,variable=\t]({1.*0.2*cos(\t r)+0.*0.2*sin(\t r)},{0.*0.2*cos(\t r)+1.*0.2*sin(\t r)});
\draw [shift={(1.4,-1.6)}] plot[domain=-3.14159265359:0.,variable=\t]({1.*0.2*cos(\t r)+0.*0.2*sin(\t r)},{0.*0.2*cos(\t r)+1.*0.2*sin(\t r)});
\draw [shift={(2.2,-1.6)}] plot[domain=0.:3.14159265359,variable=\t]({1.*0.2*cos(\t r)+0.*0.2*sin(\t r)},{0.*0.2*cos(\t r)+1.*0.2*sin(\t r)});
\draw [shift={(2.2,-1.6)}] plot[domain=-3.14159265359:0.,variable=\t]({1.*0.2*cos(\t r)+0.*0.2*sin(\t r)},{0.*0.2*cos(\t r)+1.*0.2*sin(\t r)});
\draw [shift={(3.,-1.6)}] plot[domain=0.:pi,variable=\t]({1.*0.2*cos(\t r)+0.*0.2*sin(\t r)},{0.*0.2*cos(\t r)+1.*0.2*sin(\t r)});
\draw [shift={(3.,-1.6)}] plot[domain=-3.14159265359:0.,variable=\t]({1.*0.2*cos(\t r)+0.*0.2*sin(\t r)},{0.*0.2*cos(\t r)+1.*0.2*sin(\t r)});
\draw [shift={(3.8,-1.6)}] plot[domain=0.:3.14159265359,variable=\t]({1.*0.2*cos(\t r)+0.*0.2*sin(\t r)},{0.*0.2*cos(\t r)+1.*0.2*sin(\t r)});
\draw [shift={(3.8,-1.6)}] plot[domain=-3.14159265359:0.,variable=\t]({1.*0.2*cos(\t r)+0.*0.2*sin(\t r)},{0.*0.2*cos(\t r)+1.*0.2*sin(\t r)});
\draw [shift={(4.6,-1.6)}] plot[domain=0.:3.14159265359,variable=\t]({1.*0.2*cos(\t r)+0.*0.2*sin(\t r)},{0.*0.2*cos(\t r)+1.*0.2*sin(\t r)});
\draw [shift={(4.6,-1.6)}] plot[domain=-3.14159265359:0.,variable=\t]({1.*0.2*cos(\t r)+0.*0.2*sin(\t r)},{0.*0.2*cos(\t r)+1.*0.2*sin(\t r)});
\draw [shift={(5.4,-1.6)}] plot[domain=0.:3.14159265359,variable=\t]({1.*0.2*cos(\t r)+0.*0.2*sin(\t r)},{0.*0.2*cos(\t r)+1.*0.2*sin(\t r)});
\draw [shift={(5.4,-1.6)}] plot[domain=-3.14159265359:0.,variable=\t]({1.*0.2*cos(\t r)+0.*0.2*sin(\t r)},{0.*0.2*cos(\t r)+1.*0.2*sin(\t r)});
\draw [shift={(6.2,-1.6)}] plot[domain=0.:3.14159265359,variable=\t]({1.*0.2*cos(\t r)+0.*0.2*sin(\t r)},{0.*0.2*cos(\t r)+1.*0.2*sin(\t r)});
\draw [shift={(6.2,-1.6)}] plot[domain=-3.14159265359:0.,variable=\t]({1.*0.2*cos(\t r)+0.*0.2*sin(\t r)},{0.*0.2*cos(\t r)+1.*0.2*sin(\t r)});
\draw [shift={(7.,-1.6)}] plot[domain=0.:3.14159265359,variable=\t]({1.*0.2*cos(\t r)+0.*0.2*sin(\t r)},{0.*0.2*cos(\t r)+1.*0.2*sin(\t r)});
\draw [shift={(7.,-1.6)}] plot[domain=-3.14159265359:0.,variable=\t]({1.*0.2*cos(\t r)+0.*0.2*sin(\t r)},{0.*0.2*cos(\t r)+1.*0.2*sin(\t r)});
\draw [shift={(7.8,-1.6)}] plot[domain=0.:3.14159265359,variable=\t]({1.*0.2*cos(\t r)+0.*0.2*sin(\t r)},{0.*0.2*cos(\t r)+1.*0.2*sin(\t r)});
\draw [shift={(7.8,-1.6)}] plot[domain=-3.14159265359:0.,variable=\t]({1.*0.2*cos(\t r)+0.*0.2*sin(\t r)},{0.*0.2*cos(\t r)+1.*0.2*sin(\t r)});
\draw [shift={(7.8,-1.6)}] plot[domain=-1.57079632679:1.57079632679,variable=\t]({1.*0.6*cos(\t r)+0.*0.6*sin(\t r)},{0.*0.6*cos(\t r)+1.*0.6*sin(\t r)});
\draw [shift={(1.4,-1.6)}] plot[domain=1.57079632679:4.71238898038,variable=\t]({1.*0.6*cos(\t r)+0.*0.6*sin(\t r)},{0.*0.6*cos(\t r)+1.*0.6*sin(\t r)});
\draw (1.4,-1.)-- (7.8,-1.);
\draw (1.4,-2.2)-- (7.8,-2.2);
\draw [shift={(1.4,-3.4)}] plot[domain=0.:3.14159265359,variable=\t]({1.*0.2*cos(\t r)+0.*0.2*sin(\t r)},{0.*0.2*cos(\t r)+1.*0.2*sin(\t r)});
\draw [shift={(1.4,-3.4)}] plot[domain=-3.14159265359:0.,variable=\t]({1.*0.2*cos(\t r)+0.*0.2*sin(\t r)},{0.*0.2*cos(\t r)+1.*0.2*sin(\t r)});
\draw [shift={(2.2,-3.4)}] plot[domain=0.:3.14159265359,variable=\t]({1.*0.2*cos(\t r)+0.*0.2*sin(\t r)},{0.*0.2*cos(\t r)+1.*0.2*sin(\t r)});
\draw [shift={(2.2,-3.4)}] plot[domain=-3.14159265359:0.,variable=\t]({1.*0.2*cos(\t r)+0.*0.2*sin(\t r)},{0.*0.2*cos(\t r)+1.*0.2*sin(\t r)});
\draw [shift={(3.,-3.4)}] plot[domain=0.:pi,variable=\t]({1.*0.2*cos(\t r)+0.*0.2*sin(\t r)},{0.*0.2*cos(\t r)+1.*0.2*sin(\t r)});
\draw [shift={(3.,-3.4)}] plot[domain=-3.14159265359:0.,variable=\t]({1.*0.2*cos(\t r)+0.*0.2*sin(\t r)},{0.*0.2*cos(\t r)+1.*0.2*sin(\t r)});
\draw [shift={(3.8,-3.4)}] plot[domain=0.:3.14159265359,variable=\t]({1.*0.2*cos(\t r)+0.*0.2*sin(\t r)},{0.*0.2*cos(\t r)+1.*0.2*sin(\t r)});
\draw [shift={(3.8,-3.4)}] plot[domain=-3.14159265359:0.,variable=\t]({1.*0.2*cos(\t r)+0.*0.2*sin(\t r)},{0.*0.2*cos(\t r)+1.*0.2*sin(\t r)});
\draw [shift={(4.6,-3.4)}] plot[domain=0.:3.14159265359,variable=\t]({1.*0.2*cos(\t r)+0.*0.2*sin(\t r)},{0.*0.2*cos(\t r)+1.*0.2*sin(\t r)});
\draw [shift={(4.6,-3.4)}] plot[domain=-3.14159265359:0.,variable=\t]({1.*0.2*cos(\t r)+0.*0.2*sin(\t r)},{0.*0.2*cos(\t r)+1.*0.2*sin(\t r)});
\draw [shift={(5.4,-3.4)}] plot[domain=0.:3.14159265359,variable=\t]({1.*0.2*cos(\t r)+0.*0.2*sin(\t r)},{0.*0.2*cos(\t r)+1.*0.2*sin(\t r)});
\draw [shift={(5.4,-3.4)}] plot[domain=-3.14159265359:0.,variable=\t]({1.*0.2*cos(\t r)+0.*0.2*sin(\t r)},{0.*0.2*cos(\t r)+1.*0.2*sin(\t r)});
\draw [shift={(6.2,-3.4)}] plot[domain=0.:3.14159265359,variable=\t]({1.*0.2*cos(\t r)+0.*0.2*sin(\t r)},{0.*0.2*cos(\t r)+1.*0.2*sin(\t r)});
\draw [shift={(6.2,-3.4)}] plot[domain=-3.14159265359:0.,variable=\t]({1.*0.2*cos(\t r)+0.*0.2*sin(\t r)},{0.*0.2*cos(\t r)+1.*0.2*sin(\t r)});
\draw [shift={(7.,-3.4)}] plot[domain=0.:3.14159265359,variable=\t]({1.*0.2*cos(\t r)+0.*0.2*sin(\t r)},{0.*0.2*cos(\t r)+1.*0.2*sin(\t r)});
\draw [shift={(7.,-3.4)}] plot[domain=-3.14159265359:0.,variable=\t]({1.*0.2*cos(\t r)+0.*0.2*sin(\t r)},{0.*0.2*cos(\t r)+1.*0.2*sin(\t r)});
\draw [shift={(7.8,-3.4)}] plot[domain=0.:3.14159265359,variable=\t]({1.*0.2*cos(\t r)+0.*0.2*sin(\t r)},{0.*0.2*cos(\t r)+1.*0.2*sin(\t r)});
\draw [shift={(7.8,-3.4)}] plot[domain=-3.14159265359:0.,variable=\t]({1.*0.2*cos(\t r)+0.*0.2*sin(\t r)},{0.*0.2*cos(\t r)+1.*0.2*sin(\t r)});
\draw [shift={(7.8,-3.4)}] plot[domain=-1.57079632679:1.57079632679,variable=\t]({1.*0.6*cos(\t r)+0.*0.6*sin(\t r)},{0.*0.6*cos(\t r)+1.*0.6*sin(\t r)});
\draw [shift={(1.4,-3.4)}] plot[domain=1.57079632679:4.71238898038,variable=\t]({1.*0.6*cos(\t r)+0.*0.6*sin(\t r)},{0.*0.6*cos(\t r)+1.*0.6*sin(\t r)});
\draw (1.4,-2.8)-- (7.8,-2.8);
\draw (1.4,-4.)-- (7.8,-4.);
\draw [shift={(1.4,-5.2)}] plot[domain=0.:3.14159265359,variable=\t]({1.*0.2*cos(\t r)+0.*0.2*sin(\t r)},{0.*0.2*cos(\t r)+1.*0.2*sin(\t r)});
\draw [shift={(1.4,-5.2)}] plot[domain=-3.14159265359:0.,variable=\t]({1.*0.2*cos(\t r)+0.*0.2*sin(\t r)},{0.*0.2*cos(\t r)+1.*0.2*sin(\t r)});
\draw [shift={(2.2,-5.2)}] plot[domain=0.:pi,variable=\t]({1.*0.2*cos(\t r)+0.*0.2*sin(\t r)},{0.*0.2*cos(\t r)+1.*0.2*sin(\t r)});
\draw [shift={(2.2,-5.2)}] plot[domain=-3.14159265359:0.,variable=\t]({1.*0.2*cos(\t r)+0.*0.2*sin(\t r)},{0.*0.2*cos(\t r)+1.*0.2*sin(\t r)});
\draw [shift={(3.,-5.2)}] plot[domain=0.:3.14159265359,variable=\t]({1.*0.2*cos(\t r)+0.*0.2*sin(\t r)},{0.*0.2*cos(\t r)+1.*0.2*sin(\t r)});
\draw [shift={(3.,-5.2)}] plot[domain=-3.14159265359:0.,variable=\t]({1.*0.2*cos(\t r)+0.*0.2*sin(\t r)},{0.*0.2*cos(\t r)+1.*0.2*sin(\t r)});
\draw [shift={(3.8,-5.2)}] plot[domain=0.:pi,variable=\t]({1.*0.2*cos(\t r)+0.*0.2*sin(\t r)},{0.*0.2*cos(\t r)+1.*0.2*sin(\t r)});
\draw [shift={(3.8,-5.2)}] plot[domain=-3.14159265359:0.,variable=\t]({1.*0.2*cos(\t r)+0.*0.2*sin(\t r)},{0.*0.2*cos(\t r)+1.*0.2*sin(\t r)});
\draw [shift={(4.6,-5.2)}] plot[domain=0.:pi,variable=\t]({1.*0.2*cos(\t r)+0.*0.2*sin(\t r)},{0.*0.2*cos(\t r)+1.*0.2*sin(\t r)});
\draw [shift={(4.6,-5.2)}] plot[domain=-3.14159265359:0.,variable=\t]({1.*0.2*cos(\t r)+0.*0.2*sin(\t r)},{0.*0.2*cos(\t r)+1.*0.2*sin(\t r)});
\draw [shift={(5.4,-5.2)}] plot[domain=0.:3.14159265359,variable=\t]({1.*0.2*cos(\t r)+0.*0.2*sin(\t r)},{0.*0.2*cos(\t r)+1.*0.2*sin(\t r)});
\draw [shift={(5.4,-5.2)}] plot[domain=-3.14159265359:0.,variable=\t]({1.*0.2*cos(\t r)+0.*0.2*sin(\t r)},{0.*0.2*cos(\t r)+1.*0.2*sin(\t r)});
\draw [shift={(6.2,-5.2)}] plot[domain=0.:3.14159265359,variable=\t]({1.*0.2*cos(\t r)+0.*0.2*sin(\t r)},{0.*0.2*cos(\t r)+1.*0.2*sin(\t r)});
\draw [shift={(6.2,-5.2)}] plot[domain=-3.14159265359:0.,variable=\t]({1.*0.2*cos(\t r)+0.*0.2*sin(\t r)},{0.*0.2*cos(\t r)+1.*0.2*sin(\t r)});
\draw [shift={(7.,-5.2)}] plot[domain=0.:pi,variable=\t]({1.*0.2*cos(\t r)+0.*0.2*sin(\t r)},{0.*0.2*cos(\t r)+1.*0.2*sin(\t r)});
\draw [shift={(7.,-5.2)}] plot[domain=-3.14159265359:0.,variable=\t]({1.*0.2*cos(\t r)+0.*0.2*sin(\t r)},{0.*0.2*cos(\t r)+1.*0.2*sin(\t r)});
\draw [shift={(7.8,-5.2)}] plot[domain=0.:pi,variable=\t]({1.*0.2*cos(\t r)+0.*0.2*sin(\t r)},{0.*0.2*cos(\t r)+1.*0.2*sin(\t r)});
\draw [shift={(7.8,-5.2)}] plot[domain=-3.14159265359:0.,variable=\t]({1.*0.2*cos(\t r)+0.*0.2*sin(\t r)},{0.*0.2*cos(\t r)+1.*0.2*sin(\t r)});
\draw [shift={(7.8,-5.2)}] plot[domain=-1.57079632679:1.57079632679,variable=\t]({1.*0.6*cos(\t r)+0.*0.6*sin(\t r)},{0.*0.6*cos(\t r)+1.*0.6*sin(\t r)});
\draw [shift={(1.4,-5.2)}] plot[domain=1.57079632679:4.71238898038,variable=\t]({1.*0.6*cos(\t r)+0.*0.6*sin(\t r)},{0.*0.6*cos(\t r)+1.*0.6*sin(\t r)});
\draw (1.4,-4.6)-- (7.8,-4.6);
\draw (1.4,-5.8)-- (7.8,-5.8);
\draw [shift={(9.6,0.2)}] plot[domain=0.:3.14159265359,variable=\t]({1.*0.2*cos(\t r)+0.*0.2*sin(\t r)},{0.*0.2*cos(\t r)+1.*0.2*sin(\t r)});
\draw [shift={(9.6,0.2)}] plot[domain=-3.14159265359:0.,variable=\t]({1.*0.2*cos(\t r)+0.*0.2*sin(\t r)},{0.*0.2*cos(\t r)+1.*0.2*sin(\t r)});
\draw [shift={(10.4,0.2)}] plot[domain=0.:pi,variable=\t]({1.*0.2*cos(\t r)+0.*0.2*sin(\t r)},{0.*0.2*cos(\t r)+1.*0.2*sin(\t r)});
\draw [shift={(10.4,0.2)}] plot[domain=-3.14159265359:0.,variable=\t]({1.*0.2*cos(\t r)+0.*0.2*sin(\t r)},{0.*0.2*cos(\t r)+1.*0.2*sin(\t r)});
\draw [shift={(11.2,0.2)}] plot[domain=0.:pi,variable=\t]({1.*0.2*cos(\t r)+0.*0.2*sin(\t r)},{0.*0.2*cos(\t r)+1.*0.2*sin(\t r)});
\draw [shift={(11.2,0.2)}] plot[domain=-3.14159265359:0.,variable=\t]({1.*0.2*cos(\t r)+0.*0.2*sin(\t r)},{0.*0.2*cos(\t r)+1.*0.2*sin(\t r)});
\draw [shift={(12.,0.2)}] plot[domain=0.:3.14159265359,variable=\t]({1.*0.2*cos(\t r)+0.*0.2*sin(\t r)},{0.*0.2*cos(\t r)+1.*0.2*sin(\t r)});
\draw [shift={(12.,0.2)}] plot[domain=-3.14159265359:0.,variable=\t]({1.*0.2*cos(\t r)+0.*0.2*sin(\t r)},{0.*0.2*cos(\t r)+1.*0.2*sin(\t r)});
\draw [shift={(12.8,0.2)}] plot[domain=0.:pi,variable=\t]({1.*0.2*cos(\t r)+0.*0.2*sin(\t r)},{0.*0.2*cos(\t r)+1.*0.2*sin(\t r)});
\draw [shift={(12.8,0.2)}] plot[domain=-3.14159265359:0.,variable=\t]({1.*0.2*cos(\t r)+0.*0.2*sin(\t r)},{0.*0.2*cos(\t r)+1.*0.2*sin(\t r)});
\draw [shift={(13.6,0.2)}] plot[domain=0.:3.14159265359,variable=\t]({1.*0.2*cos(\t r)+0.*0.2*sin(\t r)},{0.*0.2*cos(\t r)+1.*0.2*sin(\t r)});
\draw [shift={(13.6,0.2)}] plot[domain=-3.14159265359:0.,variable=\t]({1.*0.2*cos(\t r)+0.*0.2*sin(\t r)},{0.*0.2*cos(\t r)+1.*0.2*sin(\t r)});
\draw [shift={(14.4,0.2)}] plot[domain=0.:3.14159265359,variable=\t]({1.*0.2*cos(\t r)+0.*0.2*sin(\t r)},{0.*0.2*cos(\t r)+1.*0.2*sin(\t r)});
\draw [shift={(14.4,0.2)}] plot[domain=-3.14159265359:0.,variable=\t]({1.*0.2*cos(\t r)+0.*0.2*sin(\t r)},{0.*0.2*cos(\t r)+1.*0.2*sin(\t r)});
\draw [shift={(15.2,0.2)}] plot[domain=0.:pi,variable=\t]({1.*0.2*cos(\t r)+0.*0.2*sin(\t r)},{0.*0.2*cos(\t r)+1.*0.2*sin(\t r)});
\draw [shift={(15.2,0.2)}] plot[domain=-3.14159265359:0.,variable=\t]({1.*0.2*cos(\t r)+0.*0.2*sin(\t r)},{0.*0.2*cos(\t r)+1.*0.2*sin(\t r)});
\draw [shift={(16.,0.2)}] plot[domain=0.:3.14159265359,variable=\t]({1.*0.2*cos(\t r)+0.*0.2*sin(\t r)},{0.*0.2*cos(\t r)+1.*0.2*sin(\t r)});
\draw [shift={(16.,0.2)}] plot[domain=-3.14159265359:0.,variable=\t]({1.*0.2*cos(\t r)+0.*0.2*sin(\t r)},{0.*0.2*cos(\t r)+1.*0.2*sin(\t r)});
\draw [shift={(16.,0.2)}] plot[domain=-1.57079632679:1.57079632679,variable=\t]({1.*0.6*cos(\t r)+0.*0.6*sin(\t r)},{0.*0.6*cos(\t r)+1.*0.6*sin(\t r)});
\draw [shift={(9.6,0.2)}] plot[domain=1.57079632679:4.71238898038,variable=\t]({1.*0.6*cos(\t r)+0.*0.6*sin(\t r)},{0.*0.6*cos(\t r)+1.*0.6*sin(\t r)});
\draw (9.6,0.8)-- (16.,0.8);
\draw (9.6,-0.4)-- (16.,-0.4);
\draw [shift={(9.6,-1.6)}] plot[domain=0.:3.14159265359,variable=\t]({1.*0.2*cos(\t r)+0.*0.2*sin(\t r)},{0.*0.2*cos(\t r)+1.*0.2*sin(\t r)});
\draw [shift={(9.6,-1.6)}] plot[domain=-3.14159265359:0.,variable=\t]({1.*0.2*cos(\t r)+0.*0.2*sin(\t r)},{0.*0.2*cos(\t r)+1.*0.2*sin(\t r)});
\draw [shift={(10.4,-1.6)}] plot[domain=0.:3.14159265359,variable=\t]({1.*0.2*cos(\t r)+0.*0.2*sin(\t r)},{0.*0.2*cos(\t r)+1.*0.2*sin(\t r)});
\draw [shift={(10.4,-1.6)}] plot[domain=-3.14159265359:0.,variable=\t]({1.*0.2*cos(\t r)+0.*0.2*sin(\t r)},{0.*0.2*cos(\t r)+1.*0.2*sin(\t r)});
\draw [shift={(11.2,-1.6)}] plot[domain=0.:3.14159265359,variable=\t]({1.*0.2*cos(\t r)+0.*0.2*sin(\t r)},{0.*0.2*cos(\t r)+1.*0.2*sin(\t r)});
\draw [shift={(11.2,-1.6)}] plot[domain=-3.14159265359:0.,variable=\t]({1.*0.2*cos(\t r)+0.*0.2*sin(\t r)},{0.*0.2*cos(\t r)+1.*0.2*sin(\t r)});
\draw [shift={(12.,-1.6)}] plot[domain=0.:3.14159265359,variable=\t]({1.*0.2*cos(\t r)+0.*0.2*sin(\t r)},{0.*0.2*cos(\t r)+1.*0.2*sin(\t r)});
\draw [shift={(12.,-1.6)}] plot[domain=-3.14159265359:0.,variable=\t]({1.*0.2*cos(\t r)+0.*0.2*sin(\t r)},{0.*0.2*cos(\t r)+1.*0.2*sin(\t r)});
\draw [shift={(12.8,-1.6)}] plot[domain=0.:3.14159265359,variable=\t]({1.*0.2*cos(\t r)+0.*0.2*sin(\t r)},{0.*0.2*cos(\t r)+1.*0.2*sin(\t r)});
\draw [shift={(12.8,-1.6)}] plot[domain=-3.14159265359:0.,variable=\t]({1.*0.2*cos(\t r)+0.*0.2*sin(\t r)},{0.*0.2*cos(\t r)+1.*0.2*sin(\t r)});
\draw [shift={(13.6,-1.6)}] plot[domain=0.:3.14159265359,variable=\t]({1.*0.2*cos(\t r)+0.*0.2*sin(\t r)},{0.*0.2*cos(\t r)+1.*0.2*sin(\t r)});
\draw [shift={(13.6,-1.6)}] plot[domain=-3.14159265359:0.,variable=\t]({1.*0.2*cos(\t r)+0.*0.2*sin(\t r)},{0.*0.2*cos(\t r)+1.*0.2*sin(\t r)});
\draw [shift={(14.4,-1.6)}] plot[domain=0.:3.14159265359,variable=\t]({1.*0.2*cos(\t r)+0.*0.2*sin(\t r)},{0.*0.2*cos(\t r)+1.*0.2*sin(\t r)});
\draw [shift={(14.4,-1.6)}] plot[domain=-3.14159265359:0.,variable=\t]({1.*0.2*cos(\t r)+0.*0.2*sin(\t r)},{0.*0.2*cos(\t r)+1.*0.2*sin(\t r)});
\draw [shift={(15.2,-1.6)}] plot[domain=0.:3.14159265359,variable=\t]({1.*0.2*cos(\t r)+0.*0.2*sin(\t r)},{0.*0.2*cos(\t r)+1.*0.2*sin(\t r)});
\draw [shift={(15.2,-1.6)}] plot[domain=-3.14159265359:0.,variable=\t]({1.*0.2*cos(\t r)+0.*0.2*sin(\t r)},{0.*0.2*cos(\t r)+1.*0.2*sin(\t r)});
\draw [shift={(16.,-1.6)}] plot[domain=0.:pi,variable=\t]({1.*0.2*cos(\t r)+0.*0.2*sin(\t r)},{0.*0.2*cos(\t r)+1.*0.2*sin(\t r)});
\draw [shift={(16.,-1.6)}] plot[domain=-3.14159265359:0.,variable=\t]({1.*0.2*cos(\t r)+0.*0.2*sin(\t r)},{0.*0.2*cos(\t r)+1.*0.2*sin(\t r)});
\draw [shift={(16.,-1.6)}] plot[domain=-1.57079632679:1.57079632679,variable=\t]({1.*0.6*cos(\t r)+0.*0.6*sin(\t r)},{0.*0.6*cos(\t r)+1.*0.6*sin(\t r)});
\draw [shift={(9.6,-1.6)}] plot[domain=1.57079632679:4.71238898038,variable=\t]({1.*0.6*cos(\t r)+0.*0.6*sin(\t r)},{0.*0.6*cos(\t r)+1.*0.6*sin(\t r)});
\draw (9.6,-1.)-- (16.,-1.);
\draw (9.6,-2.2)-- (16.,-2.2);
\draw [shift={(9.6,-3.4)}] plot[domain=0.:pi,variable=\t]({1.*0.2*cos(\t r)+0.*0.2*sin(\t r)},{0.*0.2*cos(\t r)+1.*0.2*sin(\t r)});
\draw [shift={(9.6,-3.4)}] plot[domain=-3.14159265359:0.,variable=\t]({1.*0.2*cos(\t r)+0.*0.2*sin(\t r)},{0.*0.2*cos(\t r)+1.*0.2*sin(\t r)});
\draw [shift={(10.4,-3.4)}] plot[domain=0.:pi,variable=\t]({1.*0.2*cos(\t r)+0.*0.2*sin(\t r)},{0.*0.2*cos(\t r)+1.*0.2*sin(\t r)});
\draw [shift={(10.4,-3.4)}] plot[domain=-3.14159265359:0.,variable=\t]({1.*0.2*cos(\t r)+0.*0.2*sin(\t r)},{0.*0.2*cos(\t r)+1.*0.2*sin(\t r)});
\draw [shift={(11.2,-3.4)}] plot[domain=0.:3.14159265359,variable=\t]({1.*0.2*cos(\t r)+0.*0.2*sin(\t r)},{0.*0.2*cos(\t r)+1.*0.2*sin(\t r)});
\draw [shift={(11.2,-3.4)}] plot[domain=-3.14159265359:0.,variable=\t]({1.*0.2*cos(\t r)+0.*0.2*sin(\t r)},{0.*0.2*cos(\t r)+1.*0.2*sin(\t r)});
\draw [shift={(12.,-3.4)}] plot[domain=0.:3.14159265359,variable=\t]({1.*0.2*cos(\t r)+0.*0.2*sin(\t r)},{0.*0.2*cos(\t r)+1.*0.2*sin(\t r)});
\draw [shift={(12.,-3.4)}] plot[domain=-3.14159265359:0.,variable=\t]({1.*0.2*cos(\t r)+0.*0.2*sin(\t r)},{0.*0.2*cos(\t r)+1.*0.2*sin(\t r)});
\draw [shift={(12.8,-3.4)}] plot[domain=0.:3.14159265359,variable=\t]({1.*0.2*cos(\t r)+0.*0.2*sin(\t r)},{0.*0.2*cos(\t r)+1.*0.2*sin(\t r)});
\draw [shift={(12.8,-3.4)}] plot[domain=-3.14159265359:0.,variable=\t]({1.*0.2*cos(\t r)+0.*0.2*sin(\t r)},{0.*0.2*cos(\t r)+1.*0.2*sin(\t r)});
\draw [shift={(13.6,-3.4)}] plot[domain=0.:pi,variable=\t]({1.*0.2*cos(\t r)+0.*0.2*sin(\t r)},{0.*0.2*cos(\t r)+1.*0.2*sin(\t r)});
\draw [shift={(13.6,-3.4)}] plot[domain=-3.14159265359:0.,variable=\t]({1.*0.2*cos(\t r)+0.*0.2*sin(\t r)},{0.*0.2*cos(\t r)+1.*0.2*sin(\t r)});
\draw [shift={(14.4,-3.4)}] plot[domain=0.:3.14159265359,variable=\t]({1.*0.2*cos(\t r)+0.*0.2*sin(\t r)},{0.*0.2*cos(\t r)+1.*0.2*sin(\t r)});
\draw [shift={(14.4,-3.4)}] plot[domain=-3.14159265359:0.,variable=\t]({1.*0.2*cos(\t r)+0.*0.2*sin(\t r)},{0.*0.2*cos(\t r)+1.*0.2*sin(\t r)});
\draw [shift={(15.2,-3.4)}] plot[domain=0.:3.14159265359,variable=\t]({1.*0.2*cos(\t r)+0.*0.2*sin(\t r)},{0.*0.2*cos(\t r)+1.*0.2*sin(\t r)});
\draw [shift={(15.2,-3.4)}] plot[domain=-3.14159265359:0.,variable=\t]({1.*0.2*cos(\t r)+0.*0.2*sin(\t r)},{0.*0.2*cos(\t r)+1.*0.2*sin(\t r)});
\draw [shift={(16.,-3.4)}] plot[domain=0.:3.14159265359,variable=\t]({1.*0.2*cos(\t r)+0.*0.2*sin(\t r)},{0.*0.2*cos(\t r)+1.*0.2*sin(\t r)});
\draw [shift={(16.,-3.4)}] plot[domain=-3.14159265359:0.,variable=\t]({1.*0.2*cos(\t r)+0.*0.2*sin(\t r)},{0.*0.2*cos(\t r)+1.*0.2*sin(\t r)});
\draw [shift={(16.,-3.4)}] plot[domain=-1.57079632679:1.57079632679,variable=\t]({1.*0.6*cos(\t r)+0.*0.6*sin(\t r)},{0.*0.6*cos(\t r)+1.*0.6*sin(\t r)});
\draw [shift={(9.6,-3.4)}] plot[domain=1.57079632679:4.71238898038,variable=\t]({1.*0.6*cos(\t r)+0.*0.6*sin(\t r)},{0.*0.6*cos(\t r)+1.*0.6*sin(\t r)});
\draw (9.6,-2.8)-- (16.,-2.8);
\draw (9.6,-4.)-- (16.,-4.);
\draw [shift={(9.6,-5.2)}] plot[domain=0.:3.14159265359,variable=\t]({1.*0.2*cos(\t r)+0.*0.2*sin(\t r)},{0.*0.2*cos(\t r)+1.*0.2*sin(\t r)});
\draw [shift={(9.6,-5.2)}] plot[domain=-3.14159265359:0.,variable=\t]({1.*0.2*cos(\t r)+0.*0.2*sin(\t r)},{0.*0.2*cos(\t r)+1.*0.2*sin(\t r)});
\draw [shift={(10.4,-5.2)}] plot[domain=0.:3.14159265359,variable=\t]({1.*0.2*cos(\t r)+0.*0.2*sin(\t r)},{0.*0.2*cos(\t r)+1.*0.2*sin(\t r)});
\draw [shift={(10.4,-5.2)}] plot[domain=-3.14159265359:0.,variable=\t]({1.*0.2*cos(\t r)+0.*0.2*sin(\t r)},{0.*0.2*cos(\t r)+1.*0.2*sin(\t r)});
\draw [shift={(11.2,-5.2)}] plot[domain=0.:3.14159265359,variable=\t]({1.*0.2*cos(\t r)+0.*0.2*sin(\t r)},{0.*0.2*cos(\t r)+1.*0.2*sin(\t r)});
\draw [shift={(11.2,-5.2)}] plot[domain=-3.14159265359:0.,variable=\t]({1.*0.2*cos(\t r)+0.*0.2*sin(\t r)},{0.*0.2*cos(\t r)+1.*0.2*sin(\t r)});
\draw [shift={(12.,-5.2)}] plot[domain=0.:pi,variable=\t]({1.*0.2*cos(\t r)+0.*0.2*sin(\t r)},{0.*0.2*cos(\t r)+1.*0.2*sin(\t r)});
\draw [shift={(12.,-5.2)}] plot[domain=-3.14159265359:0.,variable=\t]({1.*0.2*cos(\t r)+0.*0.2*sin(\t r)},{0.*0.2*cos(\t r)+1.*0.2*sin(\t r)});
\draw [shift={(12.8,-5.2)}] plot[domain=0.:3.14159265359,variable=\t]({1.*0.2*cos(\t r)+0.*0.2*sin(\t r)},{0.*0.2*cos(\t r)+1.*0.2*sin(\t r)});
\draw [shift={(12.8,-5.2)}] plot[domain=-3.14159265359:0.,variable=\t]({1.*0.2*cos(\t r)+0.*0.2*sin(\t r)},{0.*0.2*cos(\t r)+1.*0.2*sin(\t r)});
\draw [shift={(13.6,-5.2)}] plot[domain=0.:3.14159265359,variable=\t]({1.*0.2*cos(\t r)+0.*0.2*sin(\t r)},{0.*0.2*cos(\t r)+1.*0.2*sin(\t r)});
\draw [shift={(13.6,-5.2)}] plot[domain=-3.14159265359:0.,variable=\t]({1.*0.2*cos(\t r)+0.*0.2*sin(\t r)},{0.*0.2*cos(\t r)+1.*0.2*sin(\t r)});
\draw [shift={(14.4,-5.2)}] plot[domain=0.:3.14159265359,variable=\t]({1.*0.2*cos(\t r)+0.*0.2*sin(\t r)},{0.*0.2*cos(\t r)+1.*0.2*sin(\t r)});
\draw [shift={(14.4,-5.2)}] plot[domain=-3.14159265359:0.,variable=\t]({1.*0.2*cos(\t r)+0.*0.2*sin(\t r)},{0.*0.2*cos(\t r)+1.*0.2*sin(\t r)});
\draw [shift={(15.2,-5.2)}] plot[domain=0.:pi,variable=\t]({1.*0.2*cos(\t r)+0.*0.2*sin(\t r)},{0.*0.2*cos(\t r)+1.*0.2*sin(\t r)});
\draw [shift={(15.2,-5.2)}] plot[domain=-3.14159265359:0.,variable=\t]({1.*0.2*cos(\t r)+0.*0.2*sin(\t r)},{0.*0.2*cos(\t r)+1.*0.2*sin(\t r)});
\draw [shift={(16.,-5.2)}] plot[domain=0.:3.14159265359,variable=\t]({1.*0.2*cos(\t r)+0.*0.2*sin(\t r)},{0.*0.2*cos(\t r)+1.*0.2*sin(\t r)});
\draw [shift={(16.,-5.2)}] plot[domain=-3.14159265359:0.,variable=\t]({1.*0.2*cos(\t r)+0.*0.2*sin(\t r)},{0.*0.2*cos(\t r)+1.*0.2*sin(\t r)});
\draw [shift={(16.,-5.2)}] plot[domain=-1.57079632679:1.57079632679,variable=\t]({1.*0.6*cos(\t r)+0.*0.6*sin(\t r)},{0.*0.6*cos(\t r)+1.*0.6*sin(\t r)});
\draw [shift={(9.6,-5.2)}] plot[domain=1.57079632679:4.71238898038,variable=\t]({1.*0.6*cos(\t r)+0.*0.6*sin(\t r)},{0.*0.6*cos(\t r)+1.*0.6*sin(\t r)});
\draw (9.6,-4.6)-- (16.,-4.6);
\draw (9.6,-5.8)-- (16.,-5.8);
\draw [shift={(1.8,0.3)},dash pattern=on 1pt off 1pt]  plot[domain=3.78509376238:5.63968419839,variable=\t]({1.*0.5*cos(\t r)+0.*0.5*sin(\t r)},{0.*0.5*cos(\t r)+1.*0.5*sin(\t r)});
\draw [shift={(2.6,0.3)}] plot[domain=3.78509376238:5.63968419839,variable=\t]({1.*0.5*cos(\t r)+0.*0.5*sin(\t r)},{0.*0.5*cos(\t r)+1.*0.5*sin(\t r)});
\draw [shift={(3.4,0.3)},dash pattern=on 1pt off 1pt]  plot[domain=3.78509376238:5.63968419839,variable=\t]({1.*0.5*cos(\t r)+0.*0.5*sin(\t r)},{0.*0.5*cos(\t r)+1.*0.5*sin(\t r)});
\draw [shift={(4.00963681174,0.0253679664386)}] plot[domain=-3.02116905014:0.143161706395,variable=\t]({1.*0.211166111295*cos(\t r)+0.*0.211166111295*sin(\t r)},{0.*0.211166111295*cos(\t r)+1.*0.211166111295*sin(\t r)});
\draw [shift={(4.60464555162,0.190947601665)}] plot[domain=1.58215270051:3.47907705673,variable=\t]({1.*0.409078776929*cos(\t r)+0.*0.409078776929*sin(\t r)},{0.*0.409078776929*cos(\t r)+1.*0.409078776929*sin(\t r)});
\draw [shift={(1.46485781721,-0.15)}] plot[domain=1.97890584735:4.30427945983,variable=\t]({1.*0.163421346382*cos(\t r)+0.*0.163421346382*sin(\t r)},{0.*0.163421346382*cos(\t r)+1.*0.163421346382*sin(\t r)});
\draw (1.4,-0.3)-- (7.8,-0.3);
\draw [shift={(7.75,0.15)}] plot[domain=-1.46013910562:1.46013910562,variable=\t]({1.*0.452769256907*cos(\t r)+0.*0.452769256907*sin(\t r)},{0.*0.452769256907*cos(\t r)+1.*0.452769256907*sin(\t r)});
\draw (4.6,0.6)-- (7.8,0.6);
\draw [shift={(1.15575613652,-1.66726840957)},dash pattern=on 1pt off 1pt]  plot[domain=0.830380494905:2.95471326748,variable=\t]({1.*0.362060033146*cos(\t r)+0.*0.362060033146*sin(\t r)},{0.*0.362060033146*cos(\t r)+1.*0.362060033146*sin(\t r)});
\draw [shift={(1.13310939327,-1.60067182019)}] plot[domain=3.13957584082:5.64170101116,variable=\t]({1.*0.333110070738*cos(\t r)+0.*0.333110070738*sin(\t r)},{0.*0.333110070738*cos(\t r)+1.*0.333110070738*sin(\t r)});
\draw [shift={(1.54088270844,-1.5457323914)}] plot[domain=-0.253548906183:2.33927556227,variable=\t]({1.*0.202696490943*cos(\t r)+0.*0.202696490943*sin(\t r)},{0.*0.202696490943*cos(\t r)+1.*0.202696490943*sin(\t r)});
\draw [shift={(2.12252702077,-1.63277716314)}] plot[domain=3.04794541694:4.3903458049,variable=\t]({1.*0.387124634624*cos(\t r)+0.*0.387124634624*sin(\t r)},{0.*0.387124634624*cos(\t r)+1.*0.387124634624*sin(\t r)});
\draw [shift={(1.8,-1.5)},dash pattern=on 1pt off 1pt]  plot[domain=3.78509376238:5.63968419839,variable=\t]({1.*0.5*cos(\t r)+0.*0.5*sin(\t r)},{0.*0.5*cos(\t r)+1.*0.5*sin(\t r)});
\draw [shift={(2.6,-1.5)}] plot[domain=3.78509376238:5.63968419839,variable=\t]({1.*0.5*cos(\t r)+0.*0.5*sin(\t r)},{0.*0.5*cos(\t r)+1.*0.5*sin(\t r)});
\draw [shift={(3.4,-1.5)},dash pattern=on 1pt off 1pt]  plot[domain=3.78509376238:5.63968419839,variable=\t]({1.*0.5*cos(\t r)+0.*0.5*sin(\t r)},{0.*0.5*cos(\t r)+1.*0.5*sin(\t r)});
\draw [shift={(9.86170715716,134.704386086)}] plot[domain=4.65494343949:4.69294934124,variable=\t]({1.*136.930258214*cos(\t r)+0.*136.930258214*sin(\t r)},{0.*136.930258214*cos(\t r)+1.*136.930258214*sin(\t r)});
\draw [shift={(3.94652346566,-1.8096036616)}] plot[domain=-3.20704254322:0.145296834103,variable=\t]({1.*0.146837857198*cos(\t r)+0.*0.146837857198*sin(\t r)},{0.*0.146837857198*cos(\t r)+1.*0.146837857198*sin(\t r)});
\draw [shift={(4.70184243243,-1.80161348912)}] plot[domain=1.7384887266:3.11984313642,variable=\t]({1.*0.610172656981*cos(\t r)+0.*0.610172656981*sin(\t r)},{0.*0.610172656981*cos(\t r)+1.*0.610172656981*sin(\t r)});
\draw (4.6,-1.2)-- (7.8,-1.2);
\draw [shift={(7.83328010111,-1.5)}] plot[domain=-1.68127826773:1.68127826773,variable=\t]({1.*0.301840297393*cos(\t r)+0.*0.301840297393*sin(\t r)},{0.*0.301840297393*cos(\t r)+1.*0.301840297393*sin(\t r)});
\draw [shift={(7.07511397046,-1.36267095569)},dash pattern=on 1pt off 1pt]  plot[domain=4.86044566243:5.74033750543,variable=\t]({1.*0.846591075329*cos(\t r)+0.*0.846591075329*sin(\t r)},{0.*0.846591075329*cos(\t r)+1.*0.846591075329*sin(\t r)});
\draw [shift={(1.23504070966,-3.42991858068)}] plot[domain=3.01498248749:5.48250025528,variable=\t]({1.*0.236937242042*cos(\t r)+0.*0.236937242042*sin(\t r)},{0.*0.236937242042*cos(\t r)+1.*0.236937242042*sin(\t r)});
\draw [shift={(1.42899394244,-3.42899394244)}] plot[domain=1.63827960215:3.07410937823,variable=\t]({1.*0.42997261698*cos(\t r)+0.*0.42997261698*sin(\t r)},{0.*0.42997261698*cos(\t r)+1.*0.42997261698*sin(\t r)});
\draw [shift={(1.35566223854,-3.44433776146)}] plot[domain=0.0994547024676:1.47134162433,variable=\t]({1.*0.44654438005*cos(\t r)+0.*0.44654438005*sin(\t r)},{0.*0.44654438005*cos(\t r)+1.*0.44654438005*sin(\t r)});
\draw [shift={(2.29120683405,-3.41431422071)}] plot[domain=3.11245997477:4.52463516328,variable=\t]({1.*0.491415354592*cos(\t r)+0.*0.491415354592*sin(\t r)},{0.*0.491415354592*cos(\t r)+1.*0.491415354592*sin(\t r)});
\draw [shift={(1.8,-3.3)},dash pattern=on 1pt off 1pt]  plot[domain=3.78509376238:5.63968419839,variable=\t]({1.*0.5*cos(\t r)+0.*0.5*sin(\t r)},{0.*0.5*cos(\t r)+1.*0.5*sin(\t r)});
\draw [shift={(2.6,-3.3)}] plot[domain=3.78509376238:5.63968419839,variable=\t]({1.*0.5*cos(\t r)+0.*0.5*sin(\t r)},{0.*0.5*cos(\t r)+1.*0.5*sin(\t r)});
\draw [shift={(3.4,-3.3)},dash pattern=on 1pt off 1pt]  plot[domain=3.78509376238:5.63968419839,variable=\t]({1.*0.5*cos(\t r)+0.*0.5*sin(\t r)},{0.*0.5*cos(\t r)+1.*0.5*sin(\t r)});
\draw [shift={(3.95005130793,-3.51739045595)}] plot[domain=3.64485187617:6.04182561719,variable=\t]({1.*0.171288446132*cos(\t r)+0.*0.171288446132*sin(\t r)},{0.*0.171288446132*cos(\t r)+1.*0.171288446132*sin(\t r)});
\draw [shift={(4.64289508628,-3.52577884202)}] plot[domain=1.65219993644:3.2033417008,variable=\t]({1.*0.527525714203*cos(\t r)+0.*0.527525714203*sin(\t r)},{0.*0.527525714203*cos(\t r)+1.*0.527525714203*sin(\t r)});
\draw (2.1994828513,-3.89709342668)-- (7.2,-4.);
\draw (4.6,-3.)-- (7.2,-3.);
\draw [shift={(7.03835669957,-3.56164330043)}] plot[domain=0.280230774874:1.29056555192,variable=\t]({1.*0.584441402954*cos(\t r)+0.*0.584441402954*sin(\t r)},{0.*0.584441402954*cos(\t r)+1.*0.584441402954*sin(\t r)});
\draw [shift={(7.13250577281,-3.52167051521)},dash pattern=on 1pt off 1pt]  plot[domain=-1.43061772399:0.254612516893,variable=\t]({1.*0.483067869686*cos(\t r)+0.*0.483067869686*sin(\t r)},{0.*0.483067869686*cos(\t r)+1.*0.483067869686*sin(\t r)});
\draw [shift={(1.4,-7.)}] plot[domain=0.:pi,variable=\t]({1.*0.2*cos(\t r)+0.*0.2*sin(\t r)},{0.*0.2*cos(\t r)+1.*0.2*sin(\t r)});
\draw [shift={(1.4,-7.)}] plot[domain=-3.14159265359:0.,variable=\t]({1.*0.2*cos(\t r)+0.*0.2*sin(\t r)},{0.*0.2*cos(\t r)+1.*0.2*sin(\t r)});
\draw [shift={(2.2,-7.)}] plot[domain=0.:3.14159265359,variable=\t]({1.*0.2*cos(\t r)+0.*0.2*sin(\t r)},{0.*0.2*cos(\t r)+1.*0.2*sin(\t r)});
\draw [shift={(2.2,-7.)}] plot[domain=-3.14159265359:0.,variable=\t]({1.*0.2*cos(\t r)+0.*0.2*sin(\t r)},{0.*0.2*cos(\t r)+1.*0.2*sin(\t r)});
\draw [shift={(3.,-7.)}] plot[domain=0.:3.14159265359,variable=\t]({1.*0.2*cos(\t r)+0.*0.2*sin(\t r)},{0.*0.2*cos(\t r)+1.*0.2*sin(\t r)});
\draw [shift={(3.,-7.)}] plot[domain=-3.14159265359:0.,variable=\t]({1.*0.2*cos(\t r)+0.*0.2*sin(\t r)},{0.*0.2*cos(\t r)+1.*0.2*sin(\t r)});
\draw [shift={(3.8,-7.)}] plot[domain=0.:3.14159265359,variable=\t]({1.*0.2*cos(\t r)+0.*0.2*sin(\t r)},{0.*0.2*cos(\t r)+1.*0.2*sin(\t r)});
\draw [shift={(3.8,-7.)}] plot[domain=-3.14159265359:0.,variable=\t]({1.*0.2*cos(\t r)+0.*0.2*sin(\t r)},{0.*0.2*cos(\t r)+1.*0.2*sin(\t r)});
\draw [shift={(4.6,-7.)}] plot[domain=0.:pi,variable=\t]({1.*0.2*cos(\t r)+0.*0.2*sin(\t r)},{0.*0.2*cos(\t r)+1.*0.2*sin(\t r)});
\draw [shift={(4.6,-7.)}] plot[domain=-3.14159265359:0.,variable=\t]({1.*0.2*cos(\t r)+0.*0.2*sin(\t r)},{0.*0.2*cos(\t r)+1.*0.2*sin(\t r)});
\draw [shift={(5.4,-7.)}] plot[domain=0.:3.14159265359,variable=\t]({1.*0.2*cos(\t r)+0.*0.2*sin(\t r)},{0.*0.2*cos(\t r)+1.*0.2*sin(\t r)});
\draw [shift={(5.4,-7.)}] plot[domain=-3.14159265359:0.,variable=\t]({1.*0.2*cos(\t r)+0.*0.2*sin(\t r)},{0.*0.2*cos(\t r)+1.*0.2*sin(\t r)});
\draw [shift={(6.2,-7.)}] plot[domain=0.:3.14159265359,variable=\t]({1.*0.2*cos(\t r)+0.*0.2*sin(\t r)},{0.*0.2*cos(\t r)+1.*0.2*sin(\t r)});
\draw [shift={(6.2,-7.)}] plot[domain=-3.14159265359:0.,variable=\t]({1.*0.2*cos(\t r)+0.*0.2*sin(\t r)},{0.*0.2*cos(\t r)+1.*0.2*sin(\t r)});
\draw [shift={(7.,-7.)}] plot[domain=0.:3.14159265359,variable=\t]({1.*0.2*cos(\t r)+0.*0.2*sin(\t r)},{0.*0.2*cos(\t r)+1.*0.2*sin(\t r)});
\draw [shift={(7.,-7.)}] plot[domain=-3.14159265359:0.,variable=\t]({1.*0.2*cos(\t r)+0.*0.2*sin(\t r)},{0.*0.2*cos(\t r)+1.*0.2*sin(\t r)});
\draw [shift={(7.8,-7.)}] plot[domain=0.:3.14159265359,variable=\t]({1.*0.2*cos(\t r)+0.*0.2*sin(\t r)},{0.*0.2*cos(\t r)+1.*0.2*sin(\t r)});
\draw [shift={(7.8,-7.)}] plot[domain=-3.14159265359:0.,variable=\t]({1.*0.2*cos(\t r)+0.*0.2*sin(\t r)},{0.*0.2*cos(\t r)+1.*0.2*sin(\t r)});
\draw [shift={(7.8,-7.)}] plot[domain=-1.57079632679:1.57079632679,variable=\t]({1.*0.6*cos(\t r)+0.*0.6*sin(\t r)},{0.*0.6*cos(\t r)+1.*0.6*sin(\t r)});
\draw [shift={(1.4,-7.)}] plot[domain=1.57079632679:4.71238898038,variable=\t]({1.*0.6*cos(\t r)+0.*0.6*sin(\t r)},{0.*0.6*cos(\t r)+1.*0.6*sin(\t r)});
\draw (1.4,-6.4)-- (7.8,-6.4);
\draw (1.4,-7.6)-- (7.8,-7.6);
\draw [shift={(9.6,-7.)}] plot[domain=0.:3.14159265359,variable=\t]({1.*0.2*cos(\t r)+0.*0.2*sin(\t r)},{0.*0.2*cos(\t r)+1.*0.2*sin(\t r)});
\draw [shift={(9.6,-7.)}] plot[domain=-3.14159265359:0.,variable=\t]({1.*0.2*cos(\t r)+0.*0.2*sin(\t r)},{0.*0.2*cos(\t r)+1.*0.2*sin(\t r)});
\draw [shift={(10.4,-7.)}] plot[domain=0.:3.14159265359,variable=\t]({1.*0.2*cos(\t r)+0.*0.2*sin(\t r)},{0.*0.2*cos(\t r)+1.*0.2*sin(\t r)});
\draw [shift={(10.4,-7.)}] plot[domain=-3.14159265359:0.,variable=\t]({1.*0.2*cos(\t r)+0.*0.2*sin(\t r)},{0.*0.2*cos(\t r)+1.*0.2*sin(\t r)});
\draw [shift={(11.2,-7.)}] plot[domain=0.:3.14159265359,variable=\t]({1.*0.2*cos(\t r)+0.*0.2*sin(\t r)},{0.*0.2*cos(\t r)+1.*0.2*sin(\t r)});
\draw [shift={(11.2,-7.)}] plot[domain=-3.14159265359:0.,variable=\t]({1.*0.2*cos(\t r)+0.*0.2*sin(\t r)},{0.*0.2*cos(\t r)+1.*0.2*sin(\t r)});
\draw [shift={(12.,-7.)}] plot[domain=0.:3.14159265359,variable=\t]({1.*0.2*cos(\t r)+0.*0.2*sin(\t r)},{0.*0.2*cos(\t r)+1.*0.2*sin(\t r)});
\draw [shift={(12.,-7.)}] plot[domain=-3.14159265359:0.,variable=\t]({1.*0.2*cos(\t r)+0.*0.2*sin(\t r)},{0.*0.2*cos(\t r)+1.*0.2*sin(\t r)});
\draw [shift={(12.8,-7.)}] plot[domain=0.:3.14159265359,variable=\t]({1.*0.2*cos(\t r)+0.*0.2*sin(\t r)},{0.*0.2*cos(\t r)+1.*0.2*sin(\t r)});
\draw [shift={(12.8,-7.)}] plot[domain=-3.14159265359:0.,variable=\t]({1.*0.2*cos(\t r)+0.*0.2*sin(\t r)},{0.*0.2*cos(\t r)+1.*0.2*sin(\t r)});
\draw [shift={(13.6,-7.)}] plot[domain=0.:3.14159265359,variable=\t]({1.*0.2*cos(\t r)+0.*0.2*sin(\t r)},{0.*0.2*cos(\t r)+1.*0.2*sin(\t r)});
\draw [shift={(13.6,-7.)}] plot[domain=-3.14159265359:0.,variable=\t]({1.*0.2*cos(\t r)+0.*0.2*sin(\t r)},{0.*0.2*cos(\t r)+1.*0.2*sin(\t r)});
\draw [shift={(14.4,-7.)}] plot[domain=0.:3.14159265359,variable=\t]({1.*0.2*cos(\t r)+0.*0.2*sin(\t r)},{0.*0.2*cos(\t r)+1.*0.2*sin(\t r)});
\draw [shift={(14.4,-7.)}] plot[domain=-3.14159265359:0.,variable=\t]({1.*0.2*cos(\t r)+0.*0.2*sin(\t r)},{0.*0.2*cos(\t r)+1.*0.2*sin(\t r)});
\draw [shift={(15.2,-7.)}] plot[domain=0.:3.14159265359,variable=\t]({1.*0.2*cos(\t r)+0.*0.2*sin(\t r)},{0.*0.2*cos(\t r)+1.*0.2*sin(\t r)});
\draw [shift={(15.2,-7.)}] plot[domain=-3.14159265359:0.,variable=\t]({1.*0.2*cos(\t r)+0.*0.2*sin(\t r)},{0.*0.2*cos(\t r)+1.*0.2*sin(\t r)});
\draw [shift={(16.,-7.)}] plot[domain=0.:3.14159265359,variable=\t]({1.*0.2*cos(\t r)+0.*0.2*sin(\t r)},{0.*0.2*cos(\t r)+1.*0.2*sin(\t r)});
\draw [shift={(16.,-7.)}] plot[domain=-3.14159265359:0.,variable=\t]({1.*0.2*cos(\t r)+0.*0.2*sin(\t r)},{0.*0.2*cos(\t r)+1.*0.2*sin(\t r)});
\draw [shift={(16.,-7.)}] plot[domain=-1.57079632679:1.57079632679,variable=\t]({1.*0.6*cos(\t r)+0.*0.6*sin(\t r)},{0.*0.6*cos(\t r)+1.*0.6*sin(\t r)});
\draw [shift={(9.6,-7.)}] plot[domain=1.57079632679:4.71238898038,variable=\t]({1.*0.6*cos(\t r)+0.*0.6*sin(\t r)},{0.*0.6*cos(\t r)+1.*0.6*sin(\t r)});
\draw (9.6,-6.4)-- (16.,-6.4);
\draw (9.6,-7.6)-- (16.,-7.6);
\draw [shift={(2.84412764367,-5.310844427)}] plot[domain=2.94147096474:5.76362503686,variable=\t]({1.*0.179568671171*cos(\t r)+0.*0.179568671171*sin(\t r)},{0.*0.179568671171*cos(\t r)+1.*0.179568671171*sin(\t r)});
\draw [shift={(2.21185560043,-5.25229387043)}] plot[domain=-0.0500458332301:1.30436151112,variable=\t]({1.*0.45685913469*cos(\t r)+0.*0.45685913469*sin(\t r)},{0.*0.45685913469*cos(\t r)+1.*0.45685913469*sin(\t r)});
\draw [shift={(2.20324069166,-5.19368932288)}] plot[domain=1.24545879202:3.15724127767,variable=\t]({1.*0.403290069379*cos(\t r)+0.*0.403290069379*sin(\t r)},{0.*0.403290069379*cos(\t r)+1.*0.403290069379*sin(\t r)});
\draw [shift={(2.20234461682,-5.19765538318)}] plot[domain=3.14741997223:4.70656166174,variable=\t]({1.*0.402351448259*cos(\t r)+0.*0.402351448259*sin(\t r)},{0.*0.402351448259*cos(\t r)+1.*0.402351448259*sin(\t r)});
\draw [shift={(3.4,-5.1)},dash pattern=on 1pt off 1pt]  plot[domain=3.78509376238:5.63968419839,variable=\t]({1.*0.5*cos(\t r)+0.*0.5*sin(\t r)},{0.*0.5*cos(\t r)+1.*0.5*sin(\t r)});
\draw [shift={(3.95168170858,-5.28412323312)}] plot[domain=-2.48922513648:0.273225539431,variable=\t]({1.*0.190878929749*cos(\t r)+0.*0.190878929749*sin(\t r)},{0.*0.190878929749*cos(\t r)+1.*0.190878929749*sin(\t r)});
\draw [shift={(4.52504369272,-5.18521232012)}] plot[domain=1.37861345854:3.26268318374,variable=\t]({1.*0.39243723011*cos(\t r)+0.*0.39243723011*sin(\t r)},{0.*0.39243723011*cos(\t r)+1.*0.39243723011*sin(\t r)});
\draw [shift={(6.86186813493,38.7945967856)}] plot[domain=4.66055132272:4.71608864562,variable=\t]({1.*43.6532348899*cos(\t r)+0.*43.6532348899*sin(\t r)},{0.*43.6532348899*cos(\t r)+1.*43.6532348899*sin(\t r)});
\draw [shift={(7.05172294218,-5.06016856884)}] plot[domain=-0.756092505714:1.71036228958,variable=\t]({1.*0.203810978647*cos(\t r)+0.*0.203810978647*sin(\t r)},{0.*0.203810978647*cos(\t r)+1.*0.203810978647*sin(\t r)});
\draw [shift={(7.0011693984,-5.46705646613)},dash pattern=on 1pt off 1pt]  plot[domain=-1.57430861485:0.930807506054,variable=\t]({1.*0.332945587499*cos(\t r)+0.*0.332945587499*sin(\t r)},{0.*0.332945587499*cos(\t r)+1.*0.332945587499*sin(\t r)});
\draw (2.2,-5.6)-- (7.,-5.8);
\draw [shift={(2.64235700078,-7.00921136885)},dash pattern=on 1pt off 1pt]  plot[domain=1.89374894546:3.12725368394,variable=\t]({1.*0.642423042683*cos(\t r)+0.*0.642423042683*sin(\t r)},{0.*0.642423042683*cos(\t r)+1.*0.642423042683*sin(\t r)});
\draw [shift={(2.40765444411,-6.70156613328)}] plot[domain=-0.589446383789:1.46895638004,variable=\t]({1.*0.303136747533*cos(\t r)+0.*0.303136747533*sin(\t r)},{0.*0.303136747533*cos(\t r)+1.*0.303136747533*sin(\t r)});
\draw [shift={(2.85368297531,-7.01041985871)}] plot[domain=2.51544748968:5.36969366108,variable=\t]({1.*0.239477142304*cos(\t r)+0.*0.239477142304*sin(\t r)},{0.*0.239477142304*cos(\t r)+1.*0.239477142304*sin(\t r)});
\draw [shift={(3.4,-6.9)},dash pattern=on 1pt off 1pt]  plot[domain=3.78509376238:5.63968419839,variable=\t]({1.*0.5*cos(\t r)+0.*0.5*sin(\t r)},{0.*0.5*cos(\t r)+1.*0.5*sin(\t r)});
\draw [shift={(3.96726779722,-7.09916842837)}] plot[domain=3.68407958752:6.03822128702,variable=\t]({1.*0.195308785837*cos(\t r)+0.*0.195308785837*sin(\t r)},{0.*0.195308785837*cos(\t r)+1.*0.195308785837*sin(\t r)});
\draw [shift={(4.64628041063,-7.09054746747)}] plot[domain=1.664862305:3.25546676311,variable=\t]({1.*0.492725779973*cos(\t r)+0.*0.492725779973*sin(\t r)},{0.*0.492725779973*cos(\t r)+1.*0.492725779973*sin(\t r)});
\draw [shift={(2.24489756347,-7.19777799524)}] plot[domain=2.46223787069:4.54286794338,variable=\t]({1.*0.314787153479*cos(\t r)+0.*0.314787153479*sin(\t r)},{0.*0.314787153479*cos(\t r)+1.*0.314787153479*sin(\t r)});
\draw (2.19178973748,-7.5080528908)-- (6.4,-7.6);
\draw (4.6,-6.6)-- (6.2,-6.6);
\draw [shift={(6.23381715131,-7.19927427304)}] plot[domain=0.338420729923:1.62716671657,variable=\t]({1.*0.600227668514*cos(\t r)+0.*0.600227668514*sin(\t r)},{0.*0.600227668514*cos(\t r)+1.*0.600227668514*sin(\t r)});
\draw [shift={(6.15337379735,-7.00224919823)},dash pattern=on 1pt off 1pt]  plot[domain=-1.17948355232:0.00347834522041,variable=\t]({1.*0.646630114399*cos(\t r)+0.*0.646630114399*sin(\t r)},{0.*0.646630114399*cos(\t r)+1.*0.646630114399*sin(\t r)});
\draw [shift={(11.6,0.3)},dash pattern=on 1pt off 1pt]  plot[domain=3.78509376238:5.63968419839,variable=\t]({1.*0.5*cos(\t r)+0.*0.5*sin(\t r)},{0.*0.5*cos(\t r)+1.*0.5*sin(\t r)});
\draw [shift={(11.091959188,0.6)},dash pattern=on 1pt off 1pt]  plot[domain=-1.07550508669:1.07550508669,variable=\t]({1.*0.22731655694*cos(\t r)+0.*0.22731655694*sin(\t r)},{0.*0.22731655694*cos(\t r)+1.*0.22731655694*sin(\t r)});
\draw [shift={(11.1726809256,0.4)}] plot[domain=1.50260453772:4.78058076946,variable=\t]({1.*0.400931829402*cos(\t r)+0.*0.400931829402*sin(\t r)},{0.*0.400931829402*cos(\t r)+1.*0.400931829402*sin(\t r)});
\draw [shift={(11.3504410098,0.455704007952)}] plot[domain=0.57660474968:3.49621129799,variable=\t]({1.*0.160422672774*cos(\t r)+0.*0.160422672774*sin(\t r)},{0.*0.160422672774*cos(\t r)+1.*0.160422672774*sin(\t r)});
\draw [shift={(12.7113349326,0.940033753317)}] plot[domain=3.45456114194:4.47133193267,variable=\t]({1.*1.28902452509*cos(\t r)+0.*1.28902452509*sin(\t r)},{0.*1.28902452509*cos(\t r)+1.*1.28902452509*sin(\t r)});
\draw [shift={(12.1425206687,0.0614244163826)}] plot[domain=3.54852257229:6.20819803398,variable=\t]({1.*0.155193749664*cos(\t r)+0.*0.155193749664*sin(\t r)},{0.*0.155193749664*cos(\t r)+1.*0.155193749664*sin(\t r)});
\draw [shift={(12.8965204628,0.00703850188178)}] plot[domain=1.73215804231:3.07035782848,variable=\t]({1.*0.600765792966*cos(\t r)+0.*0.600765792966*sin(\t r)},{0.*0.600765792966*cos(\t r)+1.*0.600765792966*sin(\t r)});
\draw (12.8,0.6)-- (14.0495768083,0.69627682927);
\draw [shift={(14.1115288288,0.211611119414)}] plot[domain=-0.0237658514341:1.69793112698,variable=\t]({1.*0.488609151737*cos(\t r)+0.*0.488609151737*sin(\t r)},{0.*0.488609151737*cos(\t r)+1.*0.488609151737*sin(\t r)});
\draw [shift={(14.1494099235,0.0670600509879)},dash pattern=on 1pt off 1pt]  plot[domain=-1.46290097028:0.286895763184,variable=\t]({1.*0.469791918904*cos(\t r)+0.*0.469791918904*sin(\t r)},{0.*0.469791918904*cos(\t r)+1.*0.469791918904*sin(\t r)});
\draw (12.4036070822,-0.311720134782)-- (14.2,-0.4);
\draw [shift={(13.6674432714,-1.74883509295)}] plot[domain=0.272519716492:1.69306773,variable=\t]({1.*0.552963429268*cos(\t r)+0.*0.552963429268*sin(\t r)},{0.*0.552963429268*cos(\t r)+1.*0.552963429268*sin(\t r)});
\draw [shift={(13.3358690313,-1.3358690313)},dash pattern=on 1pt off 1pt]  plot[domain=5.00903104989:5.98654323767,variable=\t]({1.*0.903596978576*cos(\t r)+0.*0.903596978576*sin(\t r)},{0.*0.903596978576*cos(\t r)+1.*0.903596978576*sin(\t r)});
\draw [shift={(12.,-3.4)}] plot[domain=1.57079632679:4.71238898038,variable=\t]({1.*0.4*cos(\t r)+0.*0.4*sin(\t r)},{0.*0.4*cos(\t r)+1.*0.4*sin(\t r)});
\draw [shift={(13.237994438,-27.6050383422)}] plot[domain=1.55595161022:1.62106860503,variable=\t]({1.*24.6361632981*cos(\t r)+0.*24.6361632981*sin(\t r)},{0.*24.6361632981*cos(\t r)+1.*24.6361632981*sin(\t r)});
\draw [shift={(13.6053551647,-3.23000918079)}] plot[domain=-0.717888170758:1.57720943668,variable=\t]({1.*0.258425019169*cos(\t r)+0.*0.258425019169*sin(\t r)},{0.*0.258425019169*cos(\t r)+1.*0.258425019169*sin(\t r)});
\draw [shift={(13.5428820686,-3.64762735621)},dash pattern=on 1pt off 1pt]  plot[domain=-1.41009878568:0.766597676882,variable=\t]({1.*0.356971901106*cos(\t r)+0.*0.356971901106*sin(\t r)},{0.*0.356971901106*cos(\t r)+1.*0.356971901106*sin(\t r)});
\draw [shift={(13.8916564715,4.83325177217)}] plot[domain=4.49668501912:4.67938295255,variable=\t]({1.*8.83806541998*cos(\t r)+0.*8.83806541998*sin(\t r)},{0.*8.83806541998*cos(\t r)+1.*8.83806541998*sin(\t r)});
\draw [shift={(12.0446193221,-4.94820644069)},dash pattern=on 1pt off 1pt]  plot[domain=1.69824222973:3.94144196866,variable=\t]({1.*0.351053570322*cos(\t r)+0.*0.351053570322*sin(\t r)},{0.*0.351053570322*cos(\t r)+1.*0.351053570322*sin(\t r)});
\draw [shift={(11.86,-6.86)}] plot[domain=0.822880600089:1.50892848093,variable=\t]({1.*2.26433213112*cos(\t r)+0.*2.26433213112*sin(\t r)},{0.*2.26433213112*cos(\t r)+1.*2.26433213112*sin(\t r)});
\draw [shift={(12.9033183525,-5.30221223501)},dash pattern=on 1pt off 1pt]  plot[domain=-1.37896205464:0.202956847541,variable=\t]({1.*0.507089735587*cos(\t r)+0.*0.507089735587*sin(\t r)},{0.*0.507089735587*cos(\t r)+1.*0.507089735587*sin(\t r)});
\draw [shift={(11.8656208584,-5.4)}] plot[domain=1.88783338487:4.39535192231,variable=\t]({1.*0.210490135306*cos(\t r)+0.*0.210490135306*sin(\t r)},{0.*0.210490135306*cos(\t r)+1.*0.210490135306*sin(\t r)});
\draw [shift={(13.0225048259,-1.9649710449)}] plot[domain=4.3879597866:4.70652081934,variable=\t]({1.*3.83509498626*cos(\t r)+0.*3.83509498626*sin(\t r)},{0.*3.83509498626*cos(\t r)+1.*3.83509498626*sin(\t r)});
\draw [shift={(10.,-7.18290483755)}] plot[domain=0.740781809811:2.40081084378,variable=\t]({1.*0.271024315515*cos(\t r)+0.*0.271024315515*sin(\t r)},{0.*0.271024315515*cos(\t r)+1.*0.271024315515*sin(\t r)});
\draw [shift={(11.6,-7.11098938077)}] plot[domain=0.506633282441:2.63495937115,variable=\t]({1.*0.228732688184*cos(\t r)+0.*0.228732688184*sin(\t r)},{0.*0.228732688184*cos(\t r)+1.*0.228732688184*sin(\t r)});
\draw [shift={(10.,-6.88323397358)},dash pattern=on 1pt off 1pt]  plot[domain=3.67003769038:5.75474027039,variable=\t]({1.*0.231590813562*cos(\t r)+0.*0.231590813562*sin(\t r)},{0.*0.231590813562*cos(\t r)+1.*0.231590813562*sin(\t r)});
\draw [shift={(11.6,-6.79398846144)},dash pattern=on 1pt off 1pt]  plot[domain=3.94179605958:5.48298190119,variable=\t]({1.*0.287124979793*cos(\t r)+0.*0.287124979793*sin(\t r)},{0.*0.287124979793*cos(\t r)+1.*0.287124979793*sin(\t r)});
\draw [shift={(9.6,-7.)}] plot[domain=1.57079632679:4.71238898038,variable=\t]({1.*0.4*cos(\t r)+0.*0.4*sin(\t r)},{0.*0.4*cos(\t r)+1.*0.4*sin(\t r)});
\draw (9.6,-6.6)-- (10.4,-6.6);
\draw (9.6,-7.4)-- (10.4,-7.4);
\draw (11.2,-6.6)-- (12.,-6.6);
\draw (11.2,-7.4)-- (12.,-7.4);
\draw [shift={(10.3050301141,-7.)}] plot[domain=-1.33768779287:1.33768779287,variable=\t]({1.*0.411119543711*cos(\t r)+0.*0.411119543711*sin(\t r)},{0.*0.411119543711*cos(\t r)+1.*0.411119543711*sin(\t r)});
\draw [shift={(11.2805659651,-7.)}] plot[domain=1.76955200867:4.51363329851,variable=\t]({1.*0.408032933386*cos(\t r)+0.*0.408032933386*sin(\t r)},{0.*0.408032933386*cos(\t r)+1.*0.408032933386*sin(\t r)});
\draw [shift={(12.,-7.)}] plot[domain=-1.57079632679:1.57079632679,variable=\t]({1.*0.4*cos(\t r)+0.*0.4*sin(\t r)},{0.*0.4*cos(\t r)+1.*0.4*sin(\t r)});
\draw (9.8,-7) node[anchor=north west] {$c_2$};
\draw (11.3,-7) node[anchor=north west] {$c_4$};
\draw (9.85705947656,-6.02537301504) node[anchor=north west] {$d_1$};
\draw (11.4224316845,-6.00939982924) node[anchor=north west] {$d_3$};
\draw (12.0613591163,-5.65798974175) node[anchor=north west] {$\Phi_{0,0,0,0}(B_8)$};
\draw (12.0453859305,-3.85301974692) node[anchor=north west] {$\Phi_{0,0,0,0}(B_7)$};
\draw (12.0613591163,-2.06402293788) node[anchor=north west] {$\Phi_{0,0,0,0}(B_6)$};
\draw (12.0453859305,-0.290999314637) node[anchor=north west] {$\Phi_{0,0,0,0}(B_5)$};
\draw (4.05879303299,-0.290999314637) node[anchor=north west] {$\Phi_{0,0,0,0}(B_0)$};
\draw (4.05879303299,-2.09596930947) node[anchor=north west] {$\Phi_{0,0,0,0}(B_1)$};
\draw (4.05879303299,-3.86899293271) node[anchor=north west] {$\Phi_{0,0,0,0}(B_2)$};
\draw (4.05879303299,-5.68993611334) node[anchor=north west] {$\Phi_{0,0,0,0}(B_3)$};
\draw (4.07476621878,-7.49490610817) node[anchor=north west] {$\Phi_{0,0,0,0}(B_4)$};
\draw [shift={(11.6,-1.5)},dash pattern=on 1pt off 1pt]  plot[domain=3.78509376238:5.63968419839,variable=\t]({1.*0.5*cos(\t r)+0.*0.5*sin(\t r)},{0.*0.5*cos(\t r)+1.*0.5*sin(\t r)});
\draw [shift={(12.1340061401,-1.71073136274)}] plot[domain=3.72923979005:6.16470004941,variable=\t]({1.*0.161017189109*cos(\t r)+0.*0.161017189109*sin(\t r)},{0.*0.161017189109*cos(\t r)+1.*0.161017189109*sin(\t r)});
\draw [shift={(12.8911459846,-1.79370930327)}] plot[domain=1.7231265864:3.0349345424,variable=\t]({1.*0.60066490433*cos(\t r)+0.*0.60066490433*sin(\t r)},{0.*0.60066490433*cos(\t r)+1.*0.60066490433*sin(\t r)});
\draw (12.8,-1.2)-- (13.6,-1.2);
\draw [shift={(11.0299780217,-1.68171841323)}] plot[domain=3.3274015868:5.67536179751,variable=\t]({1.*0.207118340245*cos(\t r)+0.*0.207118340245*sin(\t r)},{0.*0.207118340245*cos(\t r)+1.*0.207118340245*sin(\t r)});
\draw [shift={(11.2258655686,-1.58546373823)}] plot[domain=2.02311965877:3.4664293605,variable=\t]({1.*0.421483151335*cos(\t r)+0.*0.421483151335*sin(\t r)},{0.*0.421483151335*cos(\t r)+1.*0.421483151335*sin(\t r)});
\draw [shift={(11.2196210725,-1.51732665849)}] plot[domain=0.131038054063:2.09061228957,variable=\t]({1.*0.358284870949*cos(\t r)+0.*0.358284870949*sin(\t r)},{0.*0.358284870949*cos(\t r)+1.*0.358284870949*sin(\t r)});
\draw [shift={(12.1778216391,-1.47964966671)}] plot[domain=3.12643973763:4.42224323956,variable=\t]({1.*0.603056579088*cos(\t r)+0.*0.603056579088*sin(\t r)},{0.*0.603056579088*cos(\t r)+1.*0.603056579088*sin(\t r)});
\draw [shift={(14.982678579,22.267814363)}] plot[domain=4.59059609275:4.65593892069,variable=\t]({1.*24.5068508739*cos(\t r)+0.*24.5068508739*sin(\t r)},{0.*24.5068508739*cos(\t r)+1.*24.5068508739*sin(\t r)});
\end{tikzpicture}

\caption{Simple closed curves $\Phi_{K_0\sharp K_0}(B_i)$}
\label{fig:phi00}
\end{center}
\end{figure}
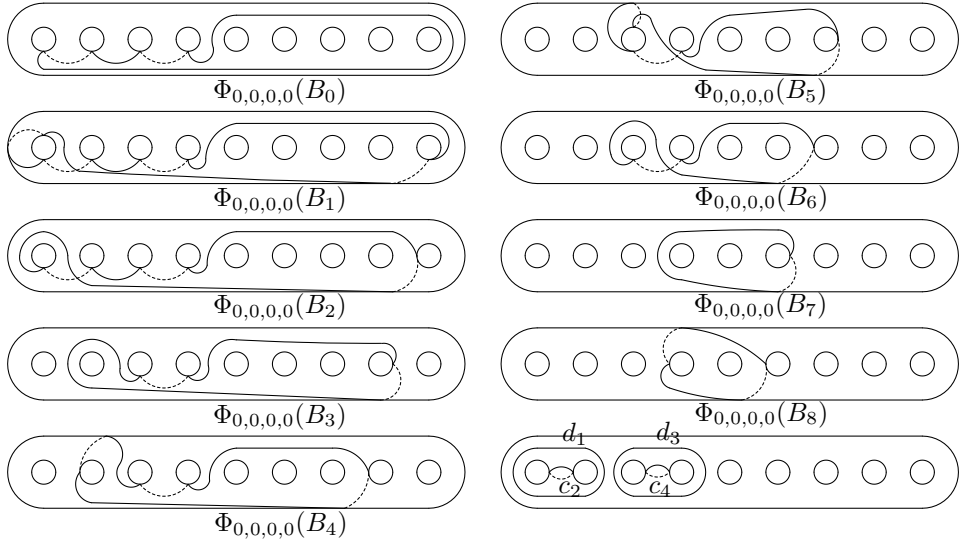

Note that we have the following relations in 
 $H_1(\Sigma_{4n+1}; \mathbb{Z}_2)$:
\[
B_0 = \sum_{i=1}^{4n} B_i + a_{2n+1}, \ \ B_{4n+1} = a_{2n+1}
\]
and from Figure~\ref{fig:phi00}, for $1 \leq j \leq n$,
\begin{align}
\Phi_{p_1, q_1, \cdots, p_n, q_n}(B_0) &= B_0 + 
\sum_{i=1}^n (a_{2i -1} + a_{2i} + b_{2i-1} + b_{2i}), \\
\Phi_{p_1, q_1, \cdots, p_n, q_n}(B_{4(j-1) + 1}) 
& = B_{4(j-1) +1} + b_{2j-1} + b_{2j} + a_{2j} + \varepsilon_{2j-1} c_{2j} \\
&+ \varepsilon_{2j} d_{2j -1} + 
\sum_{i=j+1}^n (a_{2i -1} + a_{2i} + b_{2i-1} + b_{2i}), \notag \\
\Phi_{p_1, q_1, \cdots, p_n, q_n}(B_{4(j-1) + 2}) 
& = B_{4(j-1) +2}  + b_{2j} + a_{2j-1} +  a_{2j} + \varepsilon_{2j-1} c_{2j}  \\
& + \sum_{i=j+1}^n (a_{2i -1} + a_{2i} + b_{2i-1} + b_{2i}),  \notag \\
\Phi_{p_1, q_1, \cdots, p_n, q_n}(B_{4(j-1) + 3}) 
& = B_{4(j-1) +3}  + b_{2j} + \varepsilon_{2j-1} c_{2j} \\
& + \sum_{i=j+1}^n (a_{2i -1} + a_{2i} + b_{2i-1} + b_{2i}), \notag \\
\Phi_{p_1, q_1, \cdots, p_n, q_n}(B_{4(j-1) + 4}) 
& = B_{4(j-1) +4}  + a_{2j} + \varepsilon_{2j-1} c_{2j} + \epsilon_{2j} d_{2j -1} \\
& + \sum_{i=j+1}^n (a_{2i -1} + a_{2i} + b_{2i-1} + b_{2i}). \notag
\end{align}

Now, for each $\sum_{i=1}^{2n} \varepsilon_{i} 2^{i-1} \in \{ 0, 1, \cdots, 2^{2n} -1\}$, 
we want to construct a basis $\mathcal{B}_{\sum_{i=1}^{2n} \varepsilon_{i} 2^{i-1}}$ 
of $H_1(\Sigma_{4n +1} ; \mathbb{Z}_2)$ which satisfies
\begin{itemize}
\item $\{ B_1, B_2, \cdots, B_{4n}, a_{2n+1}, b_{2n+1}\} 
\subset \mathcal{B}_{\sum_{i=1}^{2n} \varepsilon_{i} 2^{i-1}}$ and
\item $G_F(\Phi_{p_1, q_1, \cdots, p_n, q_n}(\eta_{1, 2n}^2) \cdot \eta_{1, 2n}^2 ) 
\leq G_{\Gamma( \mathcal{B}_{\sum_{i=1}^{2n} \varepsilon_{i} 2^{i-1}})}$.
\end{itemize}

Note that Equations (3.1)--(3.5)  and the second condition for 
$\mathcal{B}_{\sum_{i=1}^{2n} \varepsilon_{i} 2^{i-1}}$ imply (a)--(e) 
for $1 \leq j \leq n$, and we add one more condition (f) as follows:
\begin{enumerate}[label=(\alph*)]
\item even number of $\cup_{i=1}^n\{a_{2i-1}, a_{2i}, b_{2i-1}, b_{2i}\}$ 
has $\chi_{\Gamma(\mathcal{B}_{\sum_{i=1}^{2n} \varepsilon_{i} 2^{i-1}})} = 0$.
\item even number of $$\{b_{2j-1}, b_{2j}, a_{2j}, \varepsilon_{2j-1}c_{2j}, 
\varepsilon_{2j}d_{2j-1}\} \cup \cup_{i=j+1}^n\{a_{2i-1}, a_{2i}, b_{2i-1}, b_{2i}\}$$ 
has $\chi_{\Gamma(\mathcal{B}_{\sum_{i=1}^{2n} \varepsilon_{i} 2^{i-1}})} = 0$.
\item even number of $$\{ b_{2j}, a_{2j-1},a_{2j}, \varepsilon_{2j-1}c_{2j}\} 
\cup \cup_{i=j+1}^n\{a_{2i-1}, a_{2i}, b_{2i-1}, b_{2i}\}$$
has $\chi_{\Gamma(\mathcal{B}_{\sum_{i=1}^{2n} \varepsilon_{i} 2^{i-1}})} = 0$.
\item even number of $$\{b_{2j}, \varepsilon_{2j-1}c_{2j}\} 
\cup \cup_{i=j+1}^n\{a_{2i-1}, a_{2i}, b_{2i-1}, b_{2i}\}$$ 
has $\chi_{\Gamma(\mathcal{B}_{\sum_{i=1}^{2n} \varepsilon_{i} 2^{i-1}})} = 0$.
\item even number of 
$$\{a_{2j}, \varepsilon_{2j-1}c_{2j}, \varepsilon_{2j}d_{2j-1}\} 
\cup \cup_{i=j+1}^n\{a_{2i-1}, a_{2i}, b_{2i-1}, b_{2i}\}$$
has $\chi_{\Gamma(\mathcal{B}_{\sum_{i=1}^{2n} \varepsilon_{i} 2^{i-1}})} = 0$.
\item $\chi_{\Gamma(\mathcal{B}_{\sum_{i=1}^{2n} \varepsilon_{i} 2^{i-1}})} (c_{2j}) = 0 
= \chi_{\Gamma(\mathcal{B}_{\sum_{i=1}^{2n} \varepsilon_{i} 2^{i-1}})} (d_{2j-1})$
for $1 \leq j \leq n$.
\end{enumerate}
Then these systems of conditions give a unique solution and we can prove it 
by using induction on $n$, the number of connected summed Stallings knots. 
For example, $n=1$ case was already proved in Proposition~\ref{proposition:step1} above. 
Now we will show why the inductive step works.
Let us consider the following table:
\begin{table}[htdp]
\begin{center}
\begin{tabular}{c||cccccc|cccccc}
\hline
$\Phi_{p_1,q_1, p_2, q_2}(B_0)$ &
$b_1$ & $b_2$ & $a_1$ & $a_2$ & &  &$b_3$ &$b_4$ & $a_3$ & $a_4$ & & \\ \hline
$\Phi_{p_1,q_1, p_2, q_2}(B_1)$ &
$b_1$ & $b_2$ &   & $a_2$ & $\varepsilon_{p_1} c_2$ & $\varepsilon_{q_1} d_1$  
&$b_3$ &$b_4$ & $a_3$ & $a_4$ & &  \\
$\Phi_{p_1,q_1, p_2, q_2}(B_2)$ &
 & $b_2$ & $a_1$  & $a_2$ & $\varepsilon_{p_1} c_2$ & 
&$b_3$ &$b_4$ & $a_3$ & $a_4$ & &  \\
$\Phi_{p_1,q_1, p_2, q_2}(B_3)$ &
 &$b_2$ &   &  & $\varepsilon_{p_1} c_2$ &
&$b_3$ &$b_4$ & $a_3$ & $a_4$ & &  \\
$\Phi_{p_1,q_1, p_2, q_2}(B_4)$ &
&  &   & $a_2$ & $\varepsilon_{p_1} c_2$ & $\varepsilon_{q_1} d_1$  
&$b_3$ &$b_4$ & $a_3$ & $a_4$ & &  \\ \hline
$\Phi_{p_1,q_1, p_2, q_2}(B_5)$ &
& & & & & & 
$b_3$ & $b_4$ &   & $a_4$ & $\varepsilon_{p_2} c_4$ & $\varepsilon_{q_2} d_3$ \\ 
$\Phi_{p_1,q_1, p_2, q_2}(B_6)$ &
& & & & & & 
 & $b_4$ & $a_3$   & $a_4$ & $\varepsilon_{p_2} c_4$ &  \\ 
 $\Phi_{p_1,q_1, p_2, q_2}(B_7)$ &
 & & & & & & 
 & $b_4$ &   & & $\varepsilon_{p_2} c_4$ & \\ 
 $\Phi_{p_1,q_1, p_2, q_2}(B_8)$ &
 & & & & & & 
 &  &   & $a_4$ & $\varepsilon_{p_2} c_4$ & $\varepsilon_{q_2} d_3$ \\ \hline
\end{tabular}
\end{center}
\end{table}%

Then the last $4$ rows coming from $\Phi_{p_1,q_1, p_2, q_2}(B_j)$ 
($ 5 \leq j \leq 8$) has a unique solution for each given 
$(\varepsilon_{p_2}, \varepsilon_{q_2}) \in \{0,1\} \times \{0,1\}$. 
In each case even number of $\{b_3, b_4, a_3, a_4\}$ has modulo $2$ Euler number $0$. 
Therefore it does not give any effect to the solution of the first 4 rows 
coming from $\Phi_{p_1,q_1, p_2, q_2}(B_j)$ ($1 \leq j \leq 4$) 
and the solution has exactly same pattern as the solution for 
\[\{\Phi_{p_1, q_1}(B_1), \Phi_{p_1, q_1}(B_2), \Phi_{p_1, q_1}(B_3), 
\Phi_{p_1, q_1}(B_4)\}.\]
Note that the condition for $\Phi_{p_1,q_1, p_2, q_2}(B_0)$ is automatically satisfied. 

Therefore $\mathcal{B}_{\sum_{i=1}^{2n} \varepsilon_{i} 2^{i-1}}$ satisfies that, for $1 \leq i \leq n$,
\begin{itemize}
\item if $(\varepsilon_{2i-1}, \varepsilon_{2i}) \equiv (0,0) \pmod{2}$, then
\[
\{a_{2i-1}, a_{2i}, b_{2i-1}, b_{2i} \} 
\subset G_{\Gamma( \mathcal{B}_{\sum_{i=1}^{2n} \varepsilon_{i} 2^{i-1}})}
\]
\item if $(\varepsilon_{2i-1}, \varepsilon_{2i}) \equiv (1,0) \pmod{2}$, then
\[
a_{2i-1}, a_{2i}, b_{2i-1}, b_{2i} 
\not\in G_{\Gamma( \mathcal{B}_{\sum_{i=1}^{2n} \varepsilon_{i} 2^{i-1}})}
\]
\item if $(\varepsilon_{2i-1}, \varepsilon_{2i}) \equiv (0,1) \pmod{2}$, then
\[
\{b_{2i-1}, b_{2i} \} 
\subset G_{\Gamma( \mathcal{B}_{\sum_{i=1}^{2n} \varepsilon_{i} 2^{i-1}})}
        \ \mathrm{and} \ a_{2i-1}, a_{2i} \not\in G_{\Gamma( \mathcal{B}_{\sum_{i=1}^{2n} \varepsilon_{i} 2^{i-1}})}
\]
\item if $(\varepsilon_{2i-1}, \varepsilon_{2i}) \equiv (1,1) \pmod{2}$, then
\[
\{a_{2i-1}, a_{2i} \} 
\subset G_{\Gamma( \mathcal{B}_{\sum_{i=1}^{2n} \varepsilon_{i} 2^{i-1}})}
        \ \mathrm{and} \ b_{2i-1}, b_{2i} \not\in G_{\Gamma( \mathcal{B}_{\sum_{i=1}^{2n} \varepsilon_{i} 2^{i-1}})}.
\]
\end{itemize}

From this observation, we can define  
$\mathcal{B}_{\sum_{i=1}^{2n} \varepsilon_{i} 2^{i-1}}$ as follows:
\begin{itemize}
\item start from $\{ B_1, B_2, \cdots, B_{4n}, a_{2n+1}, b_{2n+1}\}$
\item for $1 \leq i \leq n$, 
\begin{itemize}
\item if  $(\varepsilon_{2i-1}, \varepsilon_{2i}) \equiv (0,0) \pmod{2}$, then add
\(
\{a_{2i-1}, a_{2i}, b_{2i-1}, b_{2i} \}
\)
\item if  $(\varepsilon_{2i-1}, \varepsilon_{2i}) \equiv (1,0) \pmod{2}$, then add
\(
\{e_{2i-1}, e_{2i}, f_{2i-1}, f_{2i} \}
\)
\item if  $(\varepsilon_{2i-1}, \varepsilon_{2i}) \equiv (0,1) \pmod{2}$, then add
\(
\{e_{2i-1}, e_{2i}, b_{2i-1}, b_{2i} \}
\)
\item if  $(\varepsilon_{2i-1}, \varepsilon_{2i}) \equiv (1,1) \pmod{2}$, then add
\(
\{a_{2i-1}, a_{2i}, f_{2i-1}, f_{2i} \}
\)
\end{itemize}
\end{itemize}
to the set, where
\begin{itemize}
\item $e_i$ is a simple closed curve representing $a_i + a_{2n+1}$ 
in $H_1(\Sigma_{4n + 1}; \mathbb{Z}_2)$
\item $f_i$ is a simple closed curve representing $b_i + b_{2n+1}$ 
in $H_1(\Sigma_{4n + 1}; \mathbb{Z}_2)$.
\end{itemize}
Then each resulting set $\mathcal{B}_i (0 \leq i \leq 2^{2n} -1)$ is a basis of $H_1(\Sigma_{4n+1}; \mathbb{Z}_2)$ by Lemma~\ref{lemma:basis} and it satisfies 
\[
G_F(\Phi_{p_1, q_1, \cdots, p_n, q_n}(\eta_{1, 2n}^2) \cdot \eta_{1, 2n}^2 ) 
\leq G_{\Gamma( \mathcal{B}_{\sum_{i=1}^{2n} \varepsilon_{i} 2^{i-1}})}.
\]

Now we will show that, if 
\( (p_1, q_1, \cdots, p_n, q_n) \not\equiv (r_1, s_1, \cdots, r_n, s_n) \pmod{2} \), 
\[
G_F(\Phi_{p_1, q_1, \cdots, p_n, q_n}(\eta_{1, 2n}^2) \cdot \eta_{1, 2n}^2 ) 
\ne G_F(\Phi_{r_1, s_1, \cdots, r_n, s_n}(\eta_{1, 2n}^2) \cdot \eta_{1, 2n}^2 ).
\]

Let us observe that, for $1 \leq j \leq n$,
\begin{itemize}
\item $i_2(\Phi_{p_1, q_1, \cdots, p_n, q_n}(B_{4(j-1) +1}), c_{2j}) = 
      1 = i_2(\Phi_{p_1, q_1, \cdots, p_n, q_n}(B_{4(j-1) +1}), d_{2j-1})$ 
\item $i_2(\Phi_{p_1, q_1, \cdots, p_n, q_n}(B_{4(j-1) +2}), c_{2j}) = 1$ and
      \\
      $i_2(\Phi_{p_1, q_1, \cdots, p_n, q_n}(B_{4(j-1) +2}), d_{2j-1}) = 0$ 
\item $\chi_{\Gamma(\mathcal{B}_i)}(c_{2j}) = 0 =   
      \chi_{\Gamma(\mathcal{B}_i)}(d_{2j-1})$ for $0 \leq i \leq 2^{2n} -1$.
\end{itemize}
Then this observation together with Corollary~\ref{corollary:Humphries} implies that
\begin{itemize}
\item if $(p_j, q_j) - (r_j, s_j) \equiv (1,0) \text{ or } (0,1) \pmod{2}$ 
for some $j \in \{1, 2, \cdots, n\}$, then
\[
\Phi_{p_1, q_1, \cdots, p_n, q_n}(B_{4(j-1) +1}) \in 
G_{\Gamma(\mathcal{B}_{\sum_{i=1}^n(\varepsilon_{p_i} + 2 \varepsilon_{q_i}) 2^{2(i-1)}})}
\]
and 
\[
\Phi_{r_1, s_1, \cdots, r_n, s_n}(B_{4(j-1) +1}) \not\in 
G_{\Gamma(\mathcal{B}_{\sum_{i=1}^n(\varepsilon_{p_i} + 2 \varepsilon_{q_i}) 2^{2(i-1)}})}
\]
\item if $(p_j, q_j) - (r_j, s_j) \equiv (1,1) \pmod{2} $ for some $j \in \{1, 2, \cdots, n\}$, 
then
\[
\Phi_{p_1, q_1, \cdots, p_n, q_n}(B_{4(j-1) +2}) \in 
G_{\Gamma(\mathcal{B}_{\sum_{i=1}^n(\varepsilon_{p_i} + 2 \varepsilon_{q_i}) 2^{2(i-1)}})}
\]
and 
\[
\Phi_{r_1, s_1, \cdots, r_n, s_n}(B_{4(j-1) +2}) \not\in 
G_{\Gamma(\mathcal{B}_{\sum_{i=1}^n(\varepsilon_{p_i} + 2 \varepsilon_{q_i}) 2^{2(i-1)}})}.
\]
\end{itemize}
It gives the conclusion.
\end{proof}

\begin{remark}
1. We can obtain similar results on $E(2)_K$ using a family of Kanenobu knot $K$ with a parity of type $(1,0)$ or $(0,1)$. \\
2. In Theorem~\ref{theorem:main} above  we constructed a family of simply connected minimal symplectic $4$-manifolds $E(2)_K$ admitting arbitrarily many nonisomorphic Lefschetz fibration structures over $S^2$ of the same genus fiber. 
The next goal in this direction is to show whether there exist a family of knot surgery $4$-manifolds $E(2)_K$ admitting infinitely many nonisomorphic Lefschetz fibrations. For this, let us consider  a family of knots $K_{p,q}$ used in the proof of Proposition~\ref{proposition:step1} above. Note that the knot  $K_{p,q}$ is obtained from $K_0$ by performing Stallings twists $p$ times along $c_2$  and $q$ times along $d_1$ in Figure~\ref{fig:SCC}. Then, for a fixed integer $n \in  \mathbb{Z}$, we would claim that $\{ E(2)_{K_{p,q}} \, |\, p+q = n \}$ admits mutually nonisomorphic Lefschetz fibration structures. We investigate this problem in the forthcoming paper. 
\end{remark}


\providecommand{\bysame}{\leavevmode\hbox to3em{\hrulefill}\thinspace}
\providecommand{\MR}{\relax\ifhmode\unskip\space\fi MR }
\providecommand{\MRhref}[2]{%
  \href{http://www.ams.org/mathscinet-getitem?mr=#1}{#2}
}
\providecommand{\href}[2]{#2}

\end{document}